\documentclass[a4paper,reqno,10pt]{amsart}

\usepackage{hyperref}

\usepackage[a4paper,textwidth=125mm,textheight=200mm,marginparwidth=20mm,marginparsep=10pt]{geometry}

\usepackage{comment} 

\usepackage{hyphenat}

\usepackage{manfnt} 
\reversemarginpar 

\usepackage[svgnames]{xcolor}

\newif\ifdiffmode
\diffmodefalse
\newcommand{\new}[1]{\ifdiffmode{\color{blue}#1}\else{#1}\fi}

\newcommand{\refmod}[1]{\ifdiffmode{\color{red}#1}\else{#1}\fi}

\usepackage{url}

\usepackage{footmisc}

\usepackage{amsmath}
\usepackage{amsthm}
\usepackage{amssymb}

\allowdisplaybreaks 

\usepackage{amscd}

\usepackage[T1]{fontenc}
\usepackage{lmodern}

\usepackage{bbm}

\usepackage{mathrsfs} 

\usepackage[OMLmathsfit]{isomath}

\usepackage{mathabx}

\usepackage{leftindex}

\usepackage{enumitem}
\setlist{itemsep=2pt plus 1pt}

\newlist{hypenum}{enumerate}{1}
\setlist[hypenum]{label=\refmod{(H\arabic*)}, ref=(H\arabic*)}
\newlist{conclenum}{enumerate}{1}
\setlist[conclenum]{label=\refmod{(C\arabic*)}, ref=(C\arabic*)}
\newlist{numenum}{enumerate}{1}
\setlist[numenum]{label=\refmod{(\arabic*)}, ref=(\arabic*)}
\newlist{abcenum}{enumerate}{1}
\setlist[abcenum]{label=\refmod{(\alph*)}, ref=(\alph*)}
\newlist{iiienum}{enumerate}{1}
\setlist[iiienum]{label=(\roman*), ref=(\roman*)}

\usepackage{stackengine}

\swapnumbers 

\newtheoremstyle{exercise}
  {3pt} 
  {3pt} 
  {\small\rmfamily} 
  {
} 
  {\rmfamily\scshape} 
  {.} 
  {.5em} 
  {} 

\newtheoremstyle{newplain}
  {5pt}
  {5pt}
  {\itshape}
  {}
  {\rmfamily\scshape}
  { ---}
  {.5em}
  {}

\newtheoremstyle{newremark}
  {5pt}
  {5pt}
  {\rmfamily}
  {}
  {\rmfamily\scshape}
  { ---}
  {.5em}
  {}

\theoremstyle{newplain}
\newtheorem*{Theorem*}{Theorem} 

\theoremstyle{newplain}
\newtheorem{Theorem}{Theorem}

\newtheorem{Corollary}[Theorem]{Corollary}
\newtheorem{Proposition}[Theorem]{Proposition}
\newtheorem{Conjecture}[Theorem]{Conjecture}
\newtheorem{Definition}[Theorem]{Definition}
\newtheorem{Scholia}[Theorem]{Scholia}

\theoremstyle{newremark}
\newtheorem{Empty}[Theorem]{}
\newtheorem{Remark}[Theorem]{Remark}

\newtheorem{Claim}[Theorem]{Claim}

\theoremstyle{exercise}

\numberwithin{Theorem}{section}
\numberwithin{equation}{section}

\newcommand{\N}{\mathbb{N}}

\newcommand{\R}{\mathbb{R}}

\newcommand{\Rm}{\R^m}
\newcommand{\Rn}{\R^n}

\newcommand{\C}{\mathbb{C}}

\newcommand{\bbR}{\mathbb{R}}

\newcommand{\calB}{\mathscr{B}}
\newcommand{\calC}{\mathscr{C}}
\newcommand{\calD}{\mathscr{D}}
\newcommand{\calE}{\mathscr{E}}
\newcommand{\calF}{\mathscr{F}}

\newcommand{\calM}{\mathscr{M}}

\newcommand{\calP}{\mathscr{P}}

\newcommand{\calS}{\mathscr{S}}
\newcommand{\calT}{\mathscr{T}}
\newcommand{\calU}{\mathscr{U}}
\newcommand{\calV}{\mathscr{V}}

\newcommand{\frS}{\frak S}

\newcommand{\ssfH}{\mathsfit{H}}

\newcommand{\ssfL}{\mathsfit{L}}

\newcommand{\balpha}{\boldsymbol{\alpha}}

\newcommand{\bD}{\mathbf{D}}
\newcommand{\bE}{\mathbf{E}}
\newcommand{\bF}{\mathbf{F}}

\newcommand{\bI}{\mathbf{I}}

\newcommand{\bc}{\mathbf{c}}
\newcommand{\bcG}{\mathbf{c}_{\refmod{\mathrm{(H7)}}}}
\newcommand{\bcH}{\mathbf{c}_{\refmod{\mathrm{(H8)}}}}

\newcommand{\bq}{\mathbf{q}}

\newcommand{\bv}{\mathbf{v}}

\newcommand{\bev}{\mathbf{ev}}   

\DeclareMathOperator{\rmabsconv}{\mathrm{aco}}   
\DeclareMathOperator{\rmbdry}{\mathrm{bdry}}          
\newcommand{\rmbd}{\mathrm{bd}}                       
\DeclareMathOperator{\rmclos}{\mathrm{clos}}          
\DeclareMathOperator{\rmconv}{\mathrm{co}}            
\newcommand{\rmcomp}{\mathrm{comp}}

\DeclareMathOperator{\rmdist}{\mathrm{dist}}          
\DeclareMathOperator{\rmdiv}{\mathrm{div}}            
\DeclareMathOperator{\rmev}{\mathrm{ev}}              
\newcommand{\rmid}{\mathrm{id}}                       
\DeclareMathOperator{\rmim}{\mathrm{im}}              
\DeclareMathOperator{\rmint}{\mathrm{int}}            

\DeclareMathOperator{\rmloc}{\mathrm{loc}}

\newcommand{\rmseq}{\mathrm{seq}}						
\newcommand{\rms}{\mathrm{s}}
\DeclareMathOperator{\rmspt}{\mathrm{spt}}            

\newcommand{\ip}{\,\,\begin{picture}(-1,1)(-1,-2.5)\circle*{2}\end{picture}\;\,\,}

\newcommand{\ind}{\mathbbm{1}}

\newcommand{\hel} {
\hskip2.5pt{\vrule height7pt width.5pt depth0pt}
\hskip-.2pt\vbox{\hrule height.5pt width7pt depth0pt}
\, }

\newcommand{\shel} {
\hskip2.5pt{\vrule height5pt width.5pt depth0pt}
\hskip-.2pt\vbox{\hrule height.5pt width5pt depth0pt}
\, }

\newcommand{\rmrestr}{\mathrm{restr}}
\newcommand{\rmext}{\mathrm{ext}}

\newcommand{\lno}{\pmb{\thickvert}}
\newcommand{\rno}{\mkern+1mu\pmb{\thickvert}}

\def\XXint#1#2#3{{%
\setbox0=\hbox{$#1{#2#3}{\int}$}
\vcenter{\hbox{$#2#3$}}\kern-.5\wd0}}

\newcommand{\TV}{TV}
\newcommand{\SCH}{SC\!H}

\newcommand{\cqfd} {
\renewcommand{\qedsymbol}{$\blacksquare$}
\qed
\renewcommand{\qedsymbol}{$\square$} }

\newcommand{\veps}{\varepsilon}

\newcommand{\vphi}{\varphi}

\newcommand{\la}{\langle}
\newcommand{\ra}{\rangle}

\newcommand{\ie}{{\it i.e.}\ }
\newcommand{\eg}{{\it e.g.}\ }

\renewcommand{\em}{\bf}

\renewcommand{\leq}{\leqslant}
\renewcommand{\geq}{\geqslant}

\renewcommand{\subset}{\subseteq}
\renewcommand{\supset}{\supseteq}

\newlength{\drop}

\newcommand*{\plogo}{\fbox{$\mathcal{HSH}$}}
\usepackage{trajan}

\newcommand*{\titleAT}{\begingroup
\drop=0.1\textheight
\vspace*{\drop}
\rule{0.93\textwidth}{1pt}\par 
\vspace{2pt}\vspace{-\baselineskip}
\rule{0.93\textwidth}{0.4pt}\par 
\vspace{0.5\drop}
\centering

\textcolor{Blue}{
{\trjnfamily
	\Huge MORE SO D'ANALYSE}\\[5\baselineskip]
{\trjnfamily
	\Large VOLUME I}\\[1.8\baselineskip]
{\trjnfamily
	\Huge MEASURE THEORY}\\[5\baselineskip]
}\par 
\vspace{0.25\drop}
\rule{0.3\textwidth}{0.4pt}\par 
\vspace{\drop} 

{\Large \scshape Thierry De Pauw}\par 
\vfill 
{\large \textcolor{Blue}{\plogo}}\\[0.5\baselineskip]

{\large \scshape Home Sweet Home}\par 
\vspace*{\drop}
\endgroup}

\begin{document}

\title{Localized locally convex topologies}

\author[Th. De Pauw]{Thierry De Pauw}

\address{Institute for Theoretical Sciences / School of Science, Westlake University\\
No. 600, Dunyu Road, Xihu District, Hangzhou, Zhejiang, 310030, China}

\email{thierry.depauw@westlake.edu.cn}




\begin{abstract}
Motivated by ill-posed PDEs such as $\rmdiv v = F$ we study locally convex topologies $\calT_\calC$ on real vector spaces $X$ that are a ``localized'' version of a locally convex topology $\calT$ to members of a family $\calC$ of convex subsets of $X$.
The distributions $F$ arising as $\rmdiv v$ are expected to be the members of the dual of well-chosen $X$ with respect to an appropriate localized topology $\calT_\calC$.
In this work, the emphasis is on studying the functional analytic properties of $\calT_\calC$, according to those of $\calT$ and $\calC$.
For instance, we show that in all foreseen applications, $\calT_\calC$ is sequential but \refmod{neither Fr\'echet-Urysohn, nor barrelled, nor bornological}.
These awkward phenomena are illustrated explicitly on a specific example corresponding to the distributional divergence of continuous vector fields in $\Rm$.
\new{We also show that, under suitable hypotheses, the $\calT$-compactness of the members of $\calC$ implies that $X[\calT_\calC]$ is semireflexive and identifies its localized topology with the bounded weak* topology of its bidual.}
This leads to an abstract existence theorem, thereby establishing a general scheme for characterizing those $F$ such that $\rmdiv v = F$ for various classes of regularity of $v$, various classes of domains, and various boundary conditions. 
\end{abstract}

\maketitle



\tableofcontents


\section*{Introduction}
\label{sec.INTRO}

Let $X$ be a real vector space.
Let also $\|\cdot\|$ and $\lno \cdot \rno$ be, respectively, a norm and a seminorm on $X$, the latter being lower-semicontinuous with respect to the former. \refmod{We call the latter the localizing seminorm.}
The purpose of this paper is to study a vector topology $\calT_\calC$ on $X$ such that a linear form $f : X \to \R$ is $\calT_\calC$-continuous if and only if
\begin{equation*}
(\forall \veps > 0)(\exists \theta > 0)(\forall x \in X): |f(x)| \leq \theta \cdot  \|x\| + \veps \cdot  \lno x \rno.
\end{equation*}
\par 
More generally, instead of a normed space $X[\|\cdot\|]$ we handle the case of a locally convex topological vector space $X[\calT]$ whose topology is generated by a filtering family of seminorms $\la \|\cdot\|_i\ra_{i \in I}$ and $X$ is the union of a non-decreasing sequence $\la X_k \ra_k$ of $\calT$-closed vector subspaces.
In this case, we study a vector topology $\calT_\calC$ on $X$ such that a linear form $f : X \to \R$ is $\calT_\calC$-continuous if and only if
\begin{equation}
\label{eq.INTRO.2}
(\forall \veps > 0)(\exists i \in I)(\exists \theta > 0)(\forall x \in X_{\lceil \veps^{-1} \rceil}): |f(x)| \leq \theta\cdot \|x\|_i + \veps\cdot \lno x \rno.
\end{equation}
\par 
Our initial motivation arose from the following application which will play the role of a case study in this introduction -- details are given in Section \ref{sec.OE}.
Assume that $v \in C(\Rm;\Rm)$ is a continuous vector field in $\Rm$ and consider its distributional divergence $\rmdiv v$ defined as follows
\begin{equation*}
\calD(\Rm) \to \R : \vphi \mapsto - \int_{\Rm} v \ip (\nabla \vphi) \,d\ssfL^m,
\end{equation*}
where $\calD(\Rm)$ is the space of test functions and $\ssfL^m$ the Lebesgue measure.
Given $\veps > 0$ we may choose a smooth vector field $w \in C^\infty(\Rm;\Rm)$ such that $|v(x)-w(x)| < \veps$ for all $x \in B(0,\veps^{-1})$.
Letting $\theta = \max \{ |(\rmdiv w)(x)| : x \in B(0,\veps^{-1})\}$ we have
\begin{equation}
\label{eq.INTRO.3}
|\la \vphi , \rmdiv v \ra| \leq \left| \int_{\Rm} ( \rmdiv w ) \cdot \vphi \,d\ssfL^m \right| + \left| \int_{\Rm} (v-w) \ip (\nabla \vphi) \,d\ssfL^m \right| < \theta \cdot \|\vphi\|_{L_1} + \veps \cdot  \|\nabla \vphi \|_{L_1}
\end{equation}
for all test functions $\vphi$ supported in $B(0,\veps^{-1})$.
This is a special case of the general setup above with $X = \calD(\Rm)$, $\|\cdot\| = \|\cdot\|_{L_1}$, $\lno \cdot \rno = \| \nabla \cdot \|_{L_1}$, and $X_k = \calD(\Rm) \cap \{ \vphi : \rmspt \vphi \subset B(0,k)\}$.
\par 
The ultimate aim of the techniques developed here is to establish a converse to the previous paragraph, \ie that the linear forms $F : \calD(\Rm) \to \R$ such that
\begin{multline}
\label{eq.INTRO.1}
(\forall \veps > 0)(\exists \theta > 0)(\forall \vphi \in \calD(\Rm)) : \rmspt \vphi \subset B(0,\lceil \veps^{-1} \rceil)
\\
\Rightarrow |\la \vphi, F \ra| < \theta \cdot \|\vphi \|_{L_1} + \veps \cdot \|\nabla \vphi \|_{L_1}
\end{multline}
are precisely those of the type $F = \rmdiv v$ for some $v \in C(\Rm;\Rm)$.
This particular result was proved in \cite{DEP.PFE.06b}, a different version in \cite{DEP.TOR.09}, a similar result for differential forms in \cite{DEP.MOO.PFE.08}, whereas applications to Lebesgue-summable (rather than continuous) vector fields will appear in \cite{DEP.26a}.
\par 
Our focus in the present paper is twofold:
\begin{itemize}
\item To develop a general theory, culminating in the new existence Theorem \ref{AET.FR}, that encompasses all cases quoted in the previous paragraph and more forthcoming applications.
\item To provide a careful study of these topologies $\calT_\calC$ which, we feel, is warranted by the unexpected, sometimes awkward, often tricky issues that they offer to the user and that were not always fully appreciated previously\footnote{including by the present author.}. These surprises are described in the remaining part of the introduction. 
\end{itemize}
We present in Section \ref{sec.OE} a detailed analysis of the special situation described above with continuous vector fields; our proof here differs somewhat from, but is inspired by the original given in \cite{DEP.PFE.06b}.
\\
\par 
The author hopes that the readership will include at least part of the community of analysts studying PDEs akin \new{to} $\rmdiv v = F$.
As their toolbox does not necessarily comprise topics such as sequential {\it vs.} Fr\'echet-Urysohn spaces, the Mackey-Arens theorem, and the Krein-\v{S}mulian theorem, these are treated carefully in appendices or the body of the text.
We now proceed to describing in more detail the content of this work.
\\
\par 
Our first task in Section \ref{sec.EUL} is to prove the existence and uniqueness of a locally convex topology $\calT_\calC$ on $X$ associated with a locally convex topology $\calT$ and a localizing family $\calC$, \ie a collection of convex sets containing the origin such that for all $C \in \calC$, $x \in X$, and $t \in \R$ there exists $D \in \calC$ containing $x + t \cdot C$ (Definition \ref{EUL.2}).
The main localizing family to keep in mind throughout is $\calC = \{ C_k : k \in \N \}$, where $C_k = X_k \cap \{ x : \lno x \rno \leq k \}$ under the assumptions of the second paragraph of this introduction.
The topology $\calT_\calC$ is characterized by the following universal property (see \ref{EUL.3} and \ref{GO.1}\refmod{\ref{GO.1.2}}):
\begin{enumerate}
\item[\refmod{(i)}] For every $C \in \calC$ we have $\calT_\calC \hel C = \calT \hel C$;
\item[\refmod{(ii)}] For every locally convex space $Y$ and every linear map $f : X \to Y$ if the restriction $f|_C : C \to Y$ is $\calT \hel C$-continuous for all $C \in \calC$ then $f$ is $\calT_\calC$-continuous.
\end{enumerate}
Here, $\calT \hel C$ is the topology on $C$ induced by $\calT$ and similarly for $\calT_\calC \hel C$.
Condition \refmod{(i)} is the reason why we call $\calT_\calC$ a {\em localized} topology, the ``localization'' being understood with respect to members of $\calC$.
The existence and uniqueness Theorem \ref{EUL.4} has been stated in \cite[theorem 3.3]{DEP.MOO.PFE.08} and in the course of the full proof given here we single out two useful observations \ref{EUL.9} and \eqref{eq.EUL.1}.
One should carefully notice that $\calT_\calC$ is in general {\it not} an inductive limit topology in the category of locally convex topological vector spaces, since the members of $\calC$ are not necessarily vector subspaces of $X$ but merely convex subsets.
\\
\par 
The short Section \ref{sec.GO} gathers a few facts that are true at this level of generality, for instance, that $\calT$ is coarser than $\calT_\calC$.
It also touches upon the following difficulty.
If $Y \subset X$ is a vector subspace then one checks that $\calC \hel Y = \{ C \cap Y : C \in \calC \}$ is a localizing family in $Y$, therefore, one may consider the localized topology $(\calT \hel Y)_{\calC \shel Y}$ on $Y$.
We have $\calT_\calC \hel Y \subset (\calT \hel Y)_{\calC \shel Y}$ but equality does not seem to hold in general -- Theorem \ref{CCC.1}\refmod{\ref{CCC.1.c3}} \refmod{and} Theorem \ref{CWC.6}\refmod{\ref{CWC.6.c2}} give sufficient conditions for equality\footnote{With respect to \ref{CCC.1}\refmod{\ref{CCC.1.c3}} we observe that \cite[Proposition 3.9]{DEP.MOO.PFE.08} states a wrong hypothesis; at least it doesn't correspond to the proof given there.}.
This is \refmod{an inconvenience}, since it prevents straightforward applications of Hahn-Banach's theorem in conjunction with testing $\calT_\calC$-continuity of a partially defined linear form by means of \eqref{eq.INTRO.2}, \ie in accordance with, for instance, Theorem \ref{CCC.5}, see Remark \ref{FD.4}.
This justifies the special notion of {\it uniformly sequentially $\calT_\calC$-dense subspace} introduced in \ref{FD.5} and the corresponding extension Theorem \ref{FD.6}.
\\
\par 
The results of Section \ref{sec.CCC} hold in the slightly restricted setting when the localizing family $\calC$ is assumed to be a non-decreasing sequence of $\calT$-closed convex sets whose union is $X$.
It offers characterizations of $\calT_\calC$-convergent sequences under varying circumstances, Theorem \ref{CCC.1}\refmod{\ref{CCC.1.c1}} and Theorem \ref{CCC.4}\refmod{\ref{CCC.4.c1}}.
Whether ``sequences suffice'' to characterize topological notions is subtle, as it depends on the notion under scrutiny\footnote{For instance, the notion of bounded set in a topological vector space is determined by its convergent sequences, \ref{LCTVS.5}(B).}.
Specifically, a Hausdorff topological space $X$ can be:
\begin{itemize}
\item {\it sequential}, \ie for all Hausdorff topological spaces $Y$ and all $f:X \to Y$ if $f$ is sequentially continuous then $f$ is continuous;
\item {\it Fr\'echet-Urysohn}, \ie for all $A \subset X$ and all $x \in \rmclos A$ there exists a sequence in $A$ converging to $x$ (in particular, if $A$ is dense then it is sequentially dense).
\end{itemize}
The latter implies the former but the converse fails in general.
\refmod{Confusion is at its worst} when the terminology ``sequential'' is used in place of ``Fr\'echet-Urysohn'' as is the case, for instance, in the author's \cite{DEP.97} and in \cite{YAMAMURO}.
Details are given in Appendix \ref{app.A} whose relevance here stems from the fact that in typical applications, \eg Section \ref{sec.OE}, $X[\calT_\calC]$ is sequential, \ref{OE.4}(A), but is not Fr\'echet-Urysohn\footnote{A typical analyst interested in the PDE $\rmdiv v = F$ is by now walking on eggs. Of course, as with many things, it only takes to get acquainted with the situation.}, \ref{OE.4}(B).
\par 
Theorem \ref{CWC.6}\refmod{\ref{CWC.6.c1}} gives a sufficient condition for $X[\calT_\calC]$ to be sequential: One assumes that each $C \in \calC$ is also $\calT$-compact and that $\calT \hel C$ is sequential.
Note that the latter is weaker than assuming that $\calT$ itself is sequential (the case when $\calT$ is a weak topology associated with a Banach space structure on $X$ will arise in applications and weak topologies are never sequential unless $X$ is finite-dimensional).
\par 
In the setting of Section \ref{sec.CCC}, $X[\calT_\calC]$ may not be sequential but Theorem \ref{CCC.2} yields a slightly weaker version, strong enough to imply that sequentially continuous {\it linear} maps defined on $X$ (or seminorms) are continuous.
This property of $X$ holds when $X$ is bornological, see Remark \ref{CCC.3}(2) for a reference and Appendix \ref{app.BORN} for the definition of bornological spaces.
Notwithstanding, in typical applications, see \eg Section \ref{sec.OE}, $X[\calT_\calC]$ is not bornological, \ref{OE.4}(D).
It is not barrelled either, \ref{OE.4}(C) (see Appendix \ref{app.BARREL} for a definition of barrelled space). 
The latter means that the Banach-Steinhaus theorem fails: if $\la F_n \ra_n$ is a sequence in the dual of $X[\calT_\calC]$ that converges pointwise to a linear form $F$ then $F$ may not be $\calT_\calC$-continuous, see the proof of \ref{OE.4}(C) for an example.
Theorems \ref{CCC.6} and \ref{CCC.7} prove in great generality that $X[\calT_\calC]$ is \refmod{neither Fr\'echet-Urysohn, nor bornological, nor barrelled} unless $\calT = \calT_\calC$.
\\
\par 
The setting of Section \ref{sec.CWC} is a specialization of that of Section \ref{sec.CCC}, namely one now assumes that the localizing family $\calC$ is a non-decreasing sequence of {\it $\calT$-compact} (rather than merely $\calT$-closed) convex sets whose union is $X$.
In that case, a substantial simplification occurs in describing the localized topology $\calT_\calC$, see Theorem \ref{CWC.2} where also a simple criterion is given for the $\calT_\calC$-openness and the $\calT_\calC$-closedness of arbitrary (not necessarily convex) subsets of $X$.
\par 
This remarkable simplification is instrumental in the main theorems of the paper but it necessitates appropriate modifications to the setup of our case study.
Indeed, in this case,
\begin{equation*}
C_k = \calD(\Rm) \cap \left\{ \vphi : \rmspt \vphi \subset B(0,k) \text{ and } \|\nabla \vphi \|_{L_1} \leq k \right\}
\end{equation*}
and it is plain that $C_k$ is not $\|\cdot\|_{L_1}$-compact.
Thus, we ought to replace $X = \calD(\Rm)$ with a larger set that will guarantee the requested compactness.
\par 
The natural candidate is $X = BV_c(\Rm)$, the vector space consisting of (equivalence classes of) functions of bounded variation\footnote{\ie those Lebesgue-summable  $u$ whose distributional partial derivatives $\partial_i u$ are signed measures.} whose support is compact.
Observe that $\|u\| = \|u\|_{L_1}$ is defined for $u \in BV_c(\Rm)$ and that $\|\nabla \vphi\|_{L_1}$ may be replaced by the total variation $\lno u \rno = \|Du\|(\Rm)$.
Letting $X_k = BV_c(\Rm) \cap \{ u : \rmspt u \subset B(0,k) \}$ we observe that
\begin{equation*}
C_k = X_k \cap \{ u : \lno u \rno \leq k \} = BV_c(\Rm) \cap \{ u : \rmspt u \subset B(0,k) \text{ and } \|Du\|(\Rm) \leq k \}
\end{equation*}
is indeed $\|\cdot\|_{L_1}$-compact.
Moreover, if $v \in C(\Rm;\Rm)$ then we may define its divergence acting on $BV_c(\Rm)$ by the formula
\begin{equation*}
BV_c(\Rm) \to \R : u \mapsto - \sum_{i=1}^m \int_{\Rm} v_i \, d(\partial_i u).
\end{equation*}
Using the same approximation of $v$ by a smooth $w$ as in \eqref{eq.INTRO.3} and using the results of Section \ref{sec.CCC}, it is now a matter of routine verification to show that $\rmdiv v$ defined this way is $\calT_\calC$-continuous, where $\calT_\calC$ is the localized topology on $BV_c(\Rm)$ associated with $\calT = \calT_{\|\cdot\|_{L_1}}$, $\calC = \{ C_k : k \in \N\}$, and the $C_k$ are as above.
\par 
It remains to establish a bijective correspondence between the members of the dual of $BV_c[\calT_\calC]$ and the distributions described in \eqref{eq.INTRO.1}, see \ref{OE.5}, based on the fact that $\calD(\Rm)$ is {\it uniformly sequentially $\calT_\calC$-dense} in $BV_c(\Rm)$ (recall the relevant earlier comment when describing Section \ref{sec.GO}).
\\
\par 
Section \ref{sec.FD} considers the strong topology $\beta(X^*,X)$ on the dual of $X[\calT_\calC]$, \ie the topology of uniform convergence on $\calT_\calC$-bounded subsets of $X$, under the assumption that the $\calT_\calC$-bounded subsets of $X$ are precisely the subsets of members of $\calC$.
This condition holds, for instance, in the setting of Section \ref{sec.CWC}, \ie when members of $\calC$ are $\calT$-compact.
Thus, $\calC$ is a fundamental system of $\calT_\calC$-bounded sets in $X$ and, since it is also assumed to be countable, the strong topology $\beta(X^*,X)$ on the dual is metrizable.
Unlike in the case of the dual of a normed space, this does not immediately imply that $X[\calT_\calC]^*[\beta(X^*,X)]$ is complete (\ie a Fr\'echet space), yet we easily check this in Theorem \ref{FD.2}.
It should be noted that if $X_k = X$ for all $k$ then $X[\calT_\calC]^*[\beta(X^*,X)]$ is, in fact, a Banach space, Theorem \ref{FD.3}.
The remaining part of Section \ref{sec.FD} is devoted to the problem of extending $\calT_\calC$-continuous linear forms from a uniformly sequentially $\calT_\calC$-dense subspace $Y \subset X$ to $X$.
\\
\par 
Section \ref{sec.BWS} presents a useful special case of localized topologies, namely that of a {\it bounded weak* topology}.
Let $Y$ be a locally convex topological vector space and $\calB$ a local base for its topology.
It is easy to see that the collection $\calB^\circ = \{ U^\circ : U \in \calB\}$ of polars of members of $\calB$ is a localizing family in $Y^*$.
If $\sigma(Y^*,Y)$ denotes the usual weak* topology on the dual $Y^*$ then the localized topology $\sigma(Y^*,Y)_{\calB^\circ}$ on $Y^*$ is its {\it bounded weak* topology} which we also denote by $\sigma_{\rmbd}(Y^*,Y)$.
Theorem \ref{BWS.3} gives an explicit local base for $\sigma_{\rmbd}(Y^*,Y)$ whereas theorems \ref{BWS.4} and \ref{BWS.5} are, respectively, those of Banach-Grothendieck and Krein-\v{S}mulian.
The content of this section is classical and we present the definitions and proofs for completeness.
Specifically, the proofs of \ref{BWS.3} and \ref{BWS.4} are minor generalizations of, respectively, \cite[Ch. V \S5 Lemma 4]{DUNFORD.SCHWARTZ.I} and \cite[Ch. V \S5 Theorem 6]{DUNFORD.SCHWARTZ.I}.
\\
\par 
Section \ref{sec.SD} makes the link between some localized topologies, such as that of our case study, and bounded weak* topologies.
In order to state this, we need to recall the notion of {\it semireflexivity}.
The first dual $X[\calT_\calC]^*$ equipped with the strong topology $\beta(X^*,X)$ has itself a dual space which we denote as $X^{**}$.
The evaluation map $\rmev : X \to X^{**}$ is well-defined and injective in general.
We say that $X[\calT_\calC]$ is {\it semireflexive} if $\rmev$ is also surjective, \ie whenever $X^{**}$ is linearly identified with $X$ via the evaluation map.
\new{Under the hypotheses of Theorem \ref{SD.5}, the members of $\calC$ are $\calT$-compact if and only if $X[\calT_\calC]$ is semireflexive and the evaluation map is a homeomorphism onto $X^{**}$ equipped with its bounded weak* topology. Thus, compactness yields both an algebraic identification of $X$ with its bidual and a topological identification of $\calT_\calC$ with $\sigma_{\rmbd}(X^{**},X^*)$.}
%
%
Via this identification we gain, for instance, an explicit description of a local base for the localized topology.
The proof of our semireflexivity theorem relies on the Mackey-Arens theorem, see Appendix \ref{app.DUAL}.
One implication of the semireflexivity theorem was obtained in a slightly different way in \cite[Theorem 3.16]{DEP.MOO.PFE.08}.
\new{The compactness hypothesis therefore provides the bidual identification used in the existence theorem of Section \ref{sec.AET}.}
\\
\par 
Section \ref{sec.AET} contains the main abstract existence Theorem \ref{AET.FR} that we illustrate here for our case study only.
\begin{equation*}
\begin{CD}
C(\Rm;\Rm)^* \simeq M_c(\Rm;\Rm) @<{\nabla = \rmdiv^*}<< BV_c(\Rm)[\calT_\calC]^{**} \simeq BV_c(\Rm)\\
C(\Rm;\Rm) @>{\rmdiv}>> BV_c(\Rm)[\calT_\calC]^*
\end{CD}
\end{equation*}
Here, $C(\Rm;\Rm)$ is equipped with its structure of a Fr\'echet space (local uniform convergence) and its dual is identified with the space $M_c(\Rm;\Rm)$ of compactly supported $\Rm$-valued Borel measures on $\Rm$. 
The bidual $BV_c(\Rm)[\calT_\calC]^{**}$ is identified with $BV_c(\Rm)[\calT_\calC]$ by the semireflexivity theorem.
The adjoint of $v \mapsto \rmdiv v$ is the operator $u \mapsto \nabla u$ that sends $u \in BV_c(\Rm)$ to its distributional gradient $\nabla u = (\partial_1u,\ldots,\partial_mu)$. 
\par 
The goal is to prove that the operator $\rmdiv$ is surjective.
We first observe that its range is dense.
There are at least two ways to see this.
One way is by smoothing (members of $BV_c(\Rm)[\calT_\calC]^*$) as in \cite{DEP.PFE.06b}.
In the present paper, instead, we notice that $\lno \cdot \rno$ is, in fact, a norm (by a constancy theorem) and rely on Hahn-Banach's theorem to infer that the range of $\rmdiv$ is dense -- step {\bf (iii)} of the proof of Theorem \ref{AET.FR}.
\par 
It remains to show that the range of $\rmdiv$ is closed which is equivalent, by the closed range theorem, to showing that the range of $\nabla$ is weakly* closed.
By the Krein-\v{S}mulian theorem, we reduce this to showing that $\rmim(\nabla) \cap V_k^\circ$, $k \in \N$, are weakly* closed, where $V_k = C(\Rm;\Rm) \cap \{ v : \sup\{ |v(x)| : x \in B(0,k) \} \leq k^{-1} \}$.
Upon observing that the polar $V_k^\circ$ consists of those $\Rm$-valued measures supported in $B(0,k)$ and whose total variation does not exceed $k$, the semireflexivity theorem allows to apply the $BV$-compactness theorem in the range of $\nabla$, thereby completing the proof. 
The argument for proving the abstract existence theorem follows similar lines. 
\par 
In the context of our case study the operator $\rmdiv$ does not admit a bounded linear right inverse.
The proof of this is, in essence, in \cite[Second proof of Proposition 2, p.401]{BOU.BRE.03}. 
In fact, $\rmdiv$ does not even admit a uniformly continuous non-linear right inverse, see \cite{BOU.11}.
Yet, by Michael's selection theorem (see Appendix \ref{app.MCC}) $\rmdiv$ admits a continuous, positively homogeneous, non-linear right inverse.
In Theorem \ref{AET.FR}\refmod{\ref{AET.FR.c4}} we also gather extra estimates that such right inverse may be requested to satisfy -- the same bound that a bounded linear right inverse would satisfy if it existed.
Very roughly speaking, this means that there is no linear formula (in particular, no formula involving convolution of $F$ with a kernel) giving a solution $v$ continuously in terms of $F$.
Whether there is an explicit non-linear algorithm that gives $v$ continuously in terms of $F$ will be addressed in future investigations -- in this respect, see also \cite{COH.DEV.TAD.24}.
\\
\par 
Section \ref{sec.CT} characterizes the $\beta(X^*,X)$-compact subsets of the dual $X[\calT_\calC]^*$ as those strongly closed sets $\calF$ for which the continuity condition \eqref{eq.INTRO.2} holds uniformly in $f \in \calF$, \ie given $\veps > 0$ the index $i \in I$ and $\theta > 0$ can be chosen independently of $f \in \calF$.
This is a consequence of Ascoli's theorem and interesting to know, since a common version of the Banach-Alaoglu theorem (see Appendix \ref{PR.100}) does not apply because $X[\calT_\calC]$ is in general not barrelled.
\\
\par 
Appendix \ref{app.A} contains the definition of sequential and Fr\'echet-Urysohn topological spaces.
It gives a useful characterization \ref{PR.3}(D) of Fr\'echet-Urysohn spaces in terms of the possibility of applying a diagonal argument to double-indexed sequences.
Appendix \ref{app.LCTVS} serves as ``preliminaries and notation'' for all that concerns locally convex topological vector spaces.
Appendix \ref{app.DUAL} provides a detailed presentation of dual systems, weak topologies and polar topologies (in particular, the Mackey topology and the strong topology) and offers a full proof of the Mackey-Arens theorem \ref{DUAL.6}(E) and its Corollary \ref{DUAL.6}(H) which is the main tool used to prove the semireflexivity theorem \ref{SD.5}. 
Appendices \ref{app.BARREL} and \ref{app.BORN} give a short account of barrelled and bornological spaces, respectively. 
Appendix \ref{app.MCC} proves Michael's selection theorem in the case of Fr\'echet spaces.
We hope the material contained in these appendices will help the reader find their way in the main body of the paper. 
\\
\par 
We close this introduction with a short account of the divergence equation, \ie the equation $\rmdiv v = F$.
In many cases, a solution $v$ can be obtained as deriving from a potential, \ie $v = \nabla u$, where $\triangle u = F$.
However, in some critical cases, the regularity of $\nabla u$ is not optimal.
For instance, if $F \in L^p(\Rm)$ and $p=m$ then $\nabla u \in W^{1,m}(\Rm)$ may not be continuous\footnote{The following example is due to Nirenberg: $u(x)=\vphi(x)\cdot x_1 \cdot |\log |x|_2|^\alpha$, $0 < \alpha < \frac{m-1}{m}$, see \cite[Section 3 Remark 7]{BOU.BRE.03}.} even though there exists a continuous vector field $v$ such that $\rmdiv v = F$, see Theorem \ref{OE.3} and \ref{OE.2}(D).
In fact, $p=1$, $p=m$, and $p=\infty$ are critical cases for the Poisson equation, see \eg the survey \cite{RUS.13}.
Bourgain-Brezis' landmark paper \cite{BOU.BRE.03} proved -- among many other things -- that if $F$ is $m$-summable then there exists a continuous $v$ whose divergence equals $F$ (this was in the periodic setting, $\Rm$ being replaced by $\mathbb{T}^m$) and, inspired by this, \cite{DEP.PFE.06b} characterized those distributions $F$ that are the divergence of a continuous vector field.
\refmod{Apart from the application in Section 10,} the abstract existence theorem proved here will be shown in \cite{DEP.26a} to also encompass the following cases: 
\begin{itemize}
\item a boundary condition $v=0$ at infinity (\refmod{see also} \cite{DEP.TOR.09} and \cite{MOO.PIC.18});
\item a boundary condition $v=0$ when $\Rm$ is replaced with an open set $\Omega$ \refmod{which is either a bounded, connected, $BV$-extension set or satisfies a $(p,1)$-Poincar\'e inequality for some $1 < p \leq \frac{m}{m-1}$};
\item The continuity condition of $v$ is replaced by Lebesgue-summability (this pertains to Lipschitz-free spaces, see \cite{GOD.15} for a survey and \cite{DEP.26b});
\end{itemize}
\refmod{Moreover, it may be shown that our abstract existence theorem also applies to the case where the vector fields are replaced with differential forms $\omega$ and the equation $\rmdiv v = F$ by $d\omega = F$ (see \cite{DEP.MOO.PFE.08}) and the domain $\Rm$ by manifolds, singular varieties and, more generally, some appropriate classes of doubling measure spaces.} 

\par 
\refmod{We conclude this introduction by recording our numbering conventions. Hypotheses of a theorem are labelled (H1), (H2), \ldots; conclusions drawn from these hypotheses are labelled (C1), (C2), \ldots; items of standalone statements are labelled (A), (B), \ldots.}
%

\section{Existence and uniqueness of a localization}
\label{sec.EUL}

\begin{Empty}
\label{EUL.1}
Given a set $X$, $C \subset X$, and $\calT$ a set of subsets of $X$ we define $\calT \hel C = \{ E \cap C : E \in \calT \}$. 
Thus, $\calT \hel C$ is a set of subsets of $C$. If $\calT$ is a topology on $X$ then $\calT \hel C$ is a topology on $C$, so-called \textit{induced} by $\calT$.
Notice that $(\calT \hel D) \hel C = \calT \hel (D \cap C) = \calT \hel C$ whenever $C \subset D$. 
\end{Empty}

\begin{Empty}
\label{EUL.2}
Let $X$ be a vector space and $\calC$ a set of subsets of $X$. 
We say that $\calC$ is a {\em localizing family in $X$} if it satisfies the following conditions.
\begin{enumerate}
\item[(i)] $\calC$ is non-empty;
\item[(ii)] Each member of $\calC$ is convex and contains 0;
\item[(iii)] For every $C \in \calC$, every $x \in X$, and every $t \in \R$ there exists $D \in \calC$ such that $x+t \cdot C \subset D$. 
\end{enumerate}
In that case one easily checks that $X = \cup \calC$.
\end{Empty}

\begin{Empty}
\label{EUL.3}
Let $X[\calT]$ be a locally convex topological vector space and let $\calC$ be a localizing family on $X$. 
We say that $\calS$ is a {\em localization of $\calT$ by $\calC$} if $\calS$ is a locally convex vector topology on $X$ such that
\begin{enumerate}
\item[(i)] For every $C \in \calC$, $\calS \hel C \subset \calT \hel C$;
\item[(ii)] For every locally convex topological vector space $Y[\calU]$ and every linear map $f : X \to Y$, {\em if} the restriction $f|_C : C \to Y$ is $(\calT \hel C,\calU)$-continuous, for every $C \in \calC$, {\em then} $f$ is $(\calS,\calU)$-continuous. 
\end{enumerate}
The remaining part of this section is devoted to proving the existence and uniqueness of such $\calS$ given $\calT$ and $\calC$.
We should note right away that the inclusion in (i) above is, in fact, an equality \ref{GO.1}\refmod{\ref{GO.1.2}} and is also the reason why our vocabulary includes the word \textit{ localization} for lack of better options: the topologies $\calS$ and $\calT$ coincide ``locally'' in every $C \in \calC$.
\end{Empty}

\begin{Theorem}
\label{EUL.4}
Assume $X[\calT]$ is a locally convex topological vector space and $\calC$ is a localizing family in $X$. 
There then exists a unique localization of $\calT$ by $\calC$ and it is given by the formula in \ref{EUL.6} below.
\end{Theorem}

\begin{Remark}
\label{EUL.5}
Starting from the next section this topology will be called {\em the localization of $\calT$ by $\calC$} and denoted by $\calT_\calC$. 
It is in general not an inductive limit in the category of locally convex topological vector spaces (unless each member of $\calC$ is a vector subspace of $X$) and the ensuing lack of heredity of some useful properties of locally convex vector topologies calls for care, see for instance \ref{CCC.6.R}(Z)(ii) and (iii) and compare with \cite[Ch. II \S7.2 corollary 1 and \S8.2 corollary 1]{SCHAEFER.I}.
In fact, the universal property of $X[\calT_\calC]$ is not expressible in the category of locally convex topological vector spaces because $C[\calT \hel C]$, $C \in \calC$, are not necessarily objects of this category.
However, in many cases of interest the topology $\calT_\calC$ is an inductive limit in the category of Hausdorff topological spaces, see \ref{CWC.3}.
\end{Remark}

\begin{proof}[Proof of uniqueness]
Let $\calS_j$ satisfy the two conditions of \ref{EUL.3} labelled as $\mathrm{(i)}_j$ and $\mathrm{(ii)}_j$, $j=1,2$. 
We apply $\mathrm{(ii)}_1$ to $Y$ being $X[\calS_2]$ and $f$ being $\rmid_X$. Given $C \in \calC$ and $O \in \calS_2$, $(\rmid_X)|_C^{-1}(O)=O \cap C \in \calS_2 \hel C \subset \calT \hel C$ where the last inclusion ensues from $\mathrm{(i)}_2$. 
Thus, $\rmid_X$ is $(\calS_1,\calS_2)$-continuous. 
Reversing the roles played by $\calS_1$ and $\calS_2$ we see that $\rmid_X$ is also $(\calS_2,\calS_1)$-continuous.
\end{proof}

\begin{Empty}
\label{EUL.6}
The proof of existence consists in showing that the following is a localization of $\calT$ by $\calC$:
\begin{equation*}
\begin{split}
\calS = \calP(X) \cap \big\{ E : &\text{ For every $x \in E$ there exists a convex set $V \subset X$ such that}\\
& \mathrm{(a)} \; x \in V \subset E \\
& \mathrm{(b)} \; \forall C \in \calC : V \cap C \in \calT \hel C \big\} \,.
\end{split}
\end{equation*}
Notice that $\calS$ is indeed a topology on $X$.
\end{Empty}

\begin{proof}[Proof of condition \ref{EUL.3}(i)]
Fix $C \in \calC$. 
Let $S \in \calS \hel C$, \ie $S = E \cap C$ for some $E \in \calS$. 
We ought to show that $S \in \calT \hel C$.  
With each $x \in E$ associate a convex set $V_x$ such that $x \in V_x \subset E$ and $V_x \cap C \in \calT  \hel C$. 
The latter means that $V_x \cap C = O_x \cap C$ for some $O_x \in \calT$. 
Define $O = \cup_{x \in E} O_x \in \calT$. 
Notice that $E = \cup_{x \in E} V_x$. 
Therefore,
\begin{equation*}
S = E \cap C = \cup_{x \in E} (V_x \cap C) = \cup_{x \in E} (O_x \cap C) = O \cap C \in \calT \hel C \,.
\end{equation*}
\end{proof}

\begin{proof}[Proof of condition \ref{EUL.3}(ii)]
Let $O \subset Y$ be open. 
We ought to show that $f^{-1}(O) \in \calS$. 
Let $x \in f^{-1}(O)$. 
Since $f(x) \in O$ and $Y$ is locally convex there exists an open convex set $W$ such that $f(x) \in W \subset O$. 
Consider the convex set $V = f^{-1}(W)$.  
Given $C \in \calC$ we note that $V \cap C = f^{-1}(W) \cap C = (f|_C)^{-1}(W) \in \calT \hel C$, owing to the $\calT \hel C$-continuity of $f|_C$.
\end{proof}

\begin{Empty}
\label{EUL.7}
\textit{ 
If $E \in \calS$ and $a \in X$ then $E + a \in \calS$.
}
\end{Empty}

\begin{proof}
Given $x \in E + a$, \ie $x-a \in E$, we choose a convex set $V$ such that $x - a \in V \subset E$ and $V \cap D \in \calT \hel D $ for every $D \in \calC$. 
Letting $W = V + a$ we notice that $W$ is convex and $x \in W \subset E + a$. 
We ought to show that $W \cap C \in \calT \hel C$ for every $C \in \calC$. 
Fix $C \in \calC$ and observe that
\begin{equation*}
\begin{split}
W  \cap C = (V + a) \cap C & = \big[ V \cap (C - a) \big] + a 
\intertext{(by \ref{EUL.2}(iii) there exists $D \in \calC$ such that $C - a \subset D$, since $\calC$ is a localizing family)}
& \subset (V \cap D ) + a
\intertext{($V \cap D = U \cap D$ for some $U \in \calT$, since $V \cap D \in \calT \hel D$)}
& = (U \cap D) + a \\
& = (U + a) \cap (D + a) \,.
\end{split}
\end{equation*}
Thus, $W \cap C \subset (U + a) \cap (D + a) \cap C = (U+a) \cap C$. 
We now establish the reverse inclusion. 
Let $y \in (U+a) \cap C$. 
Thus, $y-a \in U  \cap (C-a) \subset U \cap D = V \cap D$ so that $y \in (V+a) \cap (D+a) = W \cap (D+a)$ and finally $(U+a) \cap C \subset W \cap (D+a) \cap C = W \cap C$. 
This shows that $W \cap C = (U+a) \cap C \in \calT \hel C$, since $U+a \in \calT$.
\end{proof}

\begin{Empty}
\label{EUL.8}
\textit{ 
If $E \in \calS$ and $t \in \R \setminus \{0\}$ then $t\cdot E  \in \calS$.
}
\end{Empty}

\begin{proof}
Similar to that of \ref{EUL.7}
\end{proof}

\begin{Empty}
\label{EUL.9}
\textit{ 
If $V \subset X$ is convex and $V \cap C \in \calT \hel C$ for all $C \in \calC$ then $V \in \calS$.
}
\end{Empty}

\begin{proof}
Obvious.
\end{proof}

\begin{proof}[Proof that addition of vectors in $X$ is $\calS$-continuous]
Let $x,y \in X$ and let $E \in \calS$ be such that $x + y \in E$. 
We ought to find $F,G \in \calS$ such that $x \in F$, $y \in G$, and $F + G \subset E$. 
Let $V$ be a convex set associated with $E$ and $x+y$ in the definition of $\calS$. 
Thus $x + y \in V \subset E$ and $V \in \calS$, according to \ref{EUL.9}. 
Therefore, the convex set $W = V -(x+y)$ belongs to $\calS$, according to \ref{EUL.7}, and so does $\frac{1}{2} \cdot  W$, according to \ref{EUL.8}. 
Since $0 \in \frac{1}{2} \cdot  W$, the sets $F = x + \frac{1}{2} \cdot  W$ and $G = y + \frac{1}{2} \cdot  W$ are open $\calS$-neighbourhoods of $x$ and $y$, respectively and, since $\frac{1}{2} \cdot  W + \frac{1}{2} \cdot  W \subset W$, we have $F + G \subset V \subset E$.
\end{proof}

\begin{Empty}
\label{EUL.10}
\textit{ 
If $E \in \calS$ and $0 \in E$ then $E$ is absorbing.
}
\end{Empty}

\begin{proof}
Let $E$ be as in the statement and let $V$ be a convex set such that $0 \in V \subset E$ and $V \cap C \in \calT \hel C$ for every $C \in \calC$. 
Let $x \in X$. 
There exists $C \in \calC$ such that $x \in C$. 
Choose $U \in \calT$ with $V \cap C = U \cap C$. 
Recalling \ref{EUL.2}(ii) notice that $0 \in U$. 
There exists $t \geq 1$ such that $x \in t \cdot U$, according to \cite[Theorem 1.15(a)]{RUDIN}. 
Notice that $t^{-1} \cdot x \in C$, since $C$ is convex and $0,x \in C$. 
Therefore, $t^{-1} \cdot x \in U \cap C = V \cap C \subset E$. 
\end{proof}

\begin{proof}[Proof that multiplication of a scalar by a vector in $X$ is $\calS$-continuous]
Let $t \in \R$, $x \in X$, and $E \in \calS$ be such that $t \cdot x \in E$. 
We ought to find $\veps > 0$ and $F \in \calT$ containing $x$ such that $s \cdot y \in E$ whenever $|s-t| < \veps$ and $y \in F$. 
Let $V$ be a convex set associated with $E$ and $t \cdot x$ in the definition of $\calS$. 
Thus, $t \cdot x \in V \subset E$ and $V \in \calS$, according to \ref{EUL.9}. 
The convex set $W = \frac{1}{2} \cdot (V -t \cdot x)$ belongs to $\calS$, according to \ref{EUL.7} and \ref{EUL.8}, and so does $W' = W \cap (-W)$. 
Since $0 \in W'$, $W'$ is absorbing, according to \ref{EUL.10}. 
Thus, there exists $t' > 0$ such that $x \in t' \cdot W'$. 
Define $\veps' = (t')^{-1}$ and notice that $\tau \cdot x \in W'$ whenever $|\tau| \leq \veps'$, because $0,\veps' \cdot x \in W'$ and $W'$ is convex and symmetric. 
Define also $F = x + (1+|t|)^{-1} \cdot W'$ so that $F \in \calS$, according to \ref{EUL.8} and \ref{EUL.7}, and $x \in F$. 
Observe now that if $y \in F$ and $|s-t| \leq 1$ then $s \cdot (y-x) = \frac{s}{1+|t|}(1+|t|) \cdot (y-x) \in W'$. 
Therefore, if $|s-t| \leq \veps := \min\{1,\veps'\}$ and $y \in F$ then $s \cdot y - t \cdot x =s \cdot (y-x) + (s-t) \cdot x \in W' + W' \subset W + W \subset V-t \cdot x$ and in turn $s \cdot y \in V \subset E$.
\end{proof}

\begin{proof}[Proof that $\calS$ is a locally convex vector topology]
We already proved that $X \times X \to X : (x,y) \mapsto x + y$ and $\R \times X \to X : (t,x) \mapsto t \cdot x$ are continuous. 
We now show that $\{x\}$ is $\calS$-closed for each $x \in X$. 
Put $E = X \setminus \{x\}$ and let $y \in E$. 
Since $\{x\}$ is $\calT$-closed and $\calT$ is locally convex, there exists a convex set $V \in \calT$ such that $y \in V \subset E$. 
Obviously $V \cap C \in \calT \hel C$ for every $C \in \calC$. 
Therefore, $E \in \calS$. 
This proves that $\calS$ is a vector topology on $X$. 
The following collection $\calB$ is a local base of $\calS$ consisting of convex sets, according to the definition of $\calS$ and \ref{EUL.9}:
\begin{equation}
\label{eq.EUL.1}
\calB = \calP(X) \cap \big\{ V : \text{ $V$ is convex, $0 \in V$ and }
 V \cap C \in \calT \hel C \text{ for every }C \in \calC \big\} .
\end{equation}
\end{proof}

\section{General observations}
\label{sec.GO}

\begin{Empty}
Let $X$ be a vector space and let $\calC$ and $\calD$ be two localizing families in $X$.
We say that $\calC$ and $\calD$ are {\em interlaced} if each member of $\calC$ is contained in some member of $\calD$ and vice versa.
\end{Empty}

\begin{Theorem}
\label{GO.1}
Let $X[\calT]$ be a locally convex topological vector space and let $\calC$ be a localizing family in $X$. The following hold.
\begin{numenum}
\item \label{GO.1.1} $\calT \subset \calT_\calC$.
\item \label{GO.1.2} For every $C \in \calC$ one has $\calT_\calC \hel C = \calT \hel C$.
\item \label{GO.1.3} If $\calD$ is a localizing family in $X$ and $\calC$ and $\calD$ are interlaced then $\calT_\calC = \calT_\calD$.
\item \label{GO.1.4} If $B \subset X$ is $\calT_\calC$-bounded then it is $\calT$-bounded as well.
\item \label{GO.1.5} If $B \subset C \in \calC$ and $B$ is $\calT$-bounded then $B$ is $\calT_\calC$-bounded.
\item \label{GO.1.6} For all $C \in \calC$: $C$ is $\calT$-compact iff it is $\calT_\calC$-compact.
\item \label{GO.1.7} If $Y$ is a vector subspace of $X$ then $\calC \hel Y$ is a localizing family in $Y$ such that $\calT_\calC \hel Y \subset (\calT \hel Y)_{\calC \shel Y}$.
\end{numenum}
\end{Theorem}
\noindent
We recall the description of $\calT_\calC$ given in \ref{EUL.6} and its local base $\calB$ given above, see \eqref{eq.EUL.1}.

\begin{proof}[Proof of \refmod{\ref{GO.1.1}}]
Trivially, $(\rmid_X)|_C : C \to X$ is $(\calT \hel C , \calT)$-continuous for all $C \in \calC$.
Thus, $\rmid_X$ is $(\calT_\calC,\calT)$-continuous, by \ref{EUL.3}(ii).
%
%
%
%
\end{proof}

\begin{proof}[Proof of \refmod{\ref{GO.1.2}}]
Fix $C \in \calC$. 
Recall that $\calT_\calC \hel C \subset \calT \hel C$, \ref{EUL.3}(i). 
Let $E \in \calT \hel C$. 
There is $U \in \calT$ so that $E \cap C = U \cap C$. 
$U \in \calT_\calC$, by \refmod{\ref{GO.1.1}}, thus, $E  \in \calT_\calC \hel C$.
\end{proof}

\begin{proof}[Proof of \refmod{\ref{GO.1.3}}]
We shall show that $\calT_\calD$ is a localization of $\calT$ by $\calC$. 
The conclusion will ensue from the uniqueness \ref{EUL.4}  of this localization.
We refer to Definition \ref{EUL.3}
First note that $\calT_\calD$ is indeed a locally convex vector topology.
We ought to show that conditions (i) and (ii) of \ref{EUL.3} are satisfied.
\par 
\textbf{ (i)} Let $C \in \calC$.
Choose $D \in \calD$ containing $C$.
Since $\calT_\calD$ is a localization of $\calT$ by $\calD$, $\calT_\calD \hel D \subset \calT \hel D$.
Thus, $\calT_\calD \hel C = \left( \calT_\calD \hel D \right) \hel C  \subset (\calT \hel D) \hel C = \calT \hel C$.
\par 
\textbf{ (ii)} Let $Y[\calU]$ be a locally convex topological vector space and $f : X \to Y$ a linear map.
Assuming that $f|_C : C \to Y$ is $(\calT \hel C,\calU)$-continuous whenever $C \in \calC$, we ought to show that $f$ is $(\calT_\calD,\calU)$-continuous.
Given $D \in \calD$ choose $C \in \calC$ containing $D$.
The restriction $f|_D = \left( f|_C\right)|_D$ is continuous with respect to $(\calT \hel C) \hel D = \calT \hel D$.
Since $D$ is arbitrary and $\calT_\calD$ is a localization of $\calT$ by $\calD$, the conclusion follows.
\end{proof}

\begin{proof}[Proof of \refmod{\ref{GO.1.4}}]
This is an immediate consequence of \refmod{\ref{GO.1.1}}.
\end{proof}

\begin{proof}[Proof of \refmod{\ref{GO.1.5}}]
Let $B$ and $C$ be as in the statement and $V \in \calB$ (recall \eqref{eq.EUL.1}, the last displayed equation of Section \ref{sec.EUL}).
We ought to show that $B \subset t \cdot  V$ for some $t > 0$.
There is $U \in \calT$ such that $V \cap C = U \cap C$.
As $B$ is $\calT$-bounded, there exists $t \geq 1$ so that $B \subset t \cdot U$.
Notice that $C \subset t \cdot C$.
Thus, $B \subset  (t \cdot U) \cap (t \cdot C) = t \cdot (U \cap C) = t \cdot (V \cap C) \subset t \cdot  V$. 
\end{proof}

\begin{proof}[Proof of \refmod{\ref{GO.1.6}}]
For $A \subset C \in \calC$ we observe that $A$ is $\calT$-compact iff it is $(\calT \hel C)$-compact iff it is $(\calT_\calC \hel C)$-compact \refmod{\ref{GO.1.2}} iff it is $\calT_\calC$-compact. 
\end{proof}

\begin{proof}[Proof of \refmod{\ref{GO.1.7}}]
Clearly, $\emptyset \neq \calC \hel Y$ and each member of $\calC \hel Y$ is a convex set containing the origin. 
If $C \in \calC$, $y \in Y$, and $t \in \R$ then $y + t \cdot (C \cap Y) = (y + t \cdot C) \cap Y$ so that condition \ref{EUL.2}(iii) holds for $\calC \hel Y$, \ie $\calC \hel Y$ is a localizing family in $Y$.
\par 
Notice that the second assertion is equivalent to the $((\calT \hel Y)_{\calC \shel Y},\calT_\calC \hel Y)$-con\-ti\-nuity of $\rmid_Y$.
This continuity property of $\rmid_Y$ follows from showing that \ref{EUL.3}(ii) applies with $X[\calT]$, $\calC$, and $Y[\calU]$ replaced by $Y[\calT \hel Y]$, $\calC \hel Y$, and $Y[\calT_\calC \hel Y]$.
Let $D \in \calC \hel Y$.
There is $C \in \calC$ such that $D = C \cap Y$.
Since $\calT_\calC \hel C \subset \calT \hel C$, by \ref{EUL.3}(i), we have $(\calT_\calC \hel Y) \hel D = \calT_\calC \hel (C \cap Y) \subset \calT \hel (C \cap Y) = (\calT \hel Y) \hel D$, \ie $(\rmid_Y)|_D$ is $((\calT \hel Y) \hel D,\calT_\calC \hel Y)$-continuous.
\end{proof}

\section{Countable localizing families of closed convex sets}
\label{sec.CCC}

Abusing vocabulary ever so slightly we will say that a localizing family $\calC$ is {\em non-decreasing countable} if it can be numbered $\calC = \{ C_k : k \in \N\}$ in such a way that $C_k \subset C_{k+1}$, $k \in \N$.

\begin{Theorem}
\label{CCC.1}
Assume that:
\begin{hypenum}
\item \label{CCC.1.h1} $X[\calT]$ is a locally convex topological vector space;
\item \label{CCC.1.h2} $\calC$ is a non-decreasing countable localizing family in $X$;
\item \label{CCC.1.h3} Each member of $\calC$ is $\calT$-closed.
\end{hypenum}
The following hold.
\begin{conclenum}
\item \label{CCC.1.c1} A sequence $\la x_j \ra_j$ in $X$ converges to $x \in X$ with respect to the topology $\calT_\calC$ if and only if it converges to $x$ with respect to the topology $\calT$ and is contained in some $C \in \calC$;
\item \label{CCC.1.c2} A set $B \subset X$ is $\calT_\calC$-bounded if and only if it is $\calT$-bounded and contained in some $C \in \calC$.
\item \label{CCC.1.c3} If $Y$ is a vector subspace of $X$ and its topology $\calT_\calC \hel Y$ is sequential then $\calT_\calC \hel Y = (\calT \hel Y)_{\calC \shel Y}$.
\end{conclenum}
\end{Theorem}
\begin{proof}[Proof of \refmod{\ref{CCC.1.c1}}]
We must establish that for every sequence $\la x_j \ra_j$ in $X$ and $x \in X$ the following are equivalent.
\begin{enumerate}
\item[(a)] $\la x_j \ra_j$ $\calT_\calC$-converges to $x$;
\item[(b)] $\la x_j \ra_j$ $\calT$-converges to $x$ and is contained in some $C \in \calC$.
\end{enumerate}
We abbreviate $y_j = x_j - x$.
\par 
\textit{Proof that $(b) \Rightarrow (a)$}.
Assuming (b), notice that $x \in C$, since $C$ is $\calT$-closed.
By \ref{EUL.2}(iii), $\la y_j \ra_j$ is a sequence in some $D \in \calC$.
As $\la y_j \ra_j$ is $\calT$-null, it is $(\calT \hel D)$-null, hence, also $(\calT_\calC \hel D)$-null, by \ref{GO.1}\refmod{\ref{GO.1.2}}, \ie it is $\calT_\calC$-null.
\par 
\textit{Proof that $(a) \Rightarrow (b)$}.
Assuming (a), observe that $\la x_j \ra_j$ $\calT$-converges to $x$, according to \ref{GO.1}\refmod{\ref{GO.1.1}}.
It remains to show that there exists $C \in \calC$ such that $x_j \in C$ for all $j \in \N$.
Referring again to \ref{EUL.2}(iii) notice that it suffices to establish that the $\calT_\calC$-null sequence $\la y_j \ra_j$ is contained in some member of $\calC$.
Assume if possible that this is not the case.
Then, since $\la C_k \ra_k$ is non-decreasing, there exists an increasing sequence of positive integers $j_1 < j_2 < \cdots$ such that $y_{j_k} \not \in C_k$ for all $k \in \N$.
As $\calT$ is locally convex and $C_k$ is $\calT$-closed and convex, Hahn-Banach's theorem \ref{LCTVS.8}(C) applies to showing the existence of a convex $\calT$-open set $U_k$ such that $C_k \subset U_k$ and $y_{j_k} \not \in U_k$, $k \in \N$. 
Thus, the set $U = \cap_{k=1}^\infty U_k$ is convex, contains 0, since each $C_k$ does, and $y_{j_k} \not \in U$ for all $k$. 
The proof will be complete upon showing that $U$ is $\calT_\calC$-open, as this would readily contradict the $\calT_\calC$-convergence of $\la y_j \ra_j $ to 0.
By \ref{EUL.9}, it is enough to observe that $U \cap C_n \in \calT \hel C_n$ for all $n$. 
Indeed, if $k > n$ then $C_n \subset C_k  \subset U_k$ so that $C_n \cap U_k = C_n$ and in turn
\begin{equation*}
U \cap C_n = \cap_{k=1}^\infty (U_k \cap C_n) = \left( \cap_{k=1}^n U_k \right) \cap C_n \in \calT \hel C_n \,.
\end{equation*}
\end{proof}

\begin{proof}[Proof of \refmod{\ref{CCC.1.c2}}] 
We must establish that for every $B \subset X$ the following are equivalent.
\begin{enumerate}
\item[(a)] $B$ is $\calT_\calC$-bounded;
\item[(b)] $B$ is $\calT$-bounded and is contained in some $C \in \calC$.
\end{enumerate}
\par 
\textit{Proof that $(b) \Rightarrow (a)$}.
This is \ref{GO.1}\refmod{\ref{GO.1.5}}.
\par 
\textit{Proof that $(a) \Rightarrow (b)$}.
In view of \ref{GO.1}\refmod{\ref{GO.1.4}} it suffices to establish that $B$ is contained in some $C \in \calC$.
First note that for each $k\in \N$ there is $j_k \in \N$ such that $k \cdot C_k \subset C_{j_k}$, according to \ref{EUL.2}(iii).
Assume if possible that no $C \in \calC$ contains $B$. 
Then for each $k$ there exists $x_k \in B \setminus C_{j_k}$. 
Since $B$ is $\calT_\calC$-bounded, the sequence $\la \frac{x_k}{k} \ra_k$ is $\calT_\calC$-null, by \ref{LCTVS.5}(B).
Yet for every $k$, $\frac{x_k}{k} \not \in C_k$, in contradiction with \refmod{\ref{CCC.1.c1}}.
\end{proof}

\begin{proof}[Proof of \refmod{\ref{CCC.1.c3}}]
We shall show that $\rmid_Y : Y \to Y$ is sequentially $(\calT_\calC \hel Y , \allowbreak (\calT \hel Y)_{\calC \shel Y})$-con\-tin\-u\-ous.
As $\calT_\calC \hel Y$ is assumed to be sequential, it will ensue that $(\calT \hel Y)_{\calC \shel Y} \subset \calT_\calC \hel Y$.
In view of \ref{GO.1}\refmod{\ref{GO.1.7}} the proof will be complete.
\par 
Let $\la y_j \ra_j$ be $\calT_\calC \hel Y$-null sequence in $Y$.
Then $\la y_j \ra_j$ is also a $\calT_\calC$-null sequence in $Y$.
According to \refmod{\ref{CCC.1.c1}}, $\la y_j \ra_j$ is $\calT$-null and contained in some $C \in \calC$.
Therefore, $\la y_j \ra_j$ is $\calT \hel Y$-null and contained in $C \cap Y \in \calC \hel Y$.
In view of \refmod{\ref{CCC.1.c1}} again, $\la y_j \ra_j$ is $(\calT \hel Y)_{\calC \hel Y}$-null.
Since $\la y_j \ra_j$ is arbitrary, we have established the sequential $(\calT_\calC \hel Y , (\calT \hel Y)_{\calC \shel Y})$-continuity of $\rmid_Y$.
\end{proof}

\begin{Theorem}
\label{CCC.2}
Assume that:
\begin{hypenum}
\item \label{CCC.2.h1} $X[\calT]$ is a locally convex topological vector space;
\item \label{CCC.2.h2} $\calC$ is a non-decreasing countable localizing family in $X$;
\item \label{CCC.2.h3} Each member of $\calC$ is $\calT$-closed;
\item \label{CCC.2.h4} For each $C \in \calC$ the topology $\calT \hel C$ is sequential.
\end{hypenum}
Then the following hold.
\begin{conclenum}
\item \label{CCC.2.c1} If $V \subset X$ is convex and $\calT_\calC$-sequentially open then it is $\calT_\calC$-open;
\item \label{CCC.2.c2} If $Y[\calU]$ is a locally convex topological vector space and $f : X \to Y$ is linear then the following are equivalent:
\begin{enumerate}
\item[(a)] $f$ is $(\calT_\calC,\calU)$-continuous;
\item[(b)] $f$ is sequentially $(\calT_\calC,\calU)$-continuous.
\end{enumerate} 
\item \label{CCC.2.c3} If $q$ is a seminorm on $X$ then the following are equivalent:
\begin{enumerate}
\item[(a)] $q$ is $\calT_\calC$-continuous;
\item[(b)] $q$ is sequentially $\calT_\calC$-continuous.
\end{enumerate} 
\end{conclenum}
\end{Theorem}

\begin{Remark}
\label{CCC.3}
The following comments are in order.
\begin{enumerate}
\item[(1)] If we assume that $\calT$ itself is sequential then each $\calT \hel C$ is sequential, $C \in \calC$, since $C$ is assumed to be $\calT$-closed, recall \ref{PR.2}(C)(a).
\item[(2)] Conclusion \ref{CCC.2}\refmod{\ref{CCC.2.c2}} says that for a linear map defined on $X[\calT_\calC]$ to be continuous it suffices that it be sequentially continuous. This very useful property holds when the domain of the linear map is bornological \cite[Ch. II \S8.3]{SCHAEFER.I}. However, none of the localized topologies that we have in mind in forthcoming applications are bornological, see \ref{CCC.6.R}(Z)(iii).
\end{enumerate}
\end{Remark}

\begin{proof}[Proof of \refmod{\ref{CCC.2.c1}}]
Let $V \subset X$ be convex and $\calT_\calC$-sequentially open. 
We ought to show that $V \in \calT_\calC$. 
According to \ref{EUL.9}, it suffices to prove that $V \cap C \in \calT \hel C$, for every $C \in \calC$. 
Fix $C \in \calC$. 
Since we assume $\calT \hel C$ to be sequential, it is enough to show that $V \cap C$ is $\calT \hel C$-sequentially open. 
Let $\la x_j \ra_j$ be a sequence in $V \cap C$ that converges to $x \in C$ with respect to $\calT \hel C$. 
Since $\la x_j \ra_j$ converges to $x$ with respect to $\calT$, it follows from \ref{CCC.1}\refmod{\ref{CCC.1.c1}} that this convergence occurs with respect to $\calT_\calC$ as well.
Thus, $\la x_j \ra_j$ is eventually in $V$, whence, also eventually in $V \cap C$.
\end{proof}

\begin{proof}[Proof of \refmod{\ref{CCC.2.c2}}]
Since the reverse implication is trivial, we prove only that $(b) \Rightarrow (a)$.
As $Y[\calU]$ is locally convex, it is enough to show that $f^{-1}(W) \in \calT_\calC$ whenever $W \in \calU$ is convex. 
As $f^{-1}(W)$ is convex, by linearity of $f$, it is sufficient to show that $f^{-1}(W)$ is $\calT_\calC$-sequentially open, according to \refmod{\ref{CCC.2.c1}}. 
Let $\la x_j \ra_j$ be a sequence in $X$ that converges to $x \in f^{-1}(W)$ with respect to $\calT_\calC$. 
Since $f$ is sequentially continuous, $\lim_j f(x_j) = f(x) \in W$, and, since $W$ is open, $\la f(x_j) \ra_j$ is eventually in $W$, \ie $\la x_j \ra_j$ is eventually in $f^{-1}(W)$. 
\end{proof}

\begin{proof}[Proof of \refmod{\ref{CCC.2.c3}}]
As in the proof of \refmod{\ref{CCC.2.c2}} it is sufficient to show that $(b) \Rightarrow (a)$.
For each $\veps > 0$ define $I_\veps = \,]-\veps,\veps[$.
Notice that $q^{-1}(I_\veps)$ is convex, since $q$ is a seminorm, and is sequentially $\calT_\calC$-open, since $q$ is sequentially $\calT_\calC$-continuous.
It follows from \refmod{\ref{CCC.2.c1}} that $q^{-1}(I_\veps)$ is $\calT_\calC$-open.
Now if $U \subset \R$ is open and $x \in q^{-1}(U)$ choose $\veps > 0$ such that $]\new{q(x)}-\veps,\new{q(x)}+\veps[\,\subset U$ and notice that $x + q^{-1}(I_\veps) \subset q^{-1}(U)$, by the triangle inequality. 
Since $x$ is arbitrary, $q^{-1}(U)$ is $\calT_\calC$-open.
\end{proof}

\begin{Theorem}
\label{CCC.4}
Assume that
\begin{hypenum}
\item \label{CCC.4.h1} $X[\calT]$ is a locally convex topological vector space whose topology is generated by a family of seminorms $\la \| \cdot \|_i \ra_{i \in I}$;
\item \label{CCC.4.h2} $\lno \cdot \rno$ is a seminorm on $X$ that is $\calT$-lower-semicontinuous;
\item \label{CCC.4.h3} $\la X_k \ra_k$ is a non-decreasing sequence of $\calT$-closed vector subspaces whose union is $X$;
\item \label{CCC.4.h4} $C_k = X_k \cap \{ x : \lno x \rno \leq k \}$ for each $k \in \N$;
\item \label{CCC.4.h5} $\calC = \{ C_k : k \in \N \}$.
\end{hypenum}
Then $\calC$ is a non-decreasing countable localizing family consisting of $\calT$-closed sets and the following hold. 
\begin{conclenum}
\item \label{CCC.4.c1} For every sequence $\la x_j \ra_j$ in $X$ the following are equivalent:
\begin{enumerate}
\item[(a)] $\lim_j x_j=0$ with respect to $\calT_\calC$.
\item[(b)] the following hold:
\begin{enumerate}
\item[(i)] $\lim_j \|x_j\|_i = 0$ for all $i \in I$;
\item[(ii)] $\sup_j \lno x_j \rno < \infty$;
\item[(iii)] there exists $k \in \N$ such that $x_j \in X_k$ for all $j \in \N$.
\end{enumerate}
\end{enumerate}
\item \label{CCC.4.c2} For every $B \subset X$ the following are equivalent:
\begin{enumerate}
\item[(a)] $B$ is $\calT_\calC$-bounded.
\item[(b)] the following hold:
\begin{enumerate}
\item[(i)] $\sup_{x \in B}  \|x\|_i < \infty$ for all $i \in I$;
\item[(ii)] \new{$\sup_{x \in B} \lno x \rno < \infty$};
\item[(iii)] \new{there exists $k \in \N$ such that $B \subset X_k$}.
\end{enumerate}
\end{enumerate}
\item \label{CCC.4.c3} The seminorm $\lno \cdot \rno$ is $\calT_\calC$-bounded and $\calT_\calC$-lower-semicontinuous.
\end{conclenum}
\end{Theorem}

\begin{Remark}
These observations will be applied in the important cases where:
\begin{itemize}
\item the topology $\calT$ is normable;
\item or $X_k = X$ for all $k \in \N$;
\item or both.
\end{itemize}
The following obvious simplifications occur.
\begin{enumerate}
\item[(1)] If $\calT$ is defined by a unique norm $\|\cdot\|$ then each occurrence of ``for all $i \in I$'' may be removed and each $\|\cdot\|_i$ replaced by $\|\cdot\|$.
\item[(2)] If $X_k = X$ for all $k \in \N$ then one may remove (iii) in \ref{CCC.4}\refmod{\ref{CCC.4.c1}} and \ref{CCC.4}\refmod{\ref{CCC.4.c2}}.
By conclusion \ref{CCC.4}\refmod{\ref{CCC.4.c3}}, if the localizing
seminorm $\lno \cdot \rno$ is not $\calT_\calC$-continuous,
then $X[\calT_\calC]$ is neither barrelled nor bornological
(see \ref{PR.8}\ref{PR.8.a2} and
\ref{PR.13}\ref{PR.13.a2}).
\end{enumerate}

\end{Remark}

\begin{proof}[Preliminaries to the proof]
Clearly, $\calC$ is non-empty and consists of convex sets containing the origin.
If $y \in x + t \cdot C_k$, $x \in X$, $t \in \R$, and $l \in \N$ is such that $x \in X_l$ then $y \in C_{\new{\max \left\{ k , l , \lceil \lno x \rno + |t| k \rceil \right\}}}$.
Thus, $\calC$ is a localizing family.
It is obviously a non-decreasing countable localizing family.
\par 
For each $k \in \N$ the set $B_k = X \cap \{x : \lno x \rno \leq k \}$ is $\calT$-closed since $\lno \cdot \rno$ is $\calT$-lower-semicontinuous and, therefore, $C_k = B_k \cap X_k$ is $\calT-$closed as well, since so is $X_k$.
\par 
Consequently, \ref{CCC.1} applies.
\end{proof}

\begin{proof}[Proof of \refmod{\ref{CCC.4.c1}}]
This is a reformulation of \ref{CCC.1}\refmod{\ref{CCC.1.c1}} in the present context.
\end{proof}

\begin{proof}[Proof of \refmod{\ref{CCC.4.c2}}]
This is a reformulation of \ref{CCC.1}\refmod{\ref{CCC.1.c2}} in the present context.
\end{proof}

\begin{proof}[Proof of \refmod{\ref{CCC.4.c3}}]
If $B \subset X$ is $\calT_\calC$-bounded then $\sup_{x \in B} \lno x \rno < \infty$, according to \refmod{\ref{CCC.4.c2}}.
This means that $\lno \cdot \rno$ is $\calT_\calC$-bounded. 
We next show it is also $\calT_\calC$-lower-semicontinuous, \ie that $F_t := X \cap \{ x : \lno x \rno \leq t \}$ is $\calT_\calC$-closed for every $t \in \R$. 
Since $\lno \cdot \rno$ is $\calT$-lower-semicontinuous, by hypothesis, $F_t$ is $\calT$-closed, therefore also $\calT_\calC$-closed, according to \ref{GO.1}\refmod{\ref{GO.1.1}}.
\end{proof}

Below, the difference with the hypotheses of \ref{CCC.4} is underlined.

\begin{Theorem}
\label{CCC.5}
Assume that
\begin{hypenum}
\item \label{CCC.5.h1} $X[\calT]$ is a locally convex topological vector space whose topology is generated by a \underline{filtering} family of seminorms $\la \| \cdot \|_i \ra_{i \in I}$;
\item \label{CCC.5.h2} $\lno \cdot \rno$ is a seminorm on $X$ that is $\calT$-lower-semicontinuous;
\item \label{CCC.5.h3} $\la X_k \ra_k$ is a non-decreasing sequence of $\calT$-closed vector subspaces whose union is $X$;
\item \label{CCC.5.h4} $C_k = X_k \cap \{ x : \lno x \rno \leq k \}$ for each $k \in \N$;
\item \label{CCC.5.h5} $\calC = \{ C_k : k \in \N \}$;
\end{hypenum}
Then for every linear function $f : X \to \R$ the following are equivalent.
\begin{enumerate}
\item[(a)] $f$ is $\calT_\calC$-continuous.
\item[(b)] $(\forall \veps > 0)(\exists i \in I)(\exists \delta > 0)(\forall x \in X_{\lceil \veps^{-1}\rceil}):$
\begin{equation*}
\left (\|x\|_i \leq \delta \text{ and } \lno x \rno \leq \veps^{-1} \right) \Rightarrow |f(x)| \leq \veps.
\end{equation*}
\item[(c)] $(\forall \veps > 0)(\exists i \in I)(\exists \theta > 0)(\forall x \in X_{\lceil \veps^{-1}\rceil}):$
\begin{equation*}
|f(x)| \leq \theta \cdot \|x\|_i + \veps \cdot \lno x \rno .
\end{equation*}
\end{enumerate}
\end{Theorem}

\begin{Remark}
The following obvious simplification occurs.
\begin{enumerate}
\item[(1)] If $\calT$ is defined by a unique norm $\|\cdot\|$ then each occurrence of ``$\exists i \in I$'' may be removed and each $\|\cdot\|_i$ replaced by $\|\cdot\|$.
\item[(2)] If $X_k = X$ for all $k \in \N$ then each occurrence of ``$\forall x \in X_{\lceil \veps^{-1}\rceil}$'' may be replaced by ``$\forall x \in X$''.
\end{enumerate}
\end{Remark}

\begin{proof}[Proof that $(a) \Rightarrow (b)$]
Let $\veps > 0$ and choose $k \in \N$ large enough for $\frac{1}{\veps} \leq k$.
Note that, being $\calT_\calC \hel C_k$-continuous, the function $f |_{C_k}$ is $\calT \hel C_k$-continuous, according to \ref{GO.1}\refmod{\ref{GO.1.2}}.
Thus, there exists a $\calT \hel C_k$-neighbourhood of the origin, say $W$, such that $|f(x)| \leq \veps$ whenever $x \in W$.
Furthermore, $W = U \cap C_k$ for some $U \in \calT$.
Since $0 \in U$ and $\la \|\cdot\|_i \ra_{i \in I}$ is filtering, there are $i \in I$ and $\delta > 0$ such that $X \cap \{ x : \|x\|_i \leq \delta \} \subset U$.
If $x \in X_{\lceil \veps^{-1} \rceil}$ and $\lno x \rno \leq \frac{1}{\veps}$ then $x \in C_k$.
If, moreover, $\|x\|_i \leq \delta$ then $x \in U$ as well.
In that case $x \in W$, whence, $|f(x)| \leq \veps$.
\end{proof}

\begin{proof}[Proof that $(b) \Rightarrow (c)$]
It clearly suffices to establish that (c) holds for all $0 < \veps \leq 1$.
With such $\veps$ associate $i$ and $\delta$ as in (b).
Define $\theta = \frac{\veps}{\delta}$.
Let $x \in X_{\lceil \veps^{-1}\rceil}$.
The remaining part of the proof is divided in three cases.
\par 
\textbf{(i)} Assume that $\lno x \rno = 0$.
\new{If, moreover, $\|x\|_i = 0$ then, for every $\lambda > 0$, $\lambda \cdot x$ satisfies the premises of (b), whence, $|f(x)| \leq \frac{\veps}{\lambda}$ and, in turn, $f(x)=0$; thus, the conclusion holds in that case and we may henceforth assume that $\|x\|_i \neq 0$.}
Define $\hat{x} = \frac{\delta \cdot x}{\|x\|_i}$.
Notice that $\hat{x} \in X_{\lceil \veps^{-1}\rceil}$, $\left\| \hat{x} \right\|_i = \delta$, and $\lno \hat{x} \rno = 0$.
Thus, by (b),
\begin{equation*}
\left| f(x) \right| = \frac{\|x\|_i}{\delta} \cdot \left| f \left( \hat{x} \right)\right| \leq \frac{\|x\|_i}{\delta} \cdot \veps = \theta \cdot \|x\|_i = \theta \cdot \|x\|_i + \veps \cdot \lno x \rno .
\end{equation*}
\par 
We henceforth assume that $\lno x \rno \neq 0$ and we abbreviate $\tilde{x} = \frac{x}{\lno x \rno}$.
Thus, $\lno \tilde{x} \rno = 1$.
\par 
\textbf{(ii)} Assume that $\left\| \tilde{x} \right\|_i > \delta$.
Define $\hat{x} = \frac{\delta \cdot \tilde{x}}{\left\|\tilde{x}\right\|_i}$.
Notice that $\hat{x} \in X_{\lceil \veps^{-1}\rceil}$, $\left\| \hat{x} \right\|_i = \delta$, and $\lno \hat{x} \rno = \frac{\delta}{\left\|\tilde{x}\right\|_i} < 1 \leq \frac{1}{\veps}$.
Thus, by (b),
\begin{equation*}
\left| f\left(\tilde{x} \right) \right| = \frac{\left\|\tilde{x}\right\|_i}{\delta} \cdot \left| f \left( \hat{x} \right)\right| \leq \frac{\left\|\tilde{x}\right\|_i}{\delta} \cdot \veps = \theta \cdot \left\|\tilde{x}\right\|_i ,
\end{equation*}
whence, $|f(x)| \leq \theta \cdot \|x\|_i \leq \theta \cdot \|x\|_i + \veps \cdot \lno x \rno$.
\par 
\textbf{(iii)} Assume that $\left\| \tilde{x} \right\|_i \leq \delta$.
In this case, (b) applies to $\tilde{x}$ so that $\left| f\left( \tilde{x} \right) \right| \leq \veps$.
Accordingly, $|f(x)| \leq \veps \cdot \lno x \rno \leq \theta \cdot \|x\|_i + \veps \cdot \lno x \rno$.
\end{proof}

\begin{proof}[Proof that $(c) \Rightarrow (a)$]
In view of the definition of localized topology \ref{EUL.3}(ii), the $\calT_{\new{\calC}}$-continuity of $f$ is equivalent to the $\calT \hel C_k$-continuity of each restriction $f|_{C_k} : C_k \to \R$.
Fix $k \in \N$.
Let $\la x_\lambda \ra_{\lambda \in \Lambda}$ be a net in $C_k$ that $\calT \hel C_k$-converges to $x \in C_k$.
Thus, it also $\calT$-converges to $x$, \ie $\la \|x-x_\lambda\|_i \ra_{\lambda \in \Lambda}$ converges to 0 in $\R$ for all $i \in I$.
Let $0 < \veps < 1$ and put $\hat{\veps} = \frac{\veps}{4k}$.
Let $\theta$ and $i$ be associated with $\hat{\veps}$ in condition (c).
Note that $x,x_\lambda \in C_k \subset X_k$, thus, $x-x_\lambda \in X_k \subset X_{\lceil \hat{\veps}^{-1}\rceil}$ for all $\lambda$.
Therefore, we infer from (c) and the linearity of $f$ that
\begin{equation*}
|f(x) - f(x_\lambda)| \leq \theta \cdot \|x-x_\lambda\|_i + \hat{\veps} \cdot \lno x-x_\lambda \rno
\leq \theta \cdot \|x-x_\lambda\|_i + \hat{\veps} \cdot 2k \leq \theta \cdot \|x-x_\lambda\|_i + \frac{\veps}{2}
\end{equation*}
for all $\lambda \in \Lambda$, since $\lno x -x_\lambda \rno \leq \lno x \rno + \lno x_\lambda \rno \leq 2k$.
Choose $\lambda_0 \in \Lambda$ such that $\|x-x_\lambda\|_i \leq \frac{\veps}{2\theta}$ whenever $\lambda \geq \lambda_0$.
For those $\lambda$ one has $|f(x) - f(x_\lambda)| \leq \veps$.
Since $\veps$ is arbitrarily small, the proof is complete\footnote{For the definition of net, convergence of nets, and the relation to continuity see \eg \cite[Ch.2 \S2 and Ch.3 \S 1 Theorem 1]{KELLEY}.}.
\end{proof}

Below, we omit the filtration of $X$ by $\la X_k \ra_k$ (see \ref{CCC.7} for the more general version) and the extra hypothesis with respect to \ref{CCC.4} is underlined.

\begin{Theorem}
\label{CCC.6}
Assume that
\begin{hypenum}
\item \label{CCC.6.h1} $X[\calT]$ is a locally convex topological vector space;
\item \label{CCC.6.h2} $\lno \cdot \rno$ is a seminorm on $X$ that is $\calT$-lower-semicontinuous;
\item \label{CCC.6.h3} $C_k = X \cap \{ x : \lno x \rno \leq k \}$ for each $k \in \N$;
\item \label{CCC.6.h4} \underline{$C_k[\calT \hel C_k]$ is first countable for each $k \in \N$};
\item \label{CCC.6.h5} $\calC = \{ C_k : k \in \N \}$.
\end{hypenum}
Consider the propositions:
\begin{enumerate}
\item[(1)] $X[\calT]$ is barrelled.
\item[(2)] $X[\calT_\calC]$ is barrelled.
\item[(3)] $\lno \cdot \rno$ is $\calT$-continuous.
\item[(4)] $\lno \cdot \rno$ is sequentially $\calT$-continuous.
\item[(5)] $\lno \cdot \rno$ is sequentially $\calT_\calC$-continuous.
\item[(6)] $\lno \cdot \rno$ is $\calT_\calC$-continuous.
\item[(7)] $X[\calT_\calC]$ is Fr\'echet-Urysohn.
\item[(8)] $X[\calT_\calC]$ is bornological.
\item[(9)] $\calT_\rmseq = \calT_\calC$.
\item[(10)] $\calT = \calT_\calC$.
\end{enumerate}
The following hold.
\begin{conclenum}
\item \label{CCC.6.c1} $(4) \iff (5) \iff (6) \iff (7) \iff (8)$.
\item \label{CCC.6.c2} $[(1) \vee (2) \vee (3)] \Rightarrow [(4) \wedge (5) \wedge (6) \wedge (7) \wedge (8)]$.
\item \label{CCC.6.c3} $(3) \Rightarrow (10)$.
\item \label{CCC.6.c4} If $\calT_\calC$ is sequential then $[(4) \vee (5) \vee (6) \vee (7) \vee (8)] \Rightarrow (9)$.
\item \label{CCC.6.c5} If both $\calT$ and $\calT_\calC$ are sequential then
\begin{enumerate}
\item[(a)] $(3) \iff (4) \iff (5) \iff (6) \iff (7) \iff (8)$.
\item[(b)] $(1) \iff (2)$.
\item[(c)] $[(1) \vee (2)] \Rightarrow [(3) \wedge (4) \wedge (5) \wedge (6) \wedge (7) \wedge (8)]$.
\item[(d)] $[(3) \vee (4) \vee (5) \vee (6) \vee (7) \vee (8)] \Rightarrow (10)$.
\end{enumerate}
\end{conclenum}
\end{Theorem}

\begin{Remark}
\label{CCC.6.R}
Some comments are in order. 
\begin{enumerate}
\item[\refmod{(H4)}] \refmod{Hypothesis \ref{CCC.6.h4} is only used in the proof that $(6) \Rightarrow (7)$ while the weaker assumption that each $C_k[\calT \hel C_k]$ be sequential is used in the proof that $(5) \Rightarrow (6)$.}
\item[\refmod{(C4,C5)}] The hypothesis in conclusions \refmod{\ref{CCC.6.c4}} and \refmod{\ref{CCC.6.c5}} that $X[\calT_\calC]$ be sequential is satisfied if one further assumes that each $C_k[\calT \hel C_k]$ is compact, see \ref{CWC.6}\refmod{\ref{CWC.6.c1}} (use \ref{PR.20}(E)).
\item[(Z)] All localized topologies considered in the applications will:
\begin{iiienum}
\item be sequential;
\item \textbf{not} be barrelled;
\item \textbf{not} be bornological;
\item \textbf{not} be Fr\'echet-Urysohn.
\end{iiienum}
Concretely, we will show for some $X[\calT_\calC]$ defined and used in Section \ref{sec.OE} that:
\begin{iiienum}[start=2]
\item There exists a sequence $\la F_j \ra_j$ in $X[\calT_\calC]^*$ that converges pointwise to a linear function $F : X \to \R$ but $F$ is not $\calT_\calC$-continuous, see the proof of \ref{OE.4}(C). In other words, the Banach-Steinhaus theorem and uniform boundedness principle do not hold in this context.
\item There exists a linear functional $F : X \to \R$ that sends $\calT_\calC$-bounded sets to bounded sets but is not $\calT_{\new{\calC}}$-continuous, see the proof of \ref{OE.4}(D).
\item[(iv)] There exists $A \subset X$ and a point in the $\calT_\calC$-closure of $A$ which is not the $\calT_\calC$-limit of any sequence in $A$, see the proof of \ref{OE.4}(B).
\end{iiienum}
\end{enumerate}
\end{Remark}

\begin{proof}[Proof that $(4) \Rightarrow (5)$]
If $\la x_j \ra_j$ is a $\calT_\calC$-null sequence then it is also $\calT$-null, by \ref{GO.1}\refmod{\ref{GO.1.1}}, whence $\lim_j \lno x_j \rno = 0$, by (4).
Accordingly, $\lno \cdot \rno$ is sequentially $\calT_\calC$-continuous at the origin.
\end{proof}

\begin{proof}[Proof that $(5) \Rightarrow (4)$]
Let $\la x_j\ra_j$ be a $\calT$-null sequence.
Fix $\veps > 0$ and define a sequence $\la \hat{x}_{j,\veps} \ra_j$ by the formula $\hat{x}_{j,\veps} = \frac{x_j}{\max \{ \veps, \lno x_j \rno \}}$.
As $\hat{x}_{j,\veps} = t_{j,\veps} \cdot x_j$ with $|t_{j,\veps}| \leq \frac{1}{\veps}$, $\la \hat{x}_{j,\veps} \ra_j$ is $\calT$-null as well.
Abbreviate $J_\veps = \N \cap \{ j : \lno x_j \rno > \veps \}$ and notice that 
\begin{equation}
\label{eq.CCC.1}
\lno \hat{x}_{j,\veps} \rno = \begin{cases}
1 & \text{if } j \in J_\veps \\
\frac{\lno x_j \rno}{\veps} & \text{if } j \not\in J_\veps.
\end{cases}
\end{equation}
In particular, $\lno \hat{x}_{j,\veps} \rno \leq 1$ for all $j \in \N$ so that $\la \hat{x}_{j,\veps} \ra_j$ is $\calT_\calC$-null, according to \ref{CCC.1}\refmod{\ref{CCC.1.c1}}.
Thus, $\lim_j \lno \hat{x}_{j,\veps} \rno = 0$, by (5).
If $j$ is so large that $\lno \hat{x}_{j,\veps} \rno < 1$ then $\lno x_j \rno \leq \veps$, by \eqref{eq.CCC.1}.
Since $\veps > 0$ is arbitrary, we have established that $\lno \cdot \rno$ is sequentially $\calT_\calC$-continuous at the origin.
\end{proof}

\begin{proof}[Proof that $(5) \Rightarrow (6)$]
Each $C_k[\calT \hel C_k]$ being first countable, is sequential.
Therefore, \ref{CCC.2}\refmod{\ref{CCC.2.c3}} applies. 
\end{proof}

\begin{proof}[Proof that $(6) \Rightarrow (5)$]
Trivial.
\end{proof}

\begin{proof}[Proof that $(6) \Rightarrow (7)$]
Let $A \subset X$ and $a \in \rmclos_{\calT_\calC}(A)$.
We ought to prove the existence of a sequence $\la x_j \ra_j$ in $A$ which converges to $\new{a}$ with respect to $\calT_\calC$.
For each $j \in \N$ define $G_j$ by the formula
\begin{equation*}
G_j = X \cap \left\{ x : \lno x -a \rno < \frac{1}{j} \right\} \subset X \cap \left\{ x : \lno x -a \rno < 1 \right\} \subset X \cap \left\{ x : \lno x \rno \leq k \right\} = C_k
\end{equation*}
where $k$ is chosen larger than $1 + \lno a \rno$.
Observe that $G_j \in \calT_\calC$, according to (6).
Since $\calT \hel C_k$ is first countable, there exists a countable \new{non-increasing} basis of $\calT \hel C_k$-open neighbourhoods of $a$, say $\la W_j \ra_j$.
Thus, $W_j = U_j \cap C_k$ for some $U_j \in \calT$.
Notice that $a \in G_j \cap U_j$ and that $G_j \cap U_j \in \calT_\calC$, by \ref{GO.1}\refmod{\ref{GO.1.1}}.
Since $a \in \rmclos_{\calT_\calC}(A)$, there exists $x_j \in A \cap G_j \cap U_j$, for all $j \in \N$.
Clearly, $\la x_j \ra_j$ is a sequence in $A$ and in $C_k$.
As $G_j \cap U_j \subset C_k \cap U_j = W_j$, we infer that $\la x_j \ra_j$ is $\calT \hel C_k$-convergent to $a$, hence also $\calT$-convergent to $a$.
It follows from \ref{CCC.1}\refmod{\ref{CCC.1.c1}} that $\la x_j\ra_j$ is $\calT_\calC$-convergent to $a$.
\end{proof}

\begin{proof}[Proof that $(7) \Rightarrow (8)$]
This is \ref{PR.6}.
\end{proof}

\begin{proof}[Proof that $(8) \Rightarrow (6)$]
The seminorm $\lno \cdot \rno$ is $\calT_\calC$-bounded, according to \ref{CCC.4}\refmod{\ref{CCC.4.c3}}.
Therefore, it is $\calT_\calC$-continuous, by \ref{PR.13}\refmod{\ref{PR.13.a2}}, since $X[\calT_\calC]$ is assumed to be bornological.
\end{proof}

\begin{proof}[Proof of \refmod{\ref{CCC.6.c2}}]
Since the seminorm $\lno \cdot \rno$ is $\calT$-lower-semicontinuous, by assumption, we note that $(1) \Rightarrow (3)$, according to \ref{PR.8}\refmod{\ref{PR.8.a2}}.
By the same argument, $(2) \Rightarrow (6)$, since $\lno \cdot\rno$ is $\calT_\calC$-lower-semicontinuous, according to \ref{CCC.4}\refmod{\ref{CCC.4.c3}}.
Trivially, $(3) \Rightarrow (4)$.
The conclusion now follows from \refmod{\ref{CCC.6.c1}}.
\end{proof}

\begin{proof}[Proof of \refmod{\ref{CCC.6.c3}}]
In view of \ref{GO.1}\refmod{\ref{GO.1.1}}, it is sufficient to show that $\calT_\calC \subset \calT$.
Let $E \in \calT_\calC$ and $x \in E$.
Recalling \ref{EUL.6} we find a convex set $V$ such that $x \in V \subset E$ and $V \cap C_k \in \calT \hel C_k$ for all $k$.
Choose $k$ so that $\lno x \rno < k$.
Define $O_k = X \cap \{ y : \lno y \rno < k\} \subset C_k$.
By (3), $O_k \in \calT$.
There exists $U_k \in \calT$ such that $V \cap C_k = U_k \cap C_k$.
Notice that $x \in U_k \cap O_k \subset U_k \cap C_k = V \cap C_k \subset V \subset E$.
Since $U_k \cap O_k \in \calT$, the set $E$ is a $\calT$-neighbourhood of $x$.
As $x$ is arbitrary, $E \in \calT$.
\end{proof}

\begin{proof}[Proof of \refmod{\ref{CCC.6.c4}}]
In view of \refmod{\ref{CCC.6.c1}}, conclusion \refmod{\ref{CCC.6.c4}} will be established upon showing $(5) \Rightarrow (9)$.
Let $E \in \calT_\calC$, $x \in E$, and $\la x_j \ra_j$ a sequence which converges to $x$ with respect to $\calT$.
It ensues from the argument in the proof of $(5) \Rightarrow (4)$ that $\la x_j \ra_j$ is also $\calT_\calC$-convergent to $x$.
Thus, $\la x_j \ra_j$ is eventually in $E$.
Since $x$ and $\la x_j \ra_j$ are arbitrary, this shows that $E$ is sequentially $\calT$-open.
From the arbitrariness of $E$ we infer that $\calT_\calC \subset \calT_\rmseq$.
Since $\calT \subset \calT_\calC$, by \ref{GO.1}\refmod{\ref{GO.1.1}}, it follows that $\calT_\rmseq \subset \left(\calT_\calC \right)_\rmseq = \calT_\calC$, where the last equality is the hypothesis that $\calT_\calC$ be sequential.
Finally, $\calT_\calC \subset \calT_\rmseq \subset \calT_\calC$.
\end{proof}

\begin{proof}[Proof of \refmod{\ref{CCC.6.c5}}(a)]
We already observed that $(3) \Rightarrow (4)$.
The reverse implication follows \new{from} the fact that $\calT$ is sequential, \ref{PR.2}\new{(E)}.
The other equivalences have been proved in \refmod{\ref{CCC.6.c1}}.
\end{proof}

\begin{proof}[Proof of \refmod{\ref{CCC.6.c5}}(b)]
Notice that $(1) \Rightarrow (3)$ (recall the proof of \refmod{\ref{CCC.6.c2}}) and, by \refmod{\ref{CCC.6.c3}}, $(3) \Rightarrow (10)$.
Furthermore, $[(1) \wedge (10)] \Rightarrow (2)$.
This shows that $(1) \Rightarrow (2)$.
With regard to the reverse implication, we note that $(2) \Rightarrow (9)$, by \refmod{\ref{CCC.6.c2}} and \refmod{\ref{CCC.6.c4}}.
Furthermore, $(9) \Rightarrow (10)$, since $\calT$ is assumed to be sequential.
Trivially, $[(2) \wedge (10)] \Rightarrow (1)$.
\end{proof}

\begin{proof}[Proof of \refmod{\ref{CCC.6.c5}}(c)]
This follows from \refmod{\ref{CCC.6.c2}} and \refmod{\ref{CCC.6.c5}}(a).
\end{proof}

\begin{proof}[Proof of \refmod{\ref{CCC.6.c5}}(d)]
This follows from \refmod{\ref{CCC.6.c3}} and \refmod{\ref{CCC.6.c5}}(a).
\end{proof}

Below, the difference with the hypotheses of \ref{CCC.4} is underlined.

\begin{Theorem}
\label{CCC.7}
Assume that
\begin{hypenum}
\item \label{CCC.7.h1} $X[\calT]$ is a locally convex topological vector space;
\item \label{CCC.7.h2} $\lno \cdot \rno$ is a seminorm on $X$ that is $\calT$-lower-semicontinuous;
\item \label{CCC.7.h3} $\la X_k \ra_k$ is a non-decreasing sequence of $\calT$-closed vector subspaces whose union is $X$;
\item \label{CCC.7.h4} $C_k = X_k \cap \{ x : \lno x \rno \leq k \}$ for each $k \in \N$;
\item \label{CCC.7.h5} \underline{$\calT \hel C_k$ is first countable for each $k \in \N$};
\item \label{CCC.7.h6} $\calC = \{ C_k : k \in \N \}$.
\end{hypenum}
For fixed $k \in \N$ consider the following propositions.
\begin{enumerate}
\item[(1)] $X_k[\calT \hel X_k]$ is barrelled.
\item[(2)] $X_k[(\calT \hel X_k)_{\calC \shel X_k}]$ is barrelled.
\item[(3)] The restriction of $\lno \cdot \rno$ to $X_k$ is $\calT \hel X_k$-continuous.
\item[(4)] The restriction of $\lno \cdot \rno$ to $X_k$ is sequentially $\calT \hel X_k$-continuous.
\item[(5)] The restriction of $\lno \cdot \rno$ to $X_k$ is sequentially $(\calT \hel X_k)_{\calC \shel X_k}$-continuous.
\item[(6)] The restriction of $\lno \cdot \rno$ to $X_k$ is $(\calT \hel X_k)_{\calC \shel X_k}$-continuous.
\item[(7)] $X_k[(\calT \hel X_k)_{\calC \shel X_k}]$ is Fr\'echet-Urysohn.
\item[(8)] $X_k[(\calT \hel X_k)_{\calC \shel X_k}\new{]}$ is bornological.
\item[(9)] $(\calT \hel X_k)_\rmseq = (\calT \hel X_k)_{\calC \shel X_k}$.
\item[(10)] $\calT \hel X_k = (\calT \hel X_k)_{\calC \shel X_k}$.
\end{enumerate}
For every $k \in \N$ the following hold.
\begin{conclenum}
\item \label{CCC.7.c1} $(4) \iff (5) \iff (6) \iff (7) \iff (8)$.
\item \label{CCC.7.c2} $[(1) \vee (2) \vee (3)] \Rightarrow [(4) \wedge (5) \wedge (6) \wedge (7) \wedge (8)]$.
\item \label{CCC.7.c3} $(3) \Rightarrow (10)$.
\item \label{CCC.7.c4} If $\calT_\calC$ is sequential then $[(4) \vee (5) \vee (6) \vee (7) \vee (8)] \Rightarrow (9)$.
\item \label{CCC.7.c5} If both $\calT$ and $\calT_\calC$ are sequential then
\begin{enumerate}
\item[(a)] $(3) \iff (4) \iff (5) \iff (6) \iff (7) \iff (8)$.
\item[(b)] $(1) \iff (2)$.
\item[(c)] $[(1) \vee (2)] \Rightarrow [(3) \wedge (4) \wedge (5) \wedge (6) \wedge (7) \wedge (8)]$.
\item[(d)] $[(3) \vee (4) \vee (5) \vee (6) \vee (7) \vee (8)] \Rightarrow (10)$.
\end{enumerate}
\end{conclenum}
\end{Theorem}

\begin{Remark}
\label{CCC.7.R}
The theorem would be stronger if each occurrence of $(\calT \hel X_k)_{\calC \hel X_k}$ were replaced by $\calT_\calC \hel X_k$.
However, at this level of generality I do not know whether these two topologies coincide.
Notwithstanding, if each $C_k[\calT \hel C_k]$ is compact then they are equal \ref{CWC.6}\refmod{\ref{CWC.6.c2}}.
The latter is what happens in the applications.
\end{Remark}

\begin{proof}
Abbreviate $B_l = X \cap \{ x : \lno x \rno \leq l \}$, $l \in \N$.
The proof consists in applying \ref{CCC.6} to $X_k$ (in place of $X$), $\calT \hel X_k$ (in place of $\calT$), the restriction of $\lno \cdot \rno$ to $X_k$ (in place of $\lno \cdot \rno$), and $\calD_k = \{ D_l : l \in \N \}$ (in place of $\calC$) where $D_l = X_k \cap B_l$.
Most obviously, $X_k[\calT \hel X_k]$ is a locally convex topological vector space and the restriction of $\lno \cdot \rno$ to $X_k$ is $\calT \hel X_k$-lower-semicontinuous.
\par 
Fix $\new{l} \in \N$.
If $l < k$ then $D_l = X_k \cap B_l \subset X_k \cap B_k = C_k$, since $B_l \subset B_k$.
If $l \geq k$ then $D_l = X_k \cap B_l \subset X_l \cap B_l = C_l$, since $\new{X_k \subset X_l}$.
In both cases, $D_l[\calT \hel D_l]$ being a topological subspace of a first countable space is itself first countable.
\par 
To complete the proof it remains to establish that $(\calT \hel X_k)_{\calC \shel X_k} = (\calT \hel X_k)_{\calD_k}$.
In view of \ref{GO.1}\refmod{\ref{GO.1.3}} it is sufficient to show that $\calC \hel X_k$ and $\calD_k$ are interlaced.
The previous paragraph implies that each member of $\calD_k$ is contained in a member of $\calC \hel X_k$.
We ought to show that each $C \in \calC \hel X_k$ is contained in some member of $\calD_k$.
There exists $l \in \N$ such that $C = X_k \cap C_l$.
If $l < k$ then $C = X_k \cap C_l \subset X_k \cap C_k = X_k \cap B_k = D_k \in \calD_k$.
If $l \geq k$ then $C = X_k \cap C_l = X_k \cap X_l \cap B_l = X_k \cap B_l = D_l \in \calD_k$.
\par
\new{Finally, we verify that the hypotheses of \ref{CCC.6}\refmod{\ref{CCC.6.c4}} and \refmod{\ref{CCC.6.c5}} are satisfied in the present setting. If $\calT_\calC$ is sequential then so is $\calT_\calC \hel X_k$, according to \ref{PR.2}(C)(a), since $X_k$ is $\calT_\calC$-closed, by \ref{GO.1}\refmod{\ref{GO.1.1}}; whence, $(\calT \hel X_k)_{\calC \shel X_k} = \calT_\calC \hel X_k$ is sequential, by \ref{CCC.1}\refmod{\ref{CCC.1.c3}}. Similarly, if $\calT$ is sequential then so is $\calT \hel X_k$, by \ref{PR.2}(C)(a).}
\end{proof}

\section{Countable localizing families of compact convex sets}
\label{sec.CWC}

\begin{Scholia}
\label{CWC.1}
Assume that $X$ is a locally convex topological vector space, $K \subset U \subset X$, $K$ is convex and compact and $U$ is open. There then exists a convex open set $V$ such that $K \subset V$ and $\rmclos V \subset U$.
\end{Scholia}

\begin{proof}
According to \cite[1.10]{RUDIN} applied with $C = X \setminus U$, there exists a neighborhood $W$ of 0 such that $(K + W) \cap (C + W) = \emptyset$.
One can choose $W$ open and convex (since $X$ is locally convex).
Thus, $V = K + W$ is open and convex.
As $V \cap (C + W) = \emptyset$, it is clear that $\rmclos V \subset U$.
\end{proof}

\begin{Theorem}
\label{CWC.2}
Assume that:
\begin{hypenum}
\item \label{CWC.2.h1} $X[\calT]$ is a locally convex topological vector space;
\item \label{CWC.2.h2} $\calC$ is a non-decreasing countable localizing family in $X$;
\item \label{CWC.2.h3} Each member of $\calC$ is $\calT$-compact.
\end{hypenum}
The following hold.
\begin{conclenum}
\item \label{CWC.2.c1} The localization of $\calT$ by $\calC$ consists of:
\begin{equation*}
\calT_\calC = \calP(X) \cap \big\{ E : E \cap C \in \calT \hel C \text{ for every } C \in \calC \big\}.
\end{equation*}
\item \label{CWC.2.c2} For every set $E \subset X$:
\begin{enumerate}
\item[(a)] $E$ is $\calT_\calC$-open if and only if $E \cap C$ is $\calT \hel C$-open for every $C \in \calC$;
\item[(b)] $E$ is $\calT_\calC$-closed if and only if $E \cap C$ is $\calT \hel C$-closed for every $C \in \calC$.
\end{enumerate}
\item \label{CWC.2.c3} A set $B \subset X$ is $\calT_\calC$-bounded if and only it is contained in some $C \in \calC$.
\item \label{CWC.2.c4} $X[\calT_\calC]$ is sequentially complete.
\end{conclenum}
\end{Theorem}

\begin{Remark}
\label{CWC.3}
Property \ref{CWC.2}\refmod{\ref{CWC.2.c1}} is in effect saying that $X[\calT_\calC]$ is the inductive limit of the topological spaces $C[\calT \hel C]$, $C \in \calC$, in the category of topological spaces (these are called ``weak topologies'' in \cite[Chap. VI \S 8]{DUGUNDJI}) and it allows to dispense with worrying about convex sets at all. 
\end{Remark}

\begin{proof}[Proof of \refmod{\ref{CWC.2.c1}}]
We abbreviate
\begin{equation*}
\calE = \calP(X) \cap \big\{ E : E \cap C \in \calT \hel C \text{ for every } C \in \calC \big\} 
\end{equation*}
and we recall that $\calT_\calC$ is described in \ref{EUL.6}.
If $E \in \calT_\calC$ then for all $C \in \calC$ one clearly has $E \cap C \in \calT_\calC \hel C$, hence, $E \cap C \in  \calT \hel C$, by \ref{GO.1}\refmod{\ref{GO.1.2}}.
Thus, $E \in \calE$.
Since $E$ is arbitrary, we have shown that $\calT_\calC \subset \new{\calE}$.
The remainder of the argument consists in proving the reverse inclusion $\calE \subset \calT_\calC$.
As in the previous section let $\calC = \{ C_k : k \in \N\}$ be numbered so that $C_k \subset C_{k+1}$, $k \in \N$.
\par 
Let $E \in \calE$.
In view of \ref{EUL.6} we ought to show that for every $x \in E$ there exists a convex set $V$ such that $x \in V \subset E$ and $V \cap C \in \calT \hel C$ for all $C \in \calC$.
First note that, by definition of $\calE$, there are $U_k \in \calT$ such that $E \cap C_k = U_k \cap C_k$ for all $k \in \N$.
\par 
Fix $x \in E$ and let $k_0$ be large enough for $x \in C_{k_0}$.
Notice that $x \in U_{k_0}$.
Applying the scholia to $K = \{x\}$ and $U = U_{k_0}$ we find a convex set $V_{k_0} \in \calT$ such that $\{x\} \subset V_{k_0} \subset \rmclos_{\calT} V_{k_0} \subset U_{k_0}$.
We claim that there exists a sequence $\la V_k \ra_{k \geq k_0}$ of convex $\calT$-open sets such that $V_{k_0}$ is as above and 
\begin{equation}
\label{eq.CWC.2}
C_k \cap \rmclos_\calT V_k \subset V_{k+1} \subset \rmclos_\calT V_{k+1} \subset U_{k+1}
\end{equation}
for all $k \geq k_0$.
We proceed inductively with respect to $k$.
Assuming that $V_{k_0},\ldots,V_k$ have been defined we note that $\rmclos_\calT V_k \subset U_k$ (by the choice of $V_{k_0}$ if $k=k_0$ and by \eqref{eq.CWC.2} if $k > k_0$).
Accordingly, $C_k \cap \rmclos_\calT V_k \subset C_k \cap U_k = C_k \cap E \subset C_{k+1} \cap E = U_{k+1} \cap E \subset U_{k+1}$.
Therefore, the scholia applied to $K = C_k \cap \rmclos_\calT V_k$ and $U=U_{k+1}$ yields a convex open set $V_{k+1}$ satisfying \eqref{eq.CWC.2}.
\par
We claim that the sequence $\la C_k \cap V_{k} \ra_{k \geq k_0}$ is non-decreasing. 
Indeed, $C_k \cap V_k \subset C_k \subset C_{k+1}$ and, according to \new{\eqref{eq.CWC.2}}, $C_k \cap V_k \subset C_k \cap \rmclos_\calT V_k \subset V_{k+1}$ for all $k \geq k_0$.
Since each $C_k \cap V_k$ is convex, so is
\begin{equation*}
V = \bigcup_{k \geq k_0} C_k \cap V_k.
\end{equation*}
Furthermore, $x \in V$, since $x \in C_{k_0} \cap V_{k_0}$.
Observe also that $C_k \cap V_k \subset C_k \cap \rmclos_\calT V_k \subset C_k \cap U_k = C_k \cap E \subset E$ for all $k \geq k_0$.
Whence, $x \in V \subset E$.
\par 
We now show that $V \cap \new{C_j} \in \calT \hel C_j$ for all $j \in \N$.
Fix $j \in \N$.
Since $\la V_k \cap C_k \ra_{k \geq k_0}$ is non-decreasing, we have
\begin{equation*}
V = \bigcup_{k \geq \max\{j,k_0\}} C_k \cap V_k.
\end{equation*}
As, for each $k \geq \max\{j,k_0\}$ we have $C_j \cap C_k = C_j$, we conclude that
\begin{equation*}
V \cap C_j = C_j \cap \left( \bigcup_{k=\max\{j,k_0\}}^\infty  V_{k} \right) \in \calT \hel C_j  \,,
\end{equation*}
because each $V_k$ is $\calT$-open.
By the arbitrariness of $x$, this shows that $E \in \calT_\calC$.
As $E$ is arbitrary, we have shown that $\calE \subset \calT_\calC$.
\end{proof}

\begin{proof}[Proof of \refmod{\ref{CWC.2.c2}}]
Notice that (a) is a rephrasing of \refmod{\ref{CWC.2.c1}}.
In order to prove (b), let $E \subset X$ and observe that
\begin{equation*}
\begin{split}
(\forall C \in \calC) : E \cap C \text{ is } \calT \hel C \text{-closed } & \iff (\forall C \in \calC) : C \setminus (E \cap C)  \text{ is } \calT \hel C \text{-open } \\
& \iff (\forall C \in \calC) : C \cap  (X \setminus E)  \text{ is } \calT \hel C \text{-open }\\
& \iff X \setminus E \text{ is } \calT_\calC \text{-open (by (a))} \\
& \iff E \text{ is } \calT_\calC \text{-closed} \,.\\
\end{split}
\end{equation*} 
\end{proof}

\begin{proof}[Proof of \refmod{\ref{CWC.2.c3}}]
Since $\calT$-compact sets are $\calT$-bounded and subsets of $\calT$-bounded sets are $\calT$-bounded as well, this follows from \ref{CCC.1}\refmod{\ref{CCC.1.c2}}.
\end{proof}

\begin{proof}[Proof of \refmod{\ref{CWC.2.c4}}]
Let $\la x_j \ra_j$ be a $\calT_\calC$-Cauchy sequence in $X$, \ie for every $\calT_\calC$-neighbour\-hood of the origin $E$ there exists $j_0 \in \N$ such that $x_{j} - x_{k} \in E$ whenever $\min\{j,k\} \geq j_0$. 
As $\la x_j \ra_j$ is readily bounded, it follows from \refmod{\ref{CWC.2.c3}} that $\{x_j : j \in \N \} \subset C$ for some $C \in \calC$. 
Since $C$ is $\calT$-compact, it also is $\calT_\calC$-compact, according to \ref{GO.1}\refmod{\ref{GO.1.6}}, hence, countably $\calT_\calC$-compact and, therefore, $\la x_j \ra _j$ has a $\calT_\calC$-accumulation point $x \in C$, \cite[Chap. XI, 3.2]{DUGUNDJI}. 
Let $U$ be a $\calT_\calC$-neighbourhood of $x$, define $V = U-x$, and choose a $\calT_\calC$-neighbourhood of the origin $W$ such that $W + W \subset V$. 
Let $j_0 \in \N$ be such that $x_j - x_k \in W$ whenever $\min\{j,k\} \geq j_0$. 
Let also $k \in \N$ be such that $k \geq j_0$ and $x_k \in x + W$. 
Then for every $j \geq j_0$ one has $x_j - x = (x_k-x) + (x_j-x_k) \in W + W \subset V$, \ie $x_j \in U$.
\end{proof}

\begin{Theorem}
\label{CWC.6}
Assume that:
\begin{hypenum}
\item \label{CWC.6.h1} $X[\calT]$ is a locally convex topological vector space;
\item \label{CWC.6.h2} $\calC$ is a non-decreasing countable localizing family in $X$;
\item \label{CWC.6.h3} Each member of $\calC$ is $\calT$-compact;
\item \label{CWC.6.h4} For each $C \in \calC$ the topology $\calT \hel C$ is sequential.
\end{hypenum}
The following hold.
\begin{conclenum}
\item \label{CWC.6.c1} $X[\calT_\calC]$ is sequential.
\item \label{CWC.6.c2} If $Y \subset X$ is a $\calT_\calC$-closed vec\-tor subspace (in particular, if $Y$ is $\calT$-closed) then $\calT_\calC \hel Y = (\calT \hel Y)_{\calC \shel Y}$.
\end{conclenum}
\end{Theorem}

\begin{proof}[Proof of \refmod{\ref{CWC.6.c1}}]
Let $E \subset X$ be sequentially $\calT_\calC$-closed.
We ought to show that $E$ is $\calT_\calC$-closed which, according to \ref{CWC.2}\refmod{\ref{CWC.2.c2}}(b), amounts to showing that $E \cap C$ is $\calT \hel C$-closed for all $C \in \calC$.
Fix $C \in \calC$. 
Since the topology $\calT \hel C$ is sequential, it is sufficient to prove that $E\cap C$ is sequentially $\calT \hel C$-closed.
Let $\la x_j \ra_j$ be a sequence in $E \cap C$ that is $\calT \hel C$-convergent to some $x \in C$.
As $\la x_j \ra_j$ is also $\calT$-convergent to $x$, it follows from \ref{CCC.1}\refmod{\ref{CCC.1.c1}} that it is $\calT_\calC$-convergent to $x$.
Thus, $x \in E$, since $E$ is sequentially $\calT_\calC$-closed.
Accordingly, $x \in E \cap C$.
Since $\la x_j \ra_j$ and $x$ are arbitrary, $E \cap C$ is sequentially $\calT \hel C$-closed.
\end{proof}

\begin{proof}[Proof of \refmod{\ref{CWC.6.c2}}]
In view of \ref{CCC.1}\refmod{\ref{CCC.1.c3}} it is sufficient to show that $\calT_\calC \hel Y$ is sequential.
As $X[\calT_\calC]$ is sequential, by \refmod{\ref{CWC.6.c1}}, and $Y$ is $\calT_\calC$-closed, by hypothesis, this follows from \ref{PR.2}(C)(a).
\end{proof}

\section{First dual, dense subspaces, and extensions}
\label{sec.FD}

\begin{Empty}
\label{FD.0}
Let $X[\calT]$ be a locally convex topological vector space and $\calC$ a localizing family in $X$.
Here, we consider the following condition.
\begin{enumerate}
\item[($\diamond$)] \textit{ A subset of $X$ is $\calT_\calC$-bounded if and only if it is contained in some $C \in \calC$.}
\end{enumerate}
We identify circumstances that guarantee this condition ($\diamond$) is satisfied.
\begin{enumerate}
\item[(A)] \textit{ If $\calC$ is a non-decreasing countable family of $\calT$-compact sets then }($\diamond$) \textit{ is satisfied.}
\end{enumerate}
\par
{\it Proof.}
This is an immediate consequence of \ref{CWC.2}\refmod{\ref{CWC.2.c3}}.\cqfd
\begin{enumerate}
\item[(B)] \textit{ Assume that $X[\calT]$ and $\calC$ satisfy the hypotheses of \ref{CCC.4} and assume, furthermore, that for every $i \in I$ and every $k \in \N$ there exists $\Gamma_{i,k} > 0$ such that $\|x\|_i \leq \Gamma_{i,k} \cdot \lno x \rno$ for all $x \in X_k$. Then }($\diamond$) \textit{ is satisfied.}
\end{enumerate}
\par 
{\it Proof.}
In view \new{of} \ref{CCC.1}\refmod{\ref{CCC.1.c2}} it suffices to show that each subset of some $C_k$ is $\calT$-bounded.
If $B \subset C_k$ then $\sup \{ \|x\|_i : x \in B \} \leq k \cdot \Gamma_{i,k}$ for all $i$.
This completes the proof.\cqfd
\end{Empty}

\begin{Empty}
\label{FD.1}
We will be considering the vector space $X[\calT_\calC]^*$ consisting of all $\calT_\calC$-continuous linear functions $f : X \to \R$.
As the dual of a locally convex topological vector space it is itself traditionally equipped with the strong topology $\beta(X[\calT_\calC]^*,X)$ (we will denote it briefly by $\beta(X^*,X)$ when no confusion can arise), \ie the topology of uniform convergence on $\calT_\calC$-bounded subsets of $X$.
This topology is generated by the family of seminorms $p_B$, corresponding to all $\calT_\calC$-bounded subsets $B$ of $X$, where $p_B(f) = \sup \{ |f(x)| : x \in B \}$.
Note that $\calT_\calC$-continuous linear functions $X \to \R$ are $\calT_\calC$-bounded \cite[1.32]{RUDIN} but that $\calT_\calC$-bounded linear functions $X \to \R$ may not be $\calT_\calC$-continuous, as $X[\calT_\calC]$ is not bornological in most interesting cases.
\end{Empty}

\begin{Theorem}
\label{FD.2}
Assume that:
\begin{hypenum}
\item \label{FD.2.h1} $X[\calT]$ \new{is} a locally convex topological vector space;
\item \label{FD.2.h2} $\calC$ is a non-decreasing countable localizing family in $X$;
\item \label{FD.2.h3} Each member of $\calC$ is $\calT$-closed;
\item \label{FD.2.h4} A subset of $X$ is $\calT_\calC$-bounded if and only if it is contained in some $C \in \calC$.
\end{hypenum}
Then $X[\calT_\calC]^*[\beta(X^*,X)]$ is a Fr\'echet space.
\end{Theorem}

\begin{proof}
It follows from \refmod{hypothesis \ref{FD.2.h4}} that $\calC$ is a fundamental system of $\calT_\calC$-bounded sets. 
Therefore, the strong topology $\beta(X^*,X)$ of $X[\calT_\calC]^*$ is generated by the countable family of seminorms $\la p_{C} \ra_{C \in \calC}$, \ie is defined by a translation invariant metric.
It remains to show that it is complete.
\par 
Let $\la f_n \ra_n$ be a Cauchy sequence in $X[\calT_\calC]^*$ with respect to $\beta(X^*,X)$.
Then it is a Cauchy sequence with respect to each seminorm $p_C$ and, in particular, each sequence $\la f_n(x) \ra_n$, $x \in X$, is Cauchy in $\R$.
Consequently, a linear function $f: X \to \R$ is well-defined by the formula $f(x) = \lim_n f_n(x)$.
A routine argument shows that $\lim_n \sup \{ |f(x) - f_n(x)| : x \in C \} = 0$.
\par 
It remains to show that $f$ is $\calT_\calC$-continuous.
According to the definition of localized topology \ref{EUL.3}, it is sufficient to show that $f|_C$ is $\calT \hel C$-continuous for all $C \in \calC$.
Note that each $(f_n)|_C$ is $\calT \hel C$-continuous, according to its $\calT_\calC \hel C$-continuity and \ref{GO.1}\refmod{\ref{GO.1.2}}.
Since $\la (f_n)|_C \ra_n$ converges uniformly to $f|_C$, the conclusion follows.
\end{proof}

\begin{Theorem}
\label{FD.3}
Assume that:
\begin{hypenum}
\item \label{FD.3.h1} $X[\calT]$ is a locally convex topological vector space;
\item \label{FD.3.h2} $\calC$ is a non-decreasing countable localizing family in $X$;
\item \label{FD.3.h3} Each member of $\calC$ is $\calT$-closed;
\item \label{FD.3.h4} A subset of $X$ is $\calT_\calC$-bounded if and only if it is contained in some $C \in \calC$;
\item \label{FD.3.h5} There exists $C_0 \in \calC$ such that for each $C \in \calC$ there exists $t > 0$ such that $C \subset t \cdot C_0$.
\end{hypenum}
Then $X[\calT_\calC]^*[\beta(X^*,X)]$ is a Banach space.
\end{Theorem}

\begin{proof}
We know that $X[\calT_\calC]^*[\beta(X^*,X)]$ is a Fr\'echet space, by \ref{FD.2}, and \refmod{hypothesis \ref{FD.3.h5}} guarantees that the topology $\beta(X^*,X)$ is generated by the single norm $p_{C_0}$.
\end{proof}

\begin{Empty}
\label{FD.4}
Assume that the hypotheses of \ref{CCC.5} are satisfied, $Y \subset X$ is a vector subspace, and define $Y_k = X_k \cap Y$, $k \in \N$.
Assume also that $f : Y \to \R$ is linear and satisfies the following condition: 
\begin{equation*}
(\forall \veps > 0)(\exists i \in I)(\exists \theta > 0)\left(\forall y \in Y_{\lceil \veps^{-1} \rceil}\right) : \quad |f(y)| \leq \theta \cdot \|y\|_i + \veps \cdot \lno y \rno .
\end{equation*}
One may wonder whether $f$ extends to a $\calT_\calC$-continuous linear function defined on $X$.
One is tempted to argue that $f$, being continuous with respect to the localized topology, it  extends, by Hahn-Banach, since localized topologies are locally convex.
However, this should be subjected to careful scrutiny:
\begin{enumerate}
\item[(1)] \new{First, \ref{CCC.5} applies to $Y[\calT \hel Y]$ (in place of $X[\calT]$), the restriction of $\lno \cdot \rno$ to $Y$, and $\la Y_k \ra_k$ (in place of $\la X_k \ra_k$) -- one need only observe that each $Y_k$ is $\calT \hel Y$-closed in $Y$. Accordingly, $f$ is $(\calT \hel Y)_{\calC \shel Y}$-continuous.}
\item[(2)] In order to apply Hahn-Banach's theorem to extend $f$ we would need to know that $f$ is $\calT_\calC \hel Y$-continuous.
However, for all we know $f$ is $(\calT \hel Y)_{\calC \shel Y}$-continuous and the inclusion \ref{GO.1}\refmod{\ref{GO.1.7}} goes in the wrong direction.
We would need instead to be able to apply \ref{CCC.1}\refmod{\ref{CCC.1.c3}}.
This requires $\calT_\calC \hel Y$ to be sequential which is the case when $Y$ is $\calT_\calC$-closed (recall \new{\ref{CWC.6}\refmod{\ref{CWC.6.c1}}} and \ref{PR.2}\new{(C)}(a)) but what if $Y$ is $\calT_\calC$-dense?
%
%
\item[(3)] Thus, if $Y$ is $\calT_\calC$-dense then we do not know that $f$ is $\calT_\calC \hel Y$-continuous but we could try to approximate any $x \in X \setminus Y$ by the terms of a $\calT_\calC$-convergent sequence in $Y$ and use the bounds we assume on $f$ to define $f(x)$.
However, for all we know $Y$ is not sequentially $\calT_\calC$-dense, since $X[\calT_\calC]$ is not Fr\'echet-Urysohn in general \ref{CCC.7}.
\end{enumerate}
As we cannot avoid the above issues in general, we introduce the following favourable situation (that fortunately arises in applications).
\end{Empty}

\begin{Empty}
\label{FD.5}
Assume that the hypotheses of \ref{CCC.5} are satisfied and let $Y \subset X$ be a vector subspace.
We say that $Y$ is {\em uniformly sequentially $\calT_\calC$-dense} in $X$ if 
\begin{multline*}
\left(\exists \gamma \in \N^\N\right)(\exists \Gamma > 0)(\forall k \in \N)(\forall x \in X_k)(\exists \la y_j \ra_j \text{ in } \new{Y \cap X_{\gamma(k)}}):
\\
y_j \xrightarrow{\calT} x \text{ and } \sup_j \lno y_j \rno \leq \Gamma \cdot \lno x \rno .
\end{multline*}
\par 
Notice that if $\la y_j \ra_j$ is associated with $x$ as above then it $\calT_\calC$-converges to $x$, according to \ref{CCC.4}\refmod{\ref{CCC.4.c1}}.
Thus, if $Y$ is uniformly sequentially $\calT_\calC$-dense then it is sequentially $\calT_\calC$-dense, hence, it is $\calT_\calC$-dense.
If $Y$ is sequentially $\calT_\calC$-dense then for each $x$ there exists a sequence $\la y_j \ra_j$ satisfying the conditions in the definition above except for $\la y_j \ra_j$ may not be contained in an $X_l$ whose index $l$ is controlled by $k$ such that $x \in X_k$ and $\sup_j \lno y_j \rno \leq \Gamma \cdot \lno x \rno$ is replaced with the weaker $\sup_j \lno y_j \rno < \infty$.
This explains the word ``uniform'' in our terminology.
\end{Empty}

\begin{Theorem}
\label{FD.6}
Assume that
\begin{hypenum}
\item \label{FD.6.h1} $X[\calT]$ is a locally convex topological vector space whose topology is generated by a filtering family of seminorms $\la \| \cdot \|_i \ra_{i \in I}$;
\item \label{FD.6.h2} $\lno \cdot \rno$ is a seminorm on $X$ that is $\calT$-lower-semicontinuous;
\item \label{FD.6.h3} $\la X_k \ra_k$ is a non-decreasing sequence of $\calT$-closed vector subspaces whose union is $X$;
\item \label{FD.6.h4} $C_k = X_k \cap \{ x : \lno x \rno \leq k \}$ for each $k \in \N$;
\item \label{FD.6.h5} $\calC = \{ C_k : k \in \N \}$;
\item \label{FD.6.h6} $Y \subset X$ is a vector subspace and is uniformly sequentially $\calT_\calC$-dense in $X$;
\item \label{FD.6.h7} $f : Y \to \R$ is a linear function satisfying the following condition:
\begin{equation*}
(\forall \veps > 0)(\exists i \in I)(\exists \theta > 0) \left( \forall y \in Y \cap X_{\lceil \veps^{-1} \rceil}\right) : |f(y)| \leq \theta \cdot \|y\|_i + \veps \cdot \lno \new{y} \rno .
\end{equation*} 
\end{hypenum}
Then there exists a $\calT_\calC$-continuous linear function $\hat{f} : X \to \R$ whose restriction to $Y$ is $f$ and for all $k \in \N$
\begin{equation*}
p_{C_k} \left( \hat{f} \right) \leq p_{Y \cap C_l}(f)
\end{equation*}
where $l = \max \{ \lceil \Gamma k \rceil , \gamma(k)\}$ and $\gamma$ and $\Gamma$ are as \new{in} the definition of $Y$ being uniformly sequentially $\calT_\calC$-dense.
\end{Theorem}

\begin{proof}
{\bf (i)} We claim that if $\la y_j \ra_j$ is a sequence in $Y$ that is $\calT_\calC$-convergent in $X$ then the sequence $\la f(y_j) \ra_j$ is Cauchy in $\R$.
In order to prove this we first note that, according to \ref{CCC.4}\refmod{\ref{CCC.4.c1}}, $\la y_j \ra_j$ \new{is} a sequence in $X_k$ for some $k \in \N$ and $b = \sup_j \lno y_j \rno < \infty$.
Given $0 < \veps < k^{-1}$ let $i \in I$ and $\theta$ be associated with $\frac{\veps}{4(b+1)}$ in \refmod{hypothesis \ref{FD.6.h7}}.
Let $j_0$ be so that if $\min\{j_1,j_2\} \geq j_0$ then $\|y_{j_1} - y_{j_2} \|_i \leq \frac{\veps}{2\theta}$.
For such indexes $j_1$ and $j_2$ one has
\begin{equation*}
\left| f\left( y_{j_1} \right) - f\left( y_{j_2} \right)\right| \leq
\theta \cdot \left\| y_{j_1} - y_{j_2}\right\|_i + \frac{\veps}{4(b+1)} \cdot \lno y_{j_1} - y_{j_2} \rno \leq \veps.
\end{equation*}
\par 
{\bf (ii)} If $\la y'_j \ra_j$ and $\la y''_j \ra_j$ are two sequences in $Y$ both $\calT_\calC$-converging to the same limit in $X$ then $\lim_j f(y'_j) = \lim_j f(y''_j)$.
Indeed, interlacing both sequences into one single sequence $\la y_j \ra_j$ (\ie $y_{2j} = y'_j$ and $y_{2j-1} = y''_j$) we infer from {\bf (i)} that $\la f(y_j) \ra_j$ is convergent and its subsequences $\la f(y'_j) \ra_j$ and $\la f(y''_j) \ra_j$ have the same limit.
\par 
{\bf (iii)} We define $\hat{f} : X \to \R$ in the following way.
Given $x \in X$ there exists a sequence $\la y_j \ra_j$ in $Y$ that $\calT_\calC$-converges to $x$ and we define $\hat{f}(x) = \lim_j f(y_j)$ -- unambiguously so, according to {\bf (ii)}.
It is obvious that $\hat{f}|_Y = f$.
The linearity of $\hat{f}$ follows from that of $f$ and a routine application of {\bf (ii)}.
\par 
{\bf (iv)} We now show that $\hat{f}$ is $\calT_\calC$-continuous.
To this end we will show that $\hat{f}$ satisfies the condition \ref{CCC.5}(c).
Let $\gamma$ and $\Gamma$ be as in the definition of $Y$ being uniformly sequentially $\calT_\calC$-dense.
Let $\veps > 0$.
Define
\begin{equation*}
\hat{\veps} = \min \left\{ \frac{\veps}{\Gamma} , \frac{1}{\gamma(\lceil \veps^{-1} \rceil)}\right\} .
\end{equation*}
Let $i \in I$ and $\theta > 0$ be associated with $\hat{\veps}$ in \refmod{hypothesis \ref{FD.6.h7}}.
Let $x \in X_{\lceil \veps^{-1} \rceil}$.
Referring to the uniform sequential $\calT_\calC$-density of $Y$ in $X$, choose a sequence $\la y_j \ra_j$ in $\new{Y \cap X_{\gamma(\lceil \veps^{-1} \rceil)}}$ which is $\calT$-convergent to $x$ and so that $\sup_j \lno y_j \rno \leq \Gamma \cdot \lno x \rno$.
For all $j$ we have $|f(y_j)| \leq \theta \cdot \|y_j\|_i + \hat{\veps} \cdot \lno y_j \rno$.
Letting $j \to \infty$ we obtain $\left| \hat{f}(x)\right| = \lim_j |f(y_j)| \leq \new{\theta} \cdot \limsup_j \|y_j\|_i + \hat{\veps} \cdot \limsup_j \lno y_j \rno \leq \theta \cdot \|x\|_i + \veps \cdot \lno x \rno$.
\par 
{\bf (v)} Let $k \in \N$ and $l = \max \{ \lceil \Gamma k \rceil , \gamma(k)\}$.
Let $x \in C_k$ and choose a sequence $\la y_j \ra_j$ converging to $x$, according to the uniform sequential $\calT_\calC$-density of $Y$.
Notice that $\la y_j \ra_j$ is a sequence in $Y \cap C_l$.
Furthermore, $\left| \hat{f}(x) \right| = \lim_j |f(y_j)| \leq p_{Y \cap C_l}(f)$.
Since $x$ is arbitrary, the proof is complete.
\end{proof}

\begin{Corollary}
\label{FD.7}
Assume that
\begin{hypenum}
\item \label{FD.7.h1} $X[\calT]$ is a locally convex topological vector space;
\item \label{FD.7.h2} $\lno \cdot \rno$ is a seminorm on $X$ that is $\calT$-lower-semicontinuous;
\item \label{FD.7.h3} $\la X_k \ra_k$ is a non-decreasing sequence of $\calT$-closed vector subspaces whose union is $X$;
\item \label{FD.7.h4} $C_k = X_k \cap \{ x : \lno x \rno \leq k \}$ for each $k \in \N$;
\item \label{FD.7.h5} $\calC = \{ C_k : k \in \N \}$;
\item \label{FD.7.h6} $Y \subset X$ is a vector subspace and is uniformly sequentially $\calT_\calC$-dense in $X$;
\item \label{FD.7.h7} $f : Y \to \R$ is a linear function.
\end{hypenum}
The following are equivalent.
\begin{abcenum}
\item \label{FD.7.a1} $f$ is $\calT_\calC \hel Y$-continuous.
\item \label{FD.7.a2} $f$ is $(\calT \hel Y)_{\calC \shel Y}$-continuous.
\end{abcenum}
\end{Corollary}

\begin{proof}[Proof that \refmod{\ref{FD.7.a1}} $\Rightarrow$ \refmod{\ref{FD.7.a2}}]
This is an immediate consequence of \ref{GO.1}\refmod{\ref{GO.1.7}}.
\end{proof}

\begin{proof}[Proof that \refmod{\ref{FD.7.a2}} $\Rightarrow$ \refmod{\ref{FD.7.a1}}]
Assume that $f$ is $(\calT \hel Y)_{\calC \shel Y}$-continuous.
We apply \ref{CCC.5} to $Y[\calT \hel Y]$ in place of $X[\calT]$, the restriction of $\lno \cdot \rno$ to $Y$, $Y_k = Y \cap X_k$ in place of $X_k$, $D_k = Y_k \cap \{ y : \lno y \rno \leq k \} = Y \cap C_k$ in place of $C_k$.
Accordingly, $f$ satisfies \ref{CCC.5}(c) which is \refmod{hypothesis} \ref{FD.6}\refmod{\ref{FD.6.h7}}.
It follows from \ref{FD.6} that there exists a $\calT_\calC$-continuous $\hat{f} : \new{X} \to \R$ such that $f = \hat{f}|_Y$.
Therefore, $f$ is $\calT_\calC \hel Y$-continuous.
\end{proof}

\section{Bounded weak*, Banach-Grothendieck, and Krein-\v{S}mulian}
\label{sec.BWS}

\begin{Empty}[Bounded weak* topology]
\label{SD.1}
\label{BWS.1}
Let $Y$ be a locally convex topological vector space and let $\calB$ be a local base for its topology.
We abbreviate $\calB^\circ = \{ U^\circ : U \in \calB \}$, where $\circ$ denotes the polarity with respect to the dual system $(Y,Y^*,\la\cdot,\cdot\ra)$, see \ref{DUAL.2}(2) and \ref{DUAL.4}.
\begin{enumerate}
\item[(A)] {\it $\calB^\circ$ is a localizing family in $Y^*$.}
\end{enumerate}
\par 
{\it Proof.}
Clearly, $\calB^\circ$ is non-empty and each of its members is a convex set containing the origin.
Let $U \in \calB$, $y_0^* \in Y^*$, and $t \in \R$.
We ought to show that $y_0^* + t \cdot U^\circ \subset V^\circ$ for some $V \in \calB$.
One easily checks that it suffices to choose $V$ such that $(2t) \cdot V \subset U$ and $|\la y,y^*_0\ra| \leq \frac{1}{2}$ for all $y \in $V. \cqfd
\begin{enumerate}
\item[(B)] {\it If $\calB_1$ and $\calB_2$ are both a local base of $Y$ then $\calB_1^\circ$ and $\calB_2^\circ$ are interlaced.}
\end{enumerate}
\par 
{\it Proof.}
Let $V_1 \in \calB_1$.
There exists $V_2 \in \calB_2$ such that $V_2 \subset V_1$, thus, $V_1^\circ \subset V_2^\circ$.
One completes the proof by switching the role played by $\calB_1$ and $\calB_2$.\cqfd
\par 
We let $\sigma(Y^*,Y)$ be the usual weak* topology defined on $Y^*$, see \eg \ref{DUAL.3}.
It follows from (A) that the localized topology $\sigma(Y^*,Y)_{\calB^\circ}$, corresponding to a local base $\calB$ of $Y$, is well-defined and it follows from (B) and \ref{GO.1}\refmod{\ref{GO.1.3}} that it is independent of the choice of $\calB$.
We shall denote it as $\sigma_\rmbd(Y^*,Y)$ and call it the {\em bounded weak* topology} of $Y^*$.
Note that each $U^\circ \in \calB^\circ$ is $\sigma(Y^*,Y)$-compact, according to the Banach-Alaoglu theorem \ref{DUAL.4}(D).
\par 
Our interest for bounded weak* topologies is threefold:
\begin{enumerate}
\item[(1)] We will show \ref{SD.5} that each localized topology $\calT_\calC$ arising in our applications occurs as a bounded weak* topology in the dual of a Fr\'echet space;
\item[(2)] An explicit description of a local base for the bounded weak* topology in this case will yield an explicit description of a local base for our localized topologies \ref{SD.9}.
\item[(3)] The same will yield the Krein-\v{S}mulian theorem \ref{BWS.5} that will be used in the proof of our main existence theorem \ref{AET.FR}.
\end{enumerate}
\end{Empty}

\begin{Empty}[Basic sets in the dual]
\label{BWS.2}
Let $Y$ be a locally convex topological vector space.
We say that $V \subset Y^*$ is {\em basic} if there exists a null sequence $\la y_n \ra_{n \in \N}$ in $Y$ such that
\begin{equation*}
V = \bigcap_{n \in \N} \{ y_n \}^\circ = Y^* \cap \{ y^* : | \la y_n , y^* \ra | \leq 1 \text{ for all } n \in \N \}.
\end{equation*}
In other words, $V$ is basic if and only if it is the polar of a null sequence.
\end{Empty}

\begin{Theorem}
\label{BWS.3}
Let $Y$ be a metrizable locally convex topological vector space. 
Then the collection of basic sets is a local base for the bounded weak* topology $\sigma_\rmbd(Y^*,Y)$ of $Y^*$.
\end{Theorem}

\begin{proof}
{\bf (i)}
Since $Y$ is metrizable, there exists a local base $\{ U_k : k \in \N \}$ for the topology of $Y$ which is countable and non-decreasing (\ie $U_{k+1} \subset U_k$ for all $k \in \N$).
It follows from \ref{BWS.1}(A) and (B) that $\sigma_\rmbd(Y^*,Y) = \sigma(Y^*,Y)_{\calB^\circ}$ where $\calB^\circ = \{ U_k^\circ : k \in \N\}$.
Note also that $U_k^\circ \subset U_{k+1}^\circ$ for all $k \in \N$.
\par 
{\bf (ii)}
We start proving that each basic set $V \subset Y^*$ is a $\sigma_\rmbd(Y^*,Y)$-neighbourhood of 0.
Let $\la y_n \ra_{n \in \N}$ be a null sequence in $Y$ such that $V = \cap_{n \in \N} \{ y_n \}^\circ$.
For $n \in \N$ define $\hat{V}_n = Y^* \cap \{ y^* : |\la y_n,y^* \ra| < 1 \}$ and notice that $\hat{V}_n \in \sigma(Y^*,Y)$.
Letting $\hat{V} = \cap_{n \in \N} \hat{V}_n$ one readily checks that $0 \in \hat{V} \subset V$.
It remains to show that $\hat{V} \in \sigma_\rmbd(Y^*,Y)$.
Since $\hat{V}$ is convex, this is equivalent to showing that $\hat{V} \cap U_k^\circ \in \sigma(Y^*,Y) \hel U_k^\circ$ for all $k \in \N$, according to \ref{EUL.9}.
Given $k \in \N$, choose $n(k)$ such that $2 \cdot y_n \in U_k$ for all $n > n(k)$.
For such $n$ we have $U_k^\circ \subset \{2 \cdot y_n\}^\circ \subset \hat{V}_n$, therefore,
\begin{equation*}
\hat{V} \cap U_k^\circ = \cap_{n \in \N} \hat{V}_n \cap U_k^\circ = \left( \cap_{n=1}^{n(k)} \hat{V}_n \right) \cap U_k^\circ \in \sigma(Y^*,Y) \hel U_k^\circ.
\end{equation*}
\par 
{\bf (iii)}
In the remainder of the proof we will show that an arbitrary $\sigma_\rmbd(Y^*,Y)$-open set $W$ containing 0 contains a basic set.
In other words, we ought to prove the existence of a null sequence $\la y_n \ra_n$ such that $\cap_n \{y_n\}^\circ \subset W$.
Since $W$ is $\sigma_\rmbd(Y^{*},Y)$-open, it follows from \ref{GO.1}\refmod{\ref{GO.1.2}} that for all $k \in \N$ there exists a $\sigma(Y^{*},Y)$-open set $W_k$ such that $U_k^\circ \cap W = U_k^\circ \cap W_k$.
Notice that $U_k^\circ \cap W_{k+1} = U_k^\circ \cap \left( U_{k+1}^\circ \cap W_{k+1}\right) =  U_k^\circ \cap \left( U_{k+1}^\circ \cap W\right) = U_k^\circ \cap W = U_k^\circ \cap W_k$, since $U_k^\circ \subset U_{k+1}^\circ$, by {\bf (i)}.
\par 
{\bf (iv)}
We shall now show that there exists a sequence $\la F_k \ra_k$ of finite subsets of $Y$ such that $(1)_k$ holds for all $k \in \N$ and $(2)_k$ holds for all $k \geq 2$, where
\begin{enumerate}
\item[$(1)_k$] $U_k^\circ \cap \left( \cup_{j=1}^k F_j \right)^\circ \subset U_k^\circ \cap W_k$;
\item[$(2)_k$] $F_k \subset U_{k-1}$.
\end{enumerate}
To pick $F_1$ we simply observe that, by definition of the weak* topology $\sigma(Y^*,Y)$ of $Y^*$, there exists a finite set $F_1 \subset Y$ such that $F_1^\circ \subset W_1$.
Note that $(1)_1$ is trivially satisfied.
\par 
Assume that finite subsets $F_1,\ldots,F_k$ of $Y$ exist and satisfy $(1)_1,\ldots,(1)_k$.
We shall show how to define a finite subset $F_{k+1}$ of $Y$ satisfying $(1)_{k+1}$ and $(2)_{k+1}$.
Let $\calF$ consist of all those finite subsets $F$ of $Y$ such that
\begin{equation}
\label{eq.SD.1}
F \subset U_k.
\end{equation}
With each $F \in \calF$ we associate the set
\begin{equation*}
K_F = U_{k+1}^\circ \cap \left( \left( \cup_{j=1}^k F_j \right) \cup F \right)^\circ \cap W_{k+1}^c.
\end{equation*} 
Observe that if there exists $F \in \calF$ such that $K_F = \emptyset$ then the set $F_{k+1}=F$ satisfies $(1)_{k+1}$ and $(2)_{k+1}$.
Aiming for a contradiction, assume instead that $K_F \neq \emptyset$ for all $F \in \calF$.
In this case the collection $\{ K_F : F \in \calF \}$ has the finite intersection property.
Indeed, one easily checks that $K_{G_1} \cap \cdots \cap K_{G_N} = K_{G_1 \cup \cdots \cup G_N} \neq \emptyset$ whenever $G_1,\ldots,G_N \in \calF$ (recall \ref{DUAL.4}(B)(c)).
Furthermore, each $K_F$, $F \in \calF$, is a $\sigma(Y^*,Y)$-closed subset of $U_{k+1}^\circ$ and the latter is $\sigma(Y^*,Y)$-compact, according to the Banach-Alaoglu theorem \ref{DUAL.4}(D).
Thus, there exists $y^* \in  \cap_{F \in \calF} K_F$.
Observe that, on the one hand
\begin{equation*}
y^* \in  \bigcap_{F \in \calF} K_F \subset \bigcap_{F \in \calF}\left( \left( \bigcup_{j=1}^k F_j \right) \cup F \right)^\circ \subset \bigcap_{F \in \calF} F^\circ = \left( \bigcup_{F \in \calF} F\right)^\circ  = U_k^\circ,
\end{equation*}
by \eqref{eq.SD.1}, and on the other hand
\begin{equation*}
y^* \in  \bigcap_{F \in \calF} K_F \subset U_{k+1}^\circ \cap  \left( \bigcup_{j=1}^k F_j \right)^\circ \cap W_{k+1}^c.
\end{equation*}
Since $U_k^\circ \subset U_{k+1}^\circ$, by {\bf (i)}, we have
\begin{equation*}
\begin{split}
y^* & \in U_{k}^\circ \cap  \left( \bigcup_{j=1}^k F_j \right)^\circ \cap W_{k+1}^c \\
& \subset (U_k^\circ \cap W_k) \cap W_{k+1}^c \,\qquad\qquad\text{(by $(1)_k$)}\\
& = U_k^\circ \cap W_{k+1} \cap W_{k+1}^c \qquad\qquad\text{(by {\bf (iii)})}\\
& = \emptyset.
\end{split}
\end{equation*}
As this is impossible, the proof of the existence of $\la F_k \ra_k$ satisfying the stated properties is complete.
\par 
{\bf (v)}
Let $\la y_n \ra_n$ be any numbering of the countable set $\cup_{k \in \N} F_k$.
Abbreviate $V = \cap_n \{y_n\}^\circ = \cap_k F_k^\circ$.
Let $y^* \in V$.
Since $Y^* = \cup_k U_k^\circ$ (recall \ref{DUAL.4}(E)), there exists $k$ such that $y^* \in U_k^\circ$.
Thus, $y^* \in U_k^\circ \cap V \subset U_k^\circ \cap \left( \cup_{j=1}^k F_j \right)^\circ \subset U_k^\circ \cap W_k = U_k^\circ \cap W \subset W$, by $(1)_k$ and {\bf (iii)}.
Since $y^*$ is arbitrary, we conclude that $V \subset W$.
\par 
It remains to show that $\la y_n \ra_n$ is a null sequence in $Y$.
Fix $k \in \N$ and define $n(k) = \max \left\{ n : y_n \in \cup_{j=1}^k F_j \right\}$.
If $n > n(k)$ then $y_n \in F_j$ for some $j \geq k+1$, whence, $y_n \in F_j \subset U_{j-1} \subset U_k$. 
The proof is complete.
\end{proof}

\begin{Theorem}[Banach-Grothendieck]
\label{BWS.4}
Let $Y$ be a Fr\'echet space and $f : Y^* \to \R$ a linear function.
The following are equivalent.
\begin{enumerate}
\item[(a)] $f$ is $\sigma(Y^*,Y)$-continuous.
\item[(b)] $f$ is $\sigma_\rmbd(Y^*,Y)$-continuous.
\end{enumerate}
\end{Theorem}

\begin{proof}[Proof that $(a) \Rightarrow (b)$]
This is trivial, since $\sigma(Y^*,Y) \subset \sigma_\rmbd(Y^*,Y)$, by \ref{GO.1}\refmod{\ref{GO.1.1}}.
\end{proof}

\begin{proof}[Proof that $(b) \Rightarrow (a)$]
We assume that $f$ is $\sigma_\rmbd(Y^*,Y)$-continuous.
By \ref{BWS.3}, there exists a null sequence $\la y_n \ra_n$ in $Y$ such that 
\begin{equation}
\label{eq.BWS.1}
\cap_{n \in \N} \{ y_n \}^\circ \subset Y^* \cap \{ y^* : |f(y^*)| < 1 \}.
\end{equation}
For each $y^* \in Y^*$ we define a function $g(y^*) : \N \to \R$ by the formula $g(y^*)(n) = \la y_n , y^* \ra$.
Notice that $\lim_n g(y^*)(n) = 0$, in other words $g(y^*) \in c_0(\N)$.
Furthermore,
\begin{equation}
\label{eq.BWS.2}
|f(y^*)| \leq \|g(y^*)\|_{c_0}
\end{equation}
for all $y^* \in Y^*$.
Indeed, the inequality $\|g(y^*)\|_{c_0} \leq 1$ is equivalent to $y^* \in \cap_{n \in \N} \{y_n\}^\circ$.
Applying this observation and \eqref{eq.BWS.1} to scalar multiples of $y^*$ yields \eqref{eq.BWS.2}.
\par 
Notice that $g : Y^* \to c_0(\N)$ is linear and define a linear space $Z = \rmim(g) \subset c_0(\N)$.
Observe that a linear map $h : Z \to \R$ is well-defined by the relation $h(g(y^*)) = f(y^*)$ for all $y^* \in Y^*$, according to \eqref{eq.BWS.2}.
What's more, $h$ is continuous, by \eqref{eq.BWS.2} again.
It follows from Hahn-Banach's theorem that $h$ extends to a continuous linear $\hat{h} : c_0(\N) \to \R$.
Since $c_0(\N)^* \cong \ell_1(\N)$, there exists $\gamma \in \ell_1(\N)$ such that $\hat{h}(z) = \sum_{n \in \N} \gamma(n) z(n)$, in particular,
\begin{equation}
\label{eq.BWS.3}
f(y^*) = \hat{h}(g(y^*)) = \sum_{n=1}^\infty \gamma(n) \la y_n , y^* \ra .
\end{equation}
\par 
We let $s_k = \sum_{n=1}^k \gamma(n) \cdot y_n$ and we notice that if $q$ is a seminorm on $Y$ such that $\lim_n q(y_n) = 0$ then $\la s_k \ra_k$ is Cauchy with respect to $q$.
This observation applied to a family of seminorms that generates the topology of $Y$ shows that $\la s_k \ra_k$ converges to a limit $y \in Y$, since $Y$ is complete.
It then follows from \eqref{eq.BWS.3} that $f(y*) = \la y,y^*\ra$.
Thus, $f$ is $\sigma(Y^*,Y)$-continuous.
\end{proof}

\begin{Theorem}[Krein-\v{S}mulian]
\label{BWS.5}
Assume that
\begin{hypenum}
\item \label{BWS.5.h1} $Y$ is a Fr\'echet space;
\item \label{BWS.5.h2} $\{ U_k : k \in \N\}$ is a local base for the topology of $Y$ such that $U_{k+1} \subset U_k$ for all $k \in \N$;
\item \label{BWS.5.h3} $A \subset Y^*$ is absolutely convex.
\end{hypenum}
The following are equivalent.
\begin{enumerate}
\item[(a)] $A$ is $\sigma(Y^*,Y)$-closed.
\item[(b)] For all $k \in \N$, $A \cap U_k^\circ$ is $\sigma(Y^*,Y)$-closed.
\end{enumerate}
\end{Theorem}

\begin{proof}[Proof that $(a) \Rightarrow (b)$]
Trivial, by \ref{DUAL.4}(A).
\end{proof}

\begin{proof}[Proof that $(b) \Rightarrow (a)$]
Observe that the topology $\sigma(Y^*,Y)$ is compatible with the dual system $(Y^*,Y,\la \cdot,\cdot \ra)$ (recall \ref{DUAL.6} and \ref{DUAL.2} for the vocabulary), according to \ref{DUAL.3}(B), and so is $\sigma_\rmbd(Y^*,Y)$, according to \ref{BWS.4}.
Therefore, it follows from Mazur's theorem \ref{DUAL.6}(C) that $A$ is $\sigma(Y^*,Y)$-closed if and only if it is $\sigma_\rmbd(Y^*,Y)$-closed.
The latter is equivalent to $A \cap U_k^\circ$ being $\sigma(Y^*,Y) \hel U_k^\circ$-closed for all $k \in \N$, by \ref{CWC.2}, and in turn equivalent to $A \cap U_k^\circ$ being $\sigma(Y^*,Y$)-closed for all $k \in \N$, since $U_k^\circ$ is itself $\sigma(Y^*,Y)$-closed.
That \ref{CWC.2} applies is a consequence of the inclusions $U_k^\circ \subset U_{k+1}^\circ$ for all $k \in \N$ and of the $\sigma(Y^*,Y)$-compactness of each $U_k^\circ$ (Banach-Alaoglu theorem \ref{DUAL.4}(D)).
\end{proof}

\section{Second dual, semireflexivity, and local base}
\label{sec.SD}

\begin{Empty}
\label{SD.2}
Let $X$ be locally convex topological vector space.
Recall that the strong topology $\beta(X^*,X)$ on the first dual is defined in \ref{LCTVS.7} whereas the second dual $X^{**} = X^*[\beta(X^*,X)]^*$ and the evaluation map $\rmev : X \to X^{**}$ are defined in \ref{LCTVS.9}.
Remember that $X$ is called semireflexive if the evaluation map is surjective.
In this section, in order to avoid confusion, we let $\circ$ denote the polarity with respect to the dual system $( X , X^* , \la \cdot, \cdot \ra )$ and we let $\bullet$ denote the polarity with respect to the dual system $( X^* , X^{**} , \la \cdot , \cdot \ra )$.
\begin{enumerate}
\item[(A)] {\it If $X$ is semireflexive then $\rmev(\leftindex^\circ S)=S^\bullet$ for every $S \subset X^*$.}
\end{enumerate}
\par 
{\it Proof.}
For every $x^{**} \in X^{**}$ one has:
\begin{equation*}
\begin{split}
x^{**} \in \rmev(\leftindex^\circ S) & \iff (\exists x \in X) : x \in \leftindex^\circ S \text{ and } x^{**} = \rmev(x) \\
& \iff (\exists x \in X) : x^{**} \in S^\bullet \text{ and } x^{**} = \rmev(x) \\
& \iff x^{**} \in S^\bullet.
\end{split}
\end{equation*}\cqfd
\begin{enumerate}
\item[(B)] {\it If $X$ is semireflexive then $\rmev$ is a $(\sigma(X,X^*),\sigma(X^{**},X^*))$-homeomorphism.}
\end{enumerate}
\par 
{\it Proof.}
This is a consequence of (A) and the fact that a local base of $\sigma(X,X^*)$ consists of the sets $\leftindex^\circ F$ corresponding to all finite subsets $F$ of $X^*$ whereas a local base of $\sigma(X^{**},X^*)$ consists of the sets $F^\bullet$ corresponding to all finite subsets $F$ of $X^*$.\cqfd
\end{Empty}

\begin{Empty}[Setting]
\label{SD.3}
In the remainder of this section we will consider three spaces $X$, $X^*$, and $X^{**}$ as follows.
\begin{itemize}
\item $X$ is a vector space equipped with some localized topology $\calT_\calC$. 
\item $X^* = X[\calT_\calC]^*$ is its first dual. It is equipped with the strong topology $\beta(X^*,X)$. 
Furthermore, $X$ itself is also equipped with the weak topology $\sigma(X,X^*)$. 
Recall that the polarity with respect to the dual system $( X , X^* , \la \cdot, \cdot \ra )$ is denoted $\circ$ in this section.
\item $X^{**} = X^*[\beta(X^*,X)]^*$ is the bidual of $X$. 
It is equipped with the weak* topology $\sigma(X^{**},X^*)$ and with the bounded weak* topology $\sigma_\rmbd(X^{**},X^*)$. 
Recall that the polarity with respect to the dual system $( X^* , X^{**} , \la \cdot, \cdot \ra )$ is denoted $\bullet$ in this section.
\end{itemize}
\end{Empty}

\begin{Scholia}
\label{SD.4}
Assume that:
\begin{hypenum}
\item \label{SD.4.h1} $X[\calT]$ is a locally convex topological vector space;
\item \label{SD.4.h2} $\calC$ is a non-decreasing countable localizing family in $X$;
\item \label{SD.4.h3} Each member of $\calC$ is absolutely convex and $\calT$-closed;
\item \label{SD.4.h4} A subset of $X$ is $\calT_\calC$-bounded if and only if it is contained in some $C \in \calC$.
\end{hypenum}
The following hold.
\begin{conclenum}
\item \label{SD.4.c1} Letting $\calC = \{ C_k : k \in \N \}$ be numbered so that $C_k \subset C_{k+1}$ for all $k \in \N$ and abbreviating $U_k = \frac{1}{k} C_k^\circ$, the family $\calB = \{ U_k : k \in \N\}$ is a local base for the strong topology $\beta(X^*,X)$ of $X^*$.
\item \label{SD.4.c2} Letting $\calD = \{ U_k^\bullet : k \in \N \}$ one has $\sigma_\rmbd(X^{**},X^*) = \sigma(X^{**},X^*)_\calD$. Furthermore, $U_k^\bullet \subset U_{k+1}^\bullet$ for all $k \in \N$.
\item \label{SD.4.c3} If $X[\calT_\calC]$ is semireflexive then for all $k \in \N$ one has $\rmev(C_k) = \frac{1}{k} U_k^\bullet$ and $C_k = \frac{1}{k} \rmev^{-1}(U_k^\bullet)$.
\end{conclenum}
\end{Scholia}

\begin{proof}[Proof of \refmod{\ref{SD.4.c1}}]
It follows from \refmod{hypothesis \ref{SD.4.h4}} that $\calC$ is a fundamental system of $\calT_\calC$-bounded sets in $X$, therefore, the topology $\beta(X^*,X)$ is generated by the non-decreasing sequence of seminorms $\la p_{C_k} \ra_k$.
It easily ensues that the sets $X^* \cap \left\{ x^* : p_{C_k}(x^*) \leq \frac{1}{k} \right\}$, $k \in \N$, constitute a local base of $\beta(X^*,X)$.
It remains to observe that $X^* \cap \left\{ x^* : p_{C_k}(x^*) \leq \frac{1}{k} \right\} = \frac{1}{k}C_k^\circ$.
\end{proof}

\begin{proof}[Proof of \refmod{\ref{SD.4.c2}}]
The first conclusion is a direct consequence of \refmod{\ref{SD.4.c1}}, recalling \ref{SD.1}.
Regarding the second conclusion one recalls that $C_k \subset C_{k+1}$ so that $C_{k+1}^\circ \subset C_k^\circ$, hence, $U_{k+1} = \frac{1}{k+1}C_{k+1}^\circ \subset \frac{1}{k+1}C_k^\circ \subset \frac{1}{k}C_k^\circ = U_k$, thus, $U_k^\bullet \subset U_{k+1}^\bullet$.
\end{proof}

\begin{proof}[Proof of \refmod{\ref{SD.4.c3}}]
Since $\rmev$ is a bijection, it suffices to prove the first equality.
Note that $C_k$ is $\calT$-closed, by hypothesis, hence, $\calT_\calC$-closed, by \ref{GO.1}\refmod{\ref{GO.1.1}}, and, consequently, weakly closed.
As $C_k$ is also assumed to be absolutely convex, we infer from the bipolar theorem \ref{DUAL.4}(C) that $C_k = \leftindex^\circ(C_k^\circ)$.
Therefore, it follows from \ref{SD.2}(A) that
\begin{equation*}
\rmev(C_k) = \rmev \left( \leftindex^\circ(C_k^\circ)\right) = (C_k^\circ)^\bullet = (kU_k)^\bullet = \frac{1}{k}U_k^\bullet.
\end{equation*}
\end{proof}

\begin{Theorem}
\label{SD.5}
Assume that:
\begin{hypenum}
\item \label{SD.5.h1} $X[\calT]$ is a locally convex topological vector space;
\item \label{SD.5.h2} $\calC$ is a non-decreasing countable localizing family in $X$;
\item \label{SD.5.h3} Each member of $\calC$ is absolutely convex and $\calT$-closed;
\item \label{SD.5.h4} A subset of $X$ is $\calT_\calC$-bounded if and only if it is contained in some $C \in \calC$.
\end{hypenum}
The following are equivalent.
\begin{abcenum}
\item \label{SD.5.a1} Each member of $\calC$ is $\calT$-compact.
\item \label{SD.5.a2} $X[\calT_\calC]$ is semireflexive and the evaluation is a $(\calT_\calC, \allowbreak \sigma_\rmbd(X^{**},X^*))$-homeomorphism.
\end{abcenum}
\end{Theorem}

In the course of the proof we let $C_k$, $U_k$, and $\calD$ be as in \ref{SD.4}.

\begin{proof}[Proof that \refmod{\ref{SD.5.a1}} $\Rightarrow$ \refmod{\ref{SD.5.a2}}]
{\bf (i)} By hypothesis, each $\calT_\calC$-bounded subset $B$ of $X$ is contained in some $C \in \calC$.
According to \refmod{\ref{SD.5.a1}}, $B$ is relatively $\calT_\calC$-compact (recall \ref{GO.1}\refmod{\ref{GO.1.6}}) and, therefore, relatively weakly compact.
Consequently, it follows from \ref{DUAL.6}(H) that $X[\calT_\calC]$ is semireflexive.
\par 
{\bf (ii)} We claim that $\calT \hel C = \sigma(X,X^*) \hel C$ for all $C \in \calC$.
Indeed, the map $\rmid_C : C \to C$ is trivially $(\calT_\calC \hel C, \sigma(X,X^*) \hel C)$-continuous.
In view of \ref{GO.1}\refmod{\ref{GO.1.2}}, we infer that it is, therefore, also $(\calT \hel C, \sigma(X,X^*) \hel C)$-continuous.
Since $C[\calT \hel C]$ is compact, we conclude that $\rmid_C$ is a $(\calT \hel C, \sigma(X,X^*) \hel C)$-homeomorphism.
\par 
{\bf (iii)} We now show that $\rmev : X \to X^{**}$ is $(\calT_\calC,\sigma_\rmbd(X^{**},X^*))$-continuous.
By \ref{EUL.3}(ii), this amounts to showing that each restriction $\rmev|_{C_k} : C_k \to X^{**}$ is $(\calT \hel C_k,\sigma_\rmbd(X^{**},X^*))$-continuous.
Fix $k \in \N$ and $O \in \sigma_\rmbd(X^{**},X^*)$.
By \ref{GO.1}\refmod{\ref{GO.1.2}}, there exists $O_k \in \sigma(X^{**},X^*)$ such that $U_k^\bullet \cap kO = U_k^\bullet \cap O_k$.
Thus,
\begin{equation*}
\begin{split}
\left( \rmev|_{C_k} \right)^{-1}(O) & = C_k \cap \rmev^{-1}(O) \\
& = \frac{1}{k} \rmev^{-1}(U_k^\bullet \cap k O) \;\;\qquad\qquad\text{(by \ref{SD.4}\refmod{\ref{SD.4.c3}})}\\
& = \frac{1}{k} \rmev^{-1}( U_k^\bullet \cap O_k ) \\
& = C_k \cap \rmev^{-1}\left( \frac{1}{k} O_k \right) \qquad\qquad\text{(by \ref{SD.4}\refmod{\ref{SD.4.c3}})}\\
& \in \sigma(X,X^*) \hel C_k \qquad\qquad \qquad\text{(by \ref{SD.2}(B))}\\
& = \calT \hel C_k \;\;\,\qquad\qquad\qquad\qquad\text{(by {\bf (ii)})}.
\end{split}
\end{equation*}
\par 
{\bf (iv)} It remains to show that $\rmev^{-1} : X^{**} \to X^*$ is $(\sigma_\rmbd(X^{**},X^*),\calT_\calC)$-conti\-nuous.
By \ref{EUL.3}(ii) and \ref{SD.4}\refmod{\ref{SD.4.c2}}, this reduces to showing that the restrictions $(\rmev^{-1})|_{U_k^\bullet} : U_k^\bullet \to X$ are $(\sigma(X^{**},X^*) \hel U_k^\bullet,\calT_\calC)$-continuous.
Fix $k \in \N$ and $E \in \calT_\calC$.
Since $\frac{1}{k} E \in \calT_\calC$, it follows from \ref{GO.1}\refmod{\ref{GO.1.2}} that $C_k \cap \frac{1}{k} E \in \calT \hel C_k$, thus, $C_k \cap  \frac{1}{k} E \in \sigma(X,X^*) \hel C_k$, by {\bf (ii)}.
Accordingly, $C_k \cap \frac{1}{k} E = C_k \cap E_k$ for some $E_k \in \sigma(X,X^*)$.
Therefore,
\begin{equation*}
\begin{split}
\left[ (\rmev^{-1})|_{U_k^\bullet}\right]^{-1}(E) & =  U_k^\bullet \cap \rmev(E) \\
& = k \rmev \left( C_k \cap \frac{1}{k} E \right) \qquad\qquad\text{(by \ref{SD.2}(B))}\\
& = k \rmev(C_k \cap E_k) \\
& = U_k^\bullet \cap \rmev(k E_k) \;\qquad\qquad\quad\,\text{(by \ref{SD.2}(B))}\\
& \in \sigma(X^{**},X^*) \hel U_k^\bullet \quad\qquad\quad\,\text{(by \ref{SD.2}(B))}.
\end{split}
\end{equation*}
\end{proof}

\begin{proof}[Proof that \refmod{\ref{SD.5.a2}} $\Rightarrow$ \refmod{\ref{SD.5.a1}}]
Since $X[\calT_\calC]$ is assumed to be semireflexive, we infer from \ref{SD.4}\refmod{\ref{SD.4.c3}} $C_k = \frac{1}{k} \rmev^{-1}(U_k^\bullet)$.
It follows from the Banach-Alaoglu theorem \ref{DUAL.4}(D) that $U_k^\bullet$ is $\sigma(X^{**},X^*)$-compact, hence, it is also $\sigma_\rmbd(X^{**},X^*)$-compact, according to \ref{GO.1}\refmod{\ref{GO.1.6}}.
By the assumed continuity property of $\rmev^{-1}$, we infer that $C_k$ is $\calT_\calC$-compact and, in turn, $\calT$-compact, by \ref{GO.1}\refmod{\ref{GO.1.6}} again.
\end{proof}

\begin{Theorem}[Local base for localized topologies]
\label{SD.9}
Assume that:
\begin{hypenum}
\item \label{SD.9.h1} $X[\calT]$ is a locally convex topological vector space;
\item \label{SD.9.h2} $\calC$ is a non-decreasing countable localizing family in $X$;
\item \label{SD.9.h3} Each member of $\calC$ is absolutely convex and $\calT$-compact.
\end{hypenum}
Then the following is a local base of the localized topology $\calT_\calC$:
\begin{equation*}
\calB = \calP(X) \cap \left\{ \cap_{n \in \N} \leftindex^\circ\{x^*_n\} : \la x^*_n \ra_n \text{ is a $\beta(X^*,X)$-null sequence in $X[\calT_\calC]^*$}\right\}.
\end{equation*}
\end{Theorem}

\begin{proof}
Since each $C \in \calC$ is $\calT$-compact, \ref{CWC.2} applies to showing that a subset of $X$ is $\calT_\calC$-bounded if and only if it is contained in some $C \in \calC$.
Therefore, \ref{SD.5} applies and we infer that $\rmev : X \to X^{**}$ is a $(\calT_\calC,\sigma_\rmbd(X^{**},X^*))$-homeomorphism.
Thus, if $W \in \calT_\calC$ contains 0 then $\rmev(W)$ is a $\sigma_\rmbd(X^{**},X^*)$-open neighbourhood of 0.
Upon recalling that $X[\calT_\calC]^*$ is a Fr\'echet space \ref{FD.2} we infer from \ref{BWS.3} applied to it that there exists a $\beta(X^*,X)$-null sequence $\la x^*_n \ra_n$ in $X^*$ such that 
\begin{equation*}
\bigcap_{n \in \N} \{ x_n^*\}^\bullet \subset \rmev(W).
\end{equation*} 
Applying $\rmev^{-1}$ on both sides and referring to \ref{SD.2}(A) we obtain
\begin{equation*}
W \supset \bigcap_{n \in \N} \rmev^{\new{-1}} \left( \{ x_n^*\}^\bullet\right) = \bigcap_{n \in \N} \leftindex^\circ \{ x^*_n \} .
\end{equation*}
\new{Conversely, each member of $\calB$ is a $\calT_\calC$-neighbourhood of 0: if $\la x^*_n \ra_n$ is a $\beta(X^*,X)$-null sequence in $X[\calT_\calC]^*$ then $\cap_{n \in \N} \{x_n^*\}^\bullet$ is a $\sigma_\rmbd(X^{**},X^*)$-neighbourhood of 0, according to \ref{BWS.3}, and $\cap_{n \in \N} \leftindex^\circ\{x_n^*\} = \rmev^{-1}\left( \cap_{n \in \N} \{x_n^*\}^\bullet \right)$, by \ref{SD.2}(A).}
\end{proof}

\section{Abstract existence theorem}
\label{sec.AET}

The following is the cornerstone for all the existence results obtained in the applications to follow in forthcoming papers.

\begin{Theorem}
\label{AET.FR}
Assume that:
\begin{hypenum}
\item \label{AET.FR.h1} $\bE$ is a Fr\'echet space and $\la \bq_k \ra_k$ is \new{a} non-decreasing sequence of seminorms that generates its topology;
\item \label{AET.FR.h2} $X[\calT]$ is a locally convex topological vector space whose topology is generated by a filtering family of seminorms $\la \|\cdot\|_i \ra_{i \in I}$ and $\la X_k \ra_k$ is a non-decreasing sequence of $\calT$-closed linear subspaces whose union is $X$;
\item \label{AET.FR.h3} $\lno \cdot \rno$ is a norm on $X$ and $\calT$-lower-semicontinuous;
\item \label{AET.FR.h4} For every $k \in \N$ the set $C_k = X_k \cap \{ u : \lno u \rno \leq k \}$ is so that the topological space $C_k[\calT \hel C_k]$ is compact;
\item \label{AET.FR.h5} $\bD : \bE \to X^\bigstar$ is linear\footnote{Recall \ref{DUAL.1} for the notation $X^\bigstar$.} and as usual we write $\la u , \bD(v) \ra$ instead of $\bD(v)(u)$ whenever $v \in \bE$ and $u \in X$;
\item \label{AET.FR.h6} For every $v \in \bE$ and every $\veps > 0$ there are $i \in I$ and $\theta > 0$ such that for all $u \in X_{\lceil \veps^{-1}\rceil}$
\begin{equation*}
|\la u , \bD(v) \ra | \leq \theta \cdot \|u\|_i + \veps \cdot \lno u \rno ;
\end{equation*}
\item \label{AET.FR.h7} For every $k \in \N$ there exist $\bcG(k) > 0$ \new{and $\kappa(k) \in \N$} such that
\begin{equation*}
| \la u , \bD(v) \ra | \leq \bcG(k) \cdot \lno u \rno \cdot \bq_{\new{\kappa(k)}}(v)
\end{equation*}
whenever $u \in X_k$ and $v \in \bE$;
\item \label{AET.FR.h8} For every $k \in \N$ there exist $\bcH(k) > 0$ such that
\begin{equation*}
\lno u \rno \leq \bcH(k) \cdot \sup \left\{ | \la u , \bD(v) \ra | : v \in \bE \text{ and } \bq_{k}(v) \leq 1 \right\}
\end{equation*}
whenever $u \in X_{k}$;
\item \label{AET.FR.h9} For every $k \in \N$ and every $u \in X \setminus X_{k}$ there exists $v \in \bE$ such that 
\begin{equation*}
\la u , \bD(v) \ra \neq 0 \quad \text{ and } \quad \bq_k(v) = 0.
\end{equation*}
\end{hypenum}
Letting $\calC = \{ C_k : k \in \N \}$ and $\calT_\calC$ be the corresponding localized topology and letting
\begin{equation*}
\vvvert F \vvvert_k = \sup \left\{ | \la u , F \ra | : u \in X_k \text{ and } \lno u \rno \leq 1\right\} \,,
\end{equation*}
for $F \in X[\calT_\calC]^*$ and $k \in \N$, we conclude that:
\begin{conclenum}
\item \label{AET.FR.c1} $X[\calT_\calC]^*[\beta(X^*,X)]$ is a Fr\'echet space whose topology is generated by the sequence of seminorms $\la \vvvert \cdot \vvvert_k \ra_k$; 
\item \label{AET.FR.c2} $\bD : \bE \to X[\calT_\calC]^*$ is a linear continuous {\em surjective} operator;
\item \label{AET.FR.c3} For every $F \in X[\calT_\calC]^*$, every $k \in \N$ such that $\vvvert F \vvvert_k > 0$, and every $\veps > 0$ there exists $v \in \bE$ such that
\begin{equation*}
F = \bD(v) \quad\text{ and }\quad \bq_k(v) \leq (1+\veps) \cdot \bcH(k) \cdot \vvvert F \vvvert_k \,.
\end{equation*}
\item \label{AET.FR.c4} For every $k \in \N$ and $\veps > 0$ there exists
\begin{equation*}
\bI : X[\calT_\calC]^* \cap \{ F : \vvvert F \vvvert_k > 0 \} \to \bE
\end{equation*}
satisfying the following properties\footnote{\refmod{The labelling matches that of conclusion \ref{AET.BA}\refmod{\ref{AET.BA.c4}} below, of which item (b) has no counterpart here.}}
\begin{enumerate}
\item[(a)] $(\bD \circ \bI)(F) = F$ for all $F \in X[\calT_\calC]^*$ such that $\vvvert F \vvvert_k > 0$;
\item[(c)] $\bI$ is continuous;
\item[(d)] for all $F \in X[\calT_\calC]^*$ such that $\vvvert F \vvvert_k > 0$ we have
\begin{equation*}
\bq_k \left( \bI(F) \right) \leq (1+\veps) \cdot \bcH(k) \cdot \vvvert F \vvvert_k .
\end{equation*}
\end{enumerate}

\end{conclenum}
\end{Theorem}

\begin{Remark}
\label{AET.RE.1}
The following comments about the hypotheses are in order.
\begin{enumerate}
\item[\refmod{(H1,H2)}] In the important case when $\bE$ is a Banach space, $\calT$ is normable,  and $X_k = X$ for all $k \in \N$ our hypotheses simplify a fair amount and the strengthened conclusion that $X[\calT_\calC]^*$ equipped with its strong topology is a Banach space holds, see Corollary \ref{AET.BA}.
\item[\refmod{(H2)}] In case $\calT$ is normable, hypotheses \refmod{\ref{AET.FR.h4}} and \refmod{\ref{AET.FR.h6}} can be rephrased, see remarks \refmod{(H4)} and \refmod{(H6)} below.
\item[\refmod{(H3)}] Contrary to what has been assumed so far we now hypothesize that $\lno \cdot \rno$ is a {\em norm} rather than merely a seminorm. 
In fact, {\it assuming that $\lno \cdot \rno$ is a seminorm and assuming hypotheses \refmod{\ref{AET.FR.h7}} and \refmod{\ref{AET.FR.h8}} hold, the following are equivalent:
\begin{abcenum}[left=.5cm]
\item \label{AET.RE.1.a1} $\lno \cdot \rno$ is a norm;
\item \label{AET.RE.1.a2} For every $u \in X$, {\em if} $\la u , \bD(v) \ra = 0$ for every $v \in \bE$, {\em then} $u=0$.
\end{abcenum}
Proof.} 
In order to prove that \refmod{\ref{AET.RE.1.a1}} $\Rightarrow$ \refmod{\ref{AET.RE.1.a2}}, let $u \in X$ and assume that $\la u , \bD(v) \ra = 0$ for all $v \in \bE$. 
We ought to show that $u=0$. 
This is the case because $\lno u \rno = 0$, according to \refmod{hypothesis \ref{AET.FR.h8}}, and $\lno \cdot \rno$ is a norm. 
In order to prove that \refmod{\ref{AET.RE.1.a2}} $\Rightarrow$ \refmod{\ref{AET.RE.1.a1}}, let $u \in X$ and assume that $\lno u \rno = 0$. %
Note that $\la u , \bD(v) \ra = 0$ for all $v \in \bE$, according to \refmod{hypothesis \ref{AET.FR.h7}}. 
We conclude that $u=0$, by \refmod{\ref{AET.RE.1.a2}}. \cqfd 
\par
In the applications to follow, condition \refmod{\ref{AET.RE.1.a2}} will correspond to a ``constancy theorem''. 
It is used in the proof of the theorem to establish that $\rmim \bD$ is dense via an application of the Hahn-Banach theorem. 
This is a convenient substitute for smoothing and approximation used in \cite[4.1]{DEP.PFE.06b} and \cite[4.3]{DEP.TOR.09}.
\item[\refmod{(H4)}] This is a critical hypothesis, as it guarantees that $X[\calT_\calC]$ is semireflexive via an application of \ref{SD.5}.
The semireflexivity property is used often in the proof of \refmod{\ref{AET.FR.c2}} and \refmod{\ref{AET.FR.c3}}.
\item[\refmod{(H5)}] By the assumption the map 
\begin{equation*}
\la \cdot , \cdot \ra : X \times \bE \new{\to \R} : (u,v) \mapsto \la u ,\bD(v) \ra
\end{equation*}
is bilinear. 
It follows from \refmod{\ref{AET.RE.1.a2}} above that it is separating in the first variable but it is in general not separating in the second variable. 
In other words $\bD$ is not injective and in fact its kernel is in general not complemented in $\bE$, in which case conclusion \refmod{\ref{AET.FR.c3}} is valuable.
\item[\refmod{(H6)}] \new{If $\calT$ is normed by $\|\cdot\|$ then hypothesis \refmod{\ref{AET.FR.h6}} becomes the same with ``there are $i \in I$ and'' replaced by ``there is'' and $\theta \cdot \|u\|_i$ replaced by $\theta \cdot \|u\|$.}
\end{enumerate}
\end{Remark}

\begin{proof}[Preliminaries to the proof]
As in the preliminaries to the proof of \ref{CCC.4} we note that $\calC$ is a non-decreasing countable localizing family on $X$.
Therefore, the localized topology $\calT_\calC$ on $X$ is defined, by \ref{EUL.4}.
In the course of the proof we will use the notations $X^* = X[\calT_\calC]^*$ and $X^{**} = X^*[\beta(X^*,X)]^*$.
Furthermore, each $C \in \calC$ being $\calT$-compact we may apply \ref{CWC.2}\refmod{\ref{CWC.2.c3}}.
In other words, $\calC$ is a fundamental system of $\calT_\calC$-bounded sets in $X$ and \ref{FD.2} applies.
Since each $C \in \calC$ is readily absolutely convex, \ref{SD.5} applies as well: $X[\calT_\calC]$ is semireflexive, \ie the evaluation map $\rmev : X \to X^{**}$ is a bijection, a fact that will be used repeatedly in the proof.
\par 
In the course of the proof, we shall only consider one topology on the following three spaces:
\begin{itemize}
\item The localized topology $\calT_\calC$ on $X$;
\item The strong topology $\beta(X^*,X)$ on $X^*$;
\item The weak* topology $\sigma(X^{**},X^*)$ on $X^{**}$.
\end{itemize}
\end{proof}

\begin{proof}[Proof of \refmod{\ref{AET.FR.c1}}]
That $X[\calT_\calC]^*$ equipped with its strong topology $\beta(X^*,X)$ is a Fr\'echet space follows from \ref{FD.2}, since $\calC$ is a fundamental system of $\calT_\calC$-bounded sets of $X$.
For this reason also, $\la p_{C_k} \ra_k$ is a sequence of seminorms that generates the strong topology of $X^*$.
To finish the proof it remains to note that $\new{p_{C_k}(F) = k \vvvert F \vvvert_k}$ for all $F \in X^*$ and all $k \in \N$.
\end{proof}

\begin{proof}[Proof of \refmod{\ref{AET.FR.c2}}]
The proof is an application of the closed range theorem.
\par 
{\bf (i)}
We first observe that $\bD(v)$, for each $v \in \bE$, is $\calT_\calC$-continuous, \ie $\bD$ takes values in $X^*$.
Indeed, this follows from assumption \refmod{\ref{AET.FR.h6}} and \ref{CCC.5}.
\par 
{\bf (ii)}
We next note that $\bD : \bE \to X^*$ is continuous.
Indeed, for each $v \in \bE$ and $k \in \N$ it follows from hypothesis \refmod{\ref{AET.FR.h7}} that 
\begin{equation*}
\vvvert \bD(v) \vvvert_k = \sup \left \{ | \la u , \bD(v) \ra | : u \in X_k \text{ and } \lno u \rno \leq 1\right\}
\\
\leq \bcG(k) \cdot \bq_{\new{\kappa(k)}}(v).
\end{equation*}
\par 
{\bf (iii)}
We now establish that $\rmim(\bD)$ is $\beta(X^*,X)$-dense in $X^*$.
In view of the Hahn-Banach theorem \cite[3.5]{RUDIN}, this is equivalent to showing that for all $\alpha \in X^{**}$ {\bf if} $\la \bD(v),\alpha \ra = 0$ for all $v \in \bE$ {\em then} $\alpha = 0$.
With each $\alpha \in X^{**}$ there corresponds $u \in X$ such that $\alpha = \rmev(u)$, by semireflexivity.
Thus, the vanishing of $\la \bD(v),\alpha \ra$ for all $v \in \bE$ is equivalent to the vanishing of $\la u,\bD(v) \ra$ for all $v \in \bE$.
The latter implies that $u=0$, by assumptions \refmod{\ref{AET.FR.h3}} and \refmod{\ref{AET.FR.h8}} (recall the discussion in \ref{AET.RE.1}\refmod{(H3)}).
Finally, $\alpha = \rmev(u)=0$.
\par 
{\bf (iv)}
Here, we will show that $\rmim(\bD^*)$ is $\sigma(E^*,E)$-closed.
By the Krein-\v{S}mulian theorem \ref{BWS.5} applied to the Fr\'echet space $\bE$, this is equivalent to proving that $\rmim(\bD^*) \cap V_k^\circ$ is $\sigma(E^*,E)$-closed for all $k \in \N$, where $V_k = \bE \cap \left\{ v : \bq_k(v) \leq \frac{1}{k} \right\}$ (indeed, $\{ V_k : k \in \N\}$ is readily a local base for $\bE$ and $V_{k+1} \subset V_k$).
\par 
Fix $k \in \N$ and let $\la v^*_\lambda \ra_{\lambda \in \Lambda}$ be a net in $\rmim(D^*) \cap V_k^\circ$ that $\sigma(E^*,E)$-converges to some $v^* \in \bE\new{^*}$.
We ought to show that $v^* \in \rmim(\bD^*)$ (recall that $V_k^\circ$ is $\sigma(E^*,E)$-closed).
For each $\lambda \in \Lambda$ choose $\alpha_\lambda \in X^{**}$ so that $v_\lambda^* = \bD^*(\alpha_\lambda)$\new{.}
Next, referring to the semireflexivity of $X$, choose $u_\lambda \in X$ so that $\alpha_\lambda = \rmev(u_\lambda)$.
Note that for all $v \in \bE$ and all $k \in \N$ one has
\begin{equation}
\label{eq.AET.1}
\la v , v_\lambda^* \ra = \la v , \bD^*(\alpha_\lambda) \ra = \la \bD(v) , \alpha_\lambda \ra = \la \bD(v) , \rmev(u_\lambda) \ra = \la u_\lambda , \bD(v) \ra .
\end{equation}
Our next task consists in showing that the net $\la u_\lambda \ra_{\lambda \in \Lambda}$ is in some $C \in \calC$.
\par 
We claim that $u_\lambda \in X_k$ for all $\lambda \in \Lambda$.
Assume if possible that $u_\lambda \in X \setminus X_k$ for some $\lambda \in \Lambda$.
According to assumption \refmod{\ref{AET.FR.h9}}, there exists $v \in \bE$ such that $\bq_k(v) = 0$ and $\la u_\lambda , \bD(v) \ra \neq 0$.
Note that $\bq_k(tv) = 0$ for each $t \in \R \setminus \{0\}$, hence, $tv \in V_k$.
As $v_\lambda^* \in V_k^*$, we infer from \eqref{eq.AET.1} that
\begin{equation*}
1 \geq | \la tv , v_\lambda^* \ra | = | \la u_\lambda , \bD(tv) \ra | = |t|. |\la u_\lambda , \bD(v) \ra |.
\end{equation*}
Since $t$ is arbitrary, we obtain $\la u_\lambda , \bD(v) \ra = 0$, a contradiction.
\par 
Next, we claim that $\lno u_\lambda \rno \leq \bcH(k) \cdot k$ for all $\lambda \in \Lambda$.
Fix $\lambda \in \Lambda$.
First observe that if $v \in \bE$ and $\bq_k(v) \leq 1$ then $k^{-1}v \in V_k$, whence, $|\la k^{-1}v , v_\lambda^* \ra| \leq 1$, since $v_\lambda^* \in V_k^\circ$.
In turn, $| \la u_\lambda , \bD(v) \ra| \leq k$, by \eqref{eq.AET.1}.
Since $u_\lambda \in X_k$, by the previous paragraph, and $v$ is arbitrary, it follows from hypothesis \refmod{\ref{AET.FR.h8}} that $\lno u_\lambda \rno \leq \bcH(k) \cdot k$.
\par 
Letting $\kappa = \max \left\{ \lceil \bcH(k) \cdot k \rceil , k \right\}$, we infer from the previous two paragraphs, the definition of $C_\kappa$, and the fact that the sequence $\la C_n \ra_n$ is non-decreasing that the net $\la u_\lambda \ra_{\lambda \in \Lambda}$ is in $C_\kappa$.
As $\calT \hel C_\kappa$ is compact, by hypothesis \refmod{\ref{AET.FR.h4}}, and $\calT_\calC \hel C_\kappa = \calT \hel C_\kappa$, by \ref{GO.1}\refmod{\ref{GO.1.2}}, there exists a subnet\footnote{For a discussion of nets, subnets, and compactness see \eg \cite[Ch.2 and Ch.5 \S 1 Theorem 2]{KELLEY}.} $\la u_{\lambda'} \ra_{\lambda' \in \Lambda'}$ of $\la u_\lambda \ra_{\lambda \in \Lambda}$ that $\calT_\calC$-converges in $X$ to some $u \in C_\kappa$.
Let $\alpha = \rmev(u)$.
According to the $(\calT_\calC,\sigma(X^{**},X^*))$-continuity of $\rmev$ \ref{LCTVS.9}(A), we infer that $\la \alpha_{\lambda'} \ra_{\lambda' \in \Lambda'}$ is $\sigma(X^{**},X^*)$-convergent\footnote{For continuity and nets see \eg \cite[Ch.3 \S 1 Theorem 1]{KELLEY}.} to $\alpha$.
According to the $(\sigma(X^{**},X^*),\sigma(E^*,E))$-continuity of $\bD^*$, we infer that $\la v_{\lambda'}^* \ra_{\lambda' \in \Lambda'}$ is $\sigma(E^*,E)$-convergent to $\bD^*(\alpha)$.
Since it also $\sigma(E^*,E)$-converges to $v^*$, we conclude that $v^* = \bD^*(\alpha)$.
This completes the proof that $\rmim(\bD^*)$ is $\sigma(E^*,E)$-closed.
\par 
{\bf (v)}
Recalling that $X^*$ is a Fr\'echet space, by conclusion \refmod{\ref{AET.FR.c1}}, we may apply the closed range theorem \cite[8.6.13]{EDWARDS} to the continuous linear map $\bD$.
Since $\rmim(\bD^*)$ is $\sigma(E^*,E)$-closed, by {\bf (iv)}, it follows that $\rmim(\bD)$ is closed.
Since it is also dense, by {\bf (iii)}, we conclude that $\bD$ is surjective.
\end{proof}

\begin{proof}[Proof of \refmod{\ref{AET.FR.c3}}]
The proof is an application of the Hahn-Banach theorem.
\par 
{\bf (i)}
Let $F \in X^*$, $k \in \N$ such that $\vvvert F \vvvert_k > 0$, and $\veps > 0$.
Define
\begin{equation*}
A = \bE \cap \left\{ v : \bq_k(v) < (1+\veps) \cdot \bcH(k) \cdot \vvvert F \vvvert_k \right\}
\end{equation*}
and 
\begin{equation*}
B = \bE \cap \{ v : \bD(v) = F \} .
\end{equation*}
Observe that conclusion \refmod{\ref{AET.FR.c3}} is equivalent to $A \cap B \neq \emptyset$.
In the remainder of this proof we shall assume instead that $A \cap B = \emptyset$ and derive a contradiction.
\par 
{\bf (ii)}
Notice that $A \neq \emptyset$ (because $\vvvert F \vvvert_k > 0$) and $B \neq \emptyset$ (by conclusion \refmod{\ref{AET.FR.c2}}).
Observe also that $A$ and $B$ are both convex and that $A$ is open.
Since we assume that $A \cap B = \emptyset$, it follows from Hahn-Banach's theorem \cite[3.4(a)]{RUDIN} applied in $\bE$ that there are $v^* \in E^*$ and $\gamma \in \R$ such that
\begin{equation}
\label{eq.AET.2}
\la v , v^* \ra < \gamma \leq \la w , v^* \ra
\end{equation}
for all $v \in A$ and all $w \in B$.
For the rest of this proof we fix some $w_0 \in B$.
\par 
{\bf (iii)}
$\ker \bD \subset \ker v^*$.
Indeed, if $w \in \ker \bD$ and $t \in \R$ then $w_0 + tw \in B$, thus, $\gamma \leq \la w_0,v^* \ra + t \la w,v^* \ra$, by \eqref{eq.AET.2}.
Since $t$ is arbitrary, $\la w , v^* \ra = 0$.
\par 
{\bf (iv)}
{\it There exists $\alpha \in X^{**}$ such that $v^* = \alpha \circ \bD$.}
One defines $\alpha : X^* \to \R$ as follows.
Given $G \in X^*$ we choose $v \in \bE$ so that $G = \bD(v)$, by means of conclusion \refmod{\ref{AET.FR.c2}}, and we let $\alpha(G) = \la v , v^* \ra$.
This is well-defined, according to {\bf (iii)}, and, therefore, also linear.
Finally, observe that
\begin{equation*}
X^* \cap \{ G : | \la G, \alpha \ra| < 1 \} = \bD \left( \bE \cap \{ v : | \la v,v^* \ra| < 1 \} \right)
\end{equation*}
is open in $X^*$, since $v^*$ is continuous and $\bD$ is open as a consequence of the open mapping theorem \cite[2.12(a)]{RUDIN}, conclusion \refmod{\ref{AET.FR.c2}}, and the continuity of $\bD$ (see proof of \refmod{\ref{AET.FR.c2}} {\bf (ii)}).
Accordingly, $\alpha$ is continuous, \ie $\alpha \in X^{**}$.
\par 
{\bf (v)}
By semireflexivity, there exists $u \in X$ such that $\alpha = \rmev(u)$.
We note that
\begin{equation}
\label{eq.AET.3}
\la v,v^* \ra = \la v , \alpha \circ \bD \ra = \la \bD(v) , \alpha \ra = \la \bD(v),\rmev(u) \ra = \la u , \bD(v) \ra
\end{equation}
for all $v \in \bE$.
\par 
{\bf (vi)}
$u \in X_k$.
If not, by hypothesis \refmod{\ref{AET.FR.h9}}, there exists $v \in \bE$ such that $\la u , \bD(v) \ra \neq 0$ and $\bq_k(v) = 0$.
Let $t \in \R$.
Since $\bq_k(tv)=0$, one has $tv \in A$, thus, $\la tv , v^* \ra < \gamma$, by \eqref{eq.AET.2}.
As $t$ is arbitrary, $0 = \la v,v^* \ra = \la u , \bD(v) \ra$, by \eqref{eq.AET.3}, a contradiction.
\par 
{\bf (vii)}
We claim that
\begin{equation}
\label{eq.AET.4}
\lno u \rno  \leq \frac{\gamma}{\left( 1 + \frac{\veps}{2}\right) \cdot \vvvert F \vvvert_k}.
\end{equation}
In order to establish this, let $v \in \bE$ be such that $\bq_k(v) \leq 1$ and define
\begin{equation*}
v' = v \times \left( 1 + \frac{\veps}{2}\right) \cdot \bcH(k) \cdot \vvvert F \vvvert_k.
\end{equation*}
Note that $-v',v' \in A$.
Therefore,
\begin{equation*}
\pm \la u , \bD(v') \ra = \pm \la v' , v^* \ra < \gamma,
\end{equation*}
by \eqref{eq.AET.3} and \eqref{eq.AET.2}, respectively.
Consequently,
\begin{equation*}
| \la u , \bD(v) \ra | \leq \frac{\gamma}{\left( 1 + \frac{\veps}{2}\right) \cdot \bcH(k) \cdot \vvvert F \vvvert_k}.
\end{equation*}
By {\bf (vi)}, the arbitrariness of $v$, and hypothesis \refmod{\ref{AET.FR.h8}}, we conclude that \eqref{eq.AET.4} holds.
\par 
{\bf (viii)}
We observe that
\begin{equation*}
\begin{split}
\gamma &\leq \la w_0 , v^* \ra \quad\,\qquad\qquad \text{(by \eqref{eq.AET.2}, since $w_0 \in B$)}\\
& = \la u , \bD(w_0) \ra \qquad\qquad \text{(by \eqref{eq.AET.3})}\\
& = \la u , F \ra  \qquad\qquad\qquad \text{(by definition of $B$ and $w_0 \in B$)}\\
& \leq \lno u \rno \cdot \vvvert F \vvvert_k \qquad\qquad\;\; \text{(by definition of $\vvvert \cdot \vvvert_k$ and {\bf (vi)})}\\
& \leq \frac{\gamma}{1 + \frac{\veps}{2}}. \qquad\qquad\quad\;\;\, \text{(by \eqref{eq.AET.4})}
\end{split}
\end{equation*}
Upon observing that $0 \in A$ we infer from \eqref{eq.AET.2} that $0 < \gamma$.
This readily contradicts the preceding inequality.
On that account the proof is complete.
\end{proof}

\begin{proof}[Proof of \refmod{\ref{AET.FR.c4}}]
{\bf (i)}
Abbreviate $\bF = X[\calT_\calC]^* \cap \{ F : \vvvert F \vvvert_k > 0 \}$.
Note that $\bF$ is made a metric space by considering the restriction to $\bF$ of the strong topology $\beta(X^*,X)$ of $X[\calT_\calC]^*$.
Define $\Gamma : \bF \to \calP(\bE)$ by the formula
\begin{equation*}
\Gamma(F) = \bE \cap \left\{ v : \bD(v) = F \text{ and } \bq_k(v) < (1+\veps) \cdot \bcH(k) \cdot \vvvert F \vvvert_k \right\}.
\end{equation*}
We observe that for each $F \in \bF$, $\Gamma(F)$ is convex, since $\bD$ is linear and $\bq_k$ is a seminorm, and non-empty, by conclusion \refmod{\ref{AET.FR.c3}} (applied with, \eg $\frac{\veps}{2}$).
\par 
{\bf (ii)}
We now show that $\Gamma$ is lower semi-continuous in the sense of \ref{MCC.2}.
Let $O \subset \bE$ be open.
We ought to show that the set
\begin{equation*}
\calF := \bF \cap \{ F : \Gamma(F) \cap O \neq \emptyset \}
\end{equation*} 
is open.
As the case $\calF = \emptyset$ is trivial, we henceforth assume that $\calF$ is non-empty.
Let $F_0 \in \calF$.
There exists $v_0 \in \Gamma(F_0) \cap O$.
In particular, $\bq_k(v_0) < (1+\veps) \cdot \bcH(k) \cdot \vvvert F_0 \vvvert_k$.
Thus, we may choose $\eta > 0$ small enough for 
\begin{equation*}
\bq_k(v_0) + \eta < (1+\veps) \cdot \bcH(k) \left( \vvvert F_0 \vvvert_k - \eta \right).
\end{equation*}
Define an open set $O_0 = O \cap \{ v : \bq_k(v-v_0) < \eta \}$.
Since $X[\calT_\calC]^*$ is Fr\'echet, by conclusion \refmod{\ref{AET.FR.c1}}, and $\bD$ is continuous and surjective, by conclusion \refmod{\ref{AET.FR.c2}}, we infer from the open mapping theorem \cite[2.12(a)]{RUDIN} that $U_0 = \bD(O_0)$ is open.
Next, we define $U = U_0 \cap \bF \cap \{ F : \vvvert F-F_0 \vvvert_k < \eta \}$ which is an open neighbourhood of $F_0$ in $\bF$.
Let $F \in U$.
We shall show that $F \in \calF$, \ie $\Gamma(F) \cap O \neq \emptyset$.
Since $F \in U_0$, there exists $v \in \bE$ such that $\bD(v) = F$ and $\bq_k(v-v_0) < \eta$.
Upon noticing that $\vvvert F_0 \vvvert_k < \vvvert F \vvvert_k + \vvvert F -  F_0 \vvvert_k < \vvvert F \vvvert_k + \eta$, since $F \in \new{U}$, we infer that
\begin{multline*}
\bq_k(v) \leq \bq_k(v_0) + \bq_k(v-v_0) < \bq_k(v_0) + \eta < (1+\veps) \cdot \bcH(k) \left( \vvvert F_0 \vvvert_k - \eta \right)
\\
< (1+\veps) \cdot \bcH(k) \cdot \vvvert F \vvvert_k.
\end{multline*}
Therefore, $v$ is a witness that $\Gamma(F) \cap O \neq \emptyset$.
Since $F$ is arbitrary, we have $U \subset \calF$, \ie $\calF$ is a neighbourhood of $F_0$.
As $F_0$ is arbitrary, we conclude that $\calF$ is open.
\par 
{\bf (iii)}
We now consider $\bar{\Gamma} : \bF \to \calP(\bE)$ defined by $\bar{\Gamma} = \rmclos \Gamma(F)$ for all $F \in \bF$.
Clearly, $\bar{\Gamma}(F)$ is non-empty, convex, and closed, $F \in \bF$.
Moreover, $\bar{\Gamma}$ is lower semi-continuous, according to {\bf (ii)} and \ref{MCC.2}(A).
Notice that if $v \in \bar{\Gamma}(F)$ then $\bD(v) = F$, since $\bD$ is continuous, and $\bq_k(v) \leq (1+\veps) \cdot \bcH(k) \cdot \vvvert F \vvvert_k$.
\par 
{\bf (iv)} 
It follows from {\bf (iii)} and Michael's selection theorem \ref{MCC.2}(D) that there exists a continuous selector $\bI : \bF \to \bE$, \ie $\bI(F) \in \bar{\Gamma}(F)$ for all $F \in \bF$.
Thus, conclusion (c) is satisfied and conclusions (a) and (d) follow both from the last sentence of {\bf (iii)}.
\end{proof}

\begin{Corollary}
\label{AET.BA}
Assume that:
\begin{hypenum}
\item \label{AET.BA.h1} $\bE[\bq]$ is a Banach space;
\item \label{AET.BA.h2} $X[\|\cdot\|]$ is a normed space and $\calT$ is the topology associated with $\|\cdot\|$;
\item \label{AET.BA.h3} $\lno \cdot \rno$ is a $\calT$-lower-semicontinuous norm on $X$;
\item \label{AET.BA.h4} $X \cap \{ u : \lno u \rno \leq 1 \}$ is $\calT$-compact;
\item \label{AET.BA.h5} $\bD : \bE \to \new{X^\bigstar}$ is linear;
\item \label{AET.BA.h6} For every $v \in \bE$ and every $\veps > 0$ there exists $\theta > 0$ such that
\begin{equation*}
| \la u,\bD(v) \ra | \leq \theta \cdot \|u\| + \veps \cdot \lno u \rno 
\end{equation*}
whenever $u \in X$;
\item \label{AET.BA.h7} There exist $\bcG > 0$ such that
\begin{equation*}
| \la u , \bD(v) \ra | \leq \bcG \cdot \lno u \rno \cdot \bq(v)
\end{equation*}
whenever $u \in X$ and $v \in \bE$;
\item \label{AET.BA.h8} There exist $\bcH > 0$ such that
\begin{equation*}
\lno u \rno \leq \bcH \cdot \sup \left\{ | \la u , \bD(v) \ra | : v \in \bE \text{ and } \bq(v) \leq 1 \right\}
\end{equation*}
whenever $u \in X$.
\end{hypenum}
Letting $C_k = X \cap \{ u : \lno u \rno \leq k \}$, $\calC = \{ C_k : k \in \N \}$ and $\calT_\calC$ be the corresponding localized topology, and letting
\begin{equation*}
\vvvert F \vvvert = \sup \left\{ | \la u , F \ra | : u \in X \text{ and } \lno u \rno \leq 1\right\} \,,
\end{equation*}
for $F \in X[\calT_\calC]^*$, we conclude that:
\begin{conclenum}
\item \label{AET.BA.c1} $X[\calT_\calC]^*[\beta(X^*,X)]$ is a Banach space whose topology is associated with the norm $\vvvert \cdot \vvvert$;
\item \label{AET.BA.c2} $\bD : \bE \to X[\calT_\calC]^*$ is a linear continuous {\em surjective} operator;
\item \label{AET.BA.c3} For every $F \in X[\calT_\calC]^*$ and every $\veps > 0$ there exists $v \in \bE$ such that $F = \bD(v)$ and
\begin{equation*}
\bq(v) \leq (1+\veps) \cdot \bcH \cdot \vvvert F \vvvert \,.
\end{equation*}
\item \label{AET.BA.c4} For every $\veps > 0$ there exists
\begin{equation*}
\bI : X[\calT_\calC]^* \to \bE
\end{equation*}
satisfying the following properties.
\begin{enumerate}
\item[(a)] $(\bD \circ \bI)(F) = F$ for all $F \in X[\calT_\calC]^*$;
\item[(b)] $\bI$ is homogeneous of degree 1, \ie for all $F \in X[\calT_\calC]^*$ and $t > 0$ we have $\bI(t \cdot F) = t \cdot \bI(F)$;
\item[(c)] $\bI$ is continuous;
\item[(d)] for all $F \in X[\calT_\calC]^*$ we have
\begin{equation*}
\bq \left( \bI(F) \right) \leq (1+\veps) \cdot \bcH \cdot \vvvert F \vvvert .
\end{equation*}
\end{enumerate}
\end{conclenum}
\end{Corollary}

\begin{proof}
With regard to conclusions \refmod{\ref{AET.BA.c1}}, \refmod{\ref{AET.BA.c2}}, and \refmod{\ref{AET.BA.c3}} it suffices to apply \ref{AET.FR} with $X_k = X$ for all $k \in \N$, $\|\cdot\|_i = \|\cdot\|$ for all $i \in I$, \new{$\kappa(k)=k$ for all $k \in \N$,} and $C_k = k \cdot C_1$ where $C_1$ is the set considered in hypothesis \ref{AET.FR}\refmod{\ref{AET.FR.h4}}. 
\par 
We now turn to establishing conclusion \refmod{\ref{AET.BA.c4}} whose proof is similar to the homonymous conclusion of \ref{AET.FR}.
We now indicate the differences.
In {\bf (i)} we redefine $\bF = X[\calT_\calC]^* \cap \{ F : \vvvert F \vvvert = 1 \}$ to be the unit sphere of the Banach space $X[\calT_\calC]^*$ and $\Gamma : \bF \to \calP(\bE)$ is defined by means of the formula
\begin{equation*}
\Gamma(F) = \bE \cap \{ v : \bD(v) = F \text{ and } \bq(v) < (1+\veps) \cdot \bcH \}.
\end{equation*}
Each $\Gamma(F)$ is readily convex and is non-empty, according to conclusion \refmod{\ref{AET.BA.c3}} applied with $\frac{\veps}{2}$.
The proof {\bf (ii)} that $\Gamma$ is lower semi-continuous \new{is} analogous, simpler, and left to the reader.
Steps {\bf (iii)} and {\bf (iv)} are identical and yield a continuous selector $\tilde{\bI} \new{:} \bF \to \bE$, \ie $\tilde{\bI}(F) \in \rmclos \Gamma(F)$ for all $F \in \bF$.
We then define $\bI(F) = \vvvert F \vvvert \cdot \tilde{\bI}( \vvvert F \vvvert^{-1} F)$ when $F \neq 0$ and $\bI(0) = 0$.
Conclusions (a), (b), and (d) easily follow.
The continuity of $\bI$ along convergent sequences whose limit is not zero routinely follows from that of $\tilde{\bI}$ whereas the continuity of $\bI$ along null sequences follows from (d).
\end{proof}

\refmod{
\begin{Remark}[The case of complex scalars]
Let $X$ be a complex vector space and $\calT$ a locally convex topology on $X$ with respect to scalars in $\C$.
Denote by $X_{\bbR}$ and $\calT_\R$ the same sets considered, respectively, as a real vector space and a locally convex topology on $X_\R$ with respect to scalars in $\R$. 
Sections \ref{sec.EUL}--\ref{sec.CT} apply verbatim to $X_{\bbR}$ and $\calT_\R$.  
Recall that a locally convex topology on $X_{\bbR}$ with respect to scalars in $\R$ is locally convex with respect to scalars in $\C$ if and only if it is invariant under $x \mapsto \lambda \cdot x$, for all $\lambda \in \C$ with $|\lambda| = 1$.
Therefore, if every member of a localizing family $\calC$ is balanced with respect to scalars in $\C$ then $\lambda \cdot \calC = \calC$ for all scalars $\lambda \in \C$ of modulus 1, whence, the uniqueness in Theorem \ref{EUL.4} gives $\lambda \cdot (\calT_\R)_{\calC} = (\calT_\R)_{\lambda\cdot \calC} = (\calT_\R)_{\calC}$, since $\lambda \cdot (\calT_\R)_{\calC}$ is a localization of $\calT_\R$ by $\lambda \cdot \calC$.
Accordingly, $(\calT_\R)_{\calC}$ is a locally convex topology with respect to complex scalars.
This will occur in all our applications, each $C_k$ being the intersection of a subspace with a ball of the localizing seminorm $\lno \cdot \rno$, that is provided we assume that the $X_k$ are $\C$-linear subspaces and $\lno \cdot \rno$ is a $\C$-seminorm.
\par
Duality causes no difficulty: the map $f \mapsto \operatorname{Re} f$ is a bijection from the complex dual of $X$ onto the real dual of $X_{\bbR}$ (its inverse sends the $\R$-linear functional $u$ to
$x \mapsto u(x) - i\cdot u(i \cdot x)$), and since each
$X_k \cap \{ u : \lno u \rno \leq 1 \}$ is balanced,
$| \la u, f \ra | = \sup_\theta \operatorname{Re} \,\la e^{i\theta}u , f \ra$, so the seminorms $\la \vvvert \cdot \vvvert_k \ra_k$ of Theorem \ref{AET.FR} are unchanged under this bijection.  
Also, the characterisations in Theorem \ref{CCC.5},
Corollary \ref{FD.7}, and Theorem \ref{CT.1}
hold for $\C$-linear functionals.
\par 
In the complex version of Corollary \ref{AET.BA}, the right inverse in conclusion \ref{AET.BA.c4} may be chosen to be $\C$-homogeneous. The normalization in its proof first gives a continuous right inverse $\bI$ that is homogeneous for positive real scalars and satisfies
\[
\bq(\bI(F)) \leq (1+\veps)\cdot\bcH\cdot\vvvert F\vvvert.
\]
Here $\bD$ is $\C$-linear and $\bq$ is a complex norm. Define
\[
\Gamma(F)=\{v\in\bE:\bD(v)=F\text{ and }
\bq(v)\leq(1+\veps)\cdot\bcH\cdot\vvvert F\vvvert\}.
\]
These sets are closed and convex, and $\Gamma(\lambda F)=\lambda\Gamma(F)$ whenever $|\lambda|=1$. We set
\[
\bI^\sharp(F)=\frac{1}{2\pi}\int_0^{2\pi}
 e^{-i\theta}\bI(e^{i\theta}F)\,d\theta.
\]
This $\bE$-valued Riemann integral exists because $\bE$ is complete. Each value of the integrand belongs to $\Gamma(F)$, so its average also belongs to $\Gamma(F)$. Consequently, $\bD(\bI^\sharp(F))=F$ and the same norm estimate holds. The map $\bI^\sharp$ is continuous: if $F_j\to F$, joint continuity and compactness of the unit circle give
\[
\sup_{\theta\in[0,2\pi]}
\bq\bigl(\bI(e^{i\theta}F_j)-\bI(e^{i\theta}F)\bigr)\longrightarrow0.
\]
Finally, invariance of the integral under rotation, together with positive real homogeneity and $\bI^\sharp(0)=0$, gives
$\bI^\sharp(\lambda F)=\lambda\bI^\sharp(F)$ for all $\lambda\in\C$.
\end{Remark}
}

\section{Abstract compactness theorem}
\label{sec.CT}

\begin{Theorem}
\label{CT.1}
Assume that:
\begin{hypenum}
\item \label{CT.1.h1} $X[\calT]$ is a locally convex topological vector space whose topology is generated by a filtering family of seminorms $\la \|\cdot\|_i \ra_{i \in I}$;
\item \label{CT.1.h2} $\lno \cdot \rno$ is a seminorm on $X$ and $\calT$-lower-semicontinuous;
\item \label{CT.1.h3} $\la X_k \ra_k$ is a non-decreasing sequence of $\calT$-closed linear subspaces whose union is $X$;
\item \label{CT.1.h4} $C_k = X_k \cap \{ u : \lno u \rno \leq k \}$ for each $k \in \N$;
\item \label{CT.1.h5} $\calT \hel C_k$ is compact for each $k \in \N$;
\item \label{CT.1.h6} $\calC = \{ C_k : k \in \N\}$.
\end{hypenum}
Let $\calF \subset X[\calT_\calC]^*$ be $\beta(X^*,X)$-closed.
The following are equivalent.
\begin{abcenum}
\item \label{CT.1.a1} $\calF$ is $\beta(X^*,X)$-compact.
\item \label{CT.1.a2} $(\forall \veps > 0)(\exists i \in I)(\exists \theta > 0)(\forall x \in X_{\lceil \veps^{-1} \rceil})(\forall f \in \calF): |f(x)| \leq \theta \cdot \|x\|_i + \veps \cdot \lno x \rno$.
\end{abcenum}
\end{Theorem}

In the proof below, there is no restriction to assume that $\calF \neq \emptyset$.

\begin{proof}[Proof that \refmod{\ref{CT.1.a1}} $\Rightarrow$ \refmod{\ref{CT.1.a2}}]
{\bf (i)}
Recall \ref{LCTVS.7} and \ref{LCTVS.11} that the topology $\beta(X^*,X)$ is generated by the sequence of seminorms $\la p_{C_k} \ra_k$ and, therefore, is metrized by the metric $d$, where $d(f,g) = \sum_k \frac{p_{C_k}(f-g)}{2^k \cdot (1+p_{C_k}(f-g))}$, $f,g \in X[\calT_\calC]^*$.
\par 
{\bf (ii)}
Let $\veps > 0$.
Choose a positive integer $k$ such that $\lceil \veps^{-1} \rceil \leq k$.
Choose $\eta > 0$ small enough for $\eta \cdot 2^k < (1-\eta \cdot 2^k) \cdot \frac{\veps}{2}$.
The set $\calF$ being $\beta(X^*,X)$-compact and non-empty, there are $f_1,\ldots,f_N \in \calF$ such that for all $f \in \calF$ there exists $n \in \{1,\ldots,N\}$ with $d(f,f_n) < \eta$.
In particular, $\frac{p_{C_k}(f-f_n)}{2^k \cdot (1+p_{C_k}(f-f_n))} < \eta$ from which follows that $p_{C_k}(f-f_n) < \frac{\veps}{2}$, by the choice of $\eta$.
By definition of $p_{C_k}$, we infer that $|f(x)-f_n(x)| \leq \frac{\veps}{2} \cdot \lno x \rno$ for all $x \in X_k$.
\par 
{\bf (iii)}
For each $n \in \{1,\ldots,N\}$, it ensues from Theorem \ref{CCC.5} that there exist $i_n \in I$ and $\theta_n > 0$ such that $|f_n(x)| \leq \theta_n \cdot \|x\|_{i_n} + \frac{\veps}{2} \cdot \lno x \rno$ for all $x \in X_{\lceil \veps^{-1} \rceil}$.
Since $\la \|\cdot\|_i \ra_{i \in I}$ is filtering, there exists $i \in I$ such that $\max \{ \|x\|_{i_n} : n =1,\ldots,N\} \leq \|x\|_i$ for all $x \in X$.
We also define $\theta = \max \{ \theta_n : n =1,\ldots,N \}$.
\par 
{\bf (iv)}
Let $x \in X_{\lceil \veps^{-1} \rceil}$ and $f \in \calF$.
Choose $n \in \{1,\ldots,N\}$ corresponding to $f$ as in {\bf (ii)}.
Upon noticing that $x \in X_k$, we infer from {\bf (ii)} and {\bf (iii)} that
\begin{equation*}
|f(x)| \leq  |f_n(x)| + |f_n(x)-f(x)| \leq \theta_n \cdot \|x\|_{i_n} + \frac{\veps}{2} \cdot \lno x \rno + \frac{\veps}{2} \cdot \lno x \rno \leq \theta \cdot \|x\|_i + \veps \cdot \lno x \rno.
\end{equation*}
\end{proof}

\begin{proof}[Proof that \refmod{\ref{CT.1.a2}} $\Rightarrow$ \refmod{\ref{CT.1.a1}}]
\par 
{\bf (i)}
\new{First}, we show that $\calF$ is pointwise bounded, \ie $\sup \{ |f(x)| : f \in \calF \} < +\infty$ for all $x \in X$.
Let $x \in X \setminus \{0\}$.
There exists $k \in \N$ such that $x \in X_k$.
Choose $\veps > 0$ such that $k \leq \veps^{-1}$.
Let $i \in I$ and $\theta > 0$ be associated with $\veps$ in \refmod{\ref{CT.1.a2}}.
For all $f \in \calF$ we have $|f(x)| \leq \theta \cdot \|x\|_i + \veps \cdot \lno x \rno$, since $x \in X_{\lceil \veps^{-1} \rceil}$.
\par 
{\bf (ii)}
Fix a positive integer $k$.
We consider $\calF_k = \{ f|_{C_k} : f \in \calF \}$. 
For each $f \in \calF$, $f|_{C_k}$ is $\calT_\calC \hel C_k$-continuous, hence, $\calT \hel C_k$-continuous, by \ref{GO.1}\refmod{\ref{GO.1.2}}.
In other words, $\calF_k \subset C(C_k)$, where $C_k$ is equipped with its Hausdorff topology $\calT \hel C_k$.
We now show that $\calF_k$ is equicontinuous, \ie
\begin{equation*}
(\forall \veps > 0)(\forall x \in C_k)(\exists V_x \in \calV_x(\calT \hel C_k))(\forall y \in V_x)(\forall f \in \calF): |f(x)-f(y)| < \veps,
\end{equation*}
where $\calV_x(\calT \hel C_k)$ is the set of $\calT \hel C_k$-neighbourhoods of $x$.
To this end, fix $\veps > 0$, notice that there is no restriction to assume that $k \leq \frac{1}{\veps}$, and define $\hat{\veps} = \frac{\veps}{4 \cdot k}$.
Let $i \in I$ and $\theta > 0$ be associated with $\hat{\veps}$ in \refmod{\ref{CT.1.a2}}.
Fix $x \in C_k$, define $V_x = C_k \cap \left\{ y : \| y-x \|_i < \frac{\veps}{2 \cdot \theta} \right\}$ and observe that this is a $\calT \hel C_k$-neighbourhood of $x$.
If $y \in V_x$ and $f \in \calF$ then $y-x \in X_k$, whence, $y-x \in X_{\new{\lceil \hat{\veps}^{-1} \rceil}}$, so that, by \refmod{\ref{CT.1.a2}}, we have
\begin{equation*}
|f(y)-f(x)| = |f(y-x)| \leq \theta \cdot \|y-x\|_i + \hat{\veps} \cdot \lno y-x \rno \leq \theta \cdot \frac{\veps}{2 \cdot \theta} + \hat{\veps} \cdot 2k < \veps.
\end{equation*}
\par 
As $\calT \hel C_k$ is a compact Hausdorff topological space, by assumption, Ascoli's theorem \cite[A5]{RUDIN} applies in $C(C_k)$.
Thus, each sequence in $\calF_k$ admits a subsequence that uniformly converges to some member of $C(C_k)$.
\par 
{\bf (iii)}
Let $\la f_n \ra_n$ be a sequence in $\calF$.
Applying {\bf \new{(ii)}} inductively on $k$ and referring to a diagonal argument we find a subsequence $\la f_{\psi(n)} \ra_n$ of $\la f_n \ra_n$ and a sequence $\la g_k \ra_k$ such that $g_k \in C(C_k)$ and $\la f_{\psi(n)}|_{C_k} \ra_{n}$ converges uniformly to $g_k$ on $C_k$ for all $k$.
In particular, $g_j = g_k|_{C_j}$ whenever $j \leq k$ so that $f : X \to \R$ is well-defined by the requirement that $f(x) = g_k(x)$ if $x \in \new{C_k}$.
The pointwise convergence of $\la f_{\psi(n)} \ra_n$ to $f$ readily implies that $f$ is linear.
Moreover, since $f|_{C_k} = g_k$ is $\calT \hel C_k$-continuous for all $k$, it follows from Definition \ref{EUL.3} that $f$ is $\calT_\calC$-continuous.
Finally, we observe that $f$ is the limit of $\la f_{\psi(n)} \ra_n$ with respect to the $\beta(X^*,X)$ topology.
Indeed, $p_{C_k}(f-f_{\psi(n)}) = \| f|_{C_k} - f_{\psi(n)}|_{C_k}\|_{\infty} = \| g_k - f_{\psi(n)}|_{C_k}\|_{\infty}$.
\par 
{\bf (iv)}
It follows from {\bf (iii)} that $\calF$ is sequentially $\beta(X^*,X)$-compact.
Upon recalling that $\beta(X^*,X)$ is metrizable, this completes the proof.
\end{proof}

\section{One example: The divergence of continuous vector fields}
\label{sec.OE}

\begin{Empty}[The localized topology $\calM_{\TV}$]
\label{OE.1}
Here, $m \geq 2$ is an integer.
We introduce the following notation:
\begin{itemize}
\item $X = BV_c(\Rm)$ is the vector space consisting of equivalence classes (with respect to equality\footnote{$\ssfL^m$ is the Lebesgue measure in $\Rm$.} $\ssfL^m$-a.e.) of functions of bounded variation $u : \Rm \to \R$ such that the support of $\ssfL^m \hel u$ -- which we abbreviate as $\rmspt u$ -- is compact;
\item $X_k = BV_c(\Rm) \cap \{ u : \rmspt u \subset B(0,k) \}$, $k \in \N$;
\item $\calM$ is the topology of convergence in the mean on $X$, \ie the (locally convex vector) topology corresponding to the norm $\|u\| = \|u\|_{L_1}$;
\item $\lno u \rno = \|Du\|(\Rm)$ is the total variation of $u \in BV_c(\Rm)$;
\item $C_k = X_k \cap \{ u : \lno u \rno \leq k \} = BV_c(\Rm) \cap \{ u : \rmspt u \subset B(0,k) \text{ and } \|Du\|(\Rm) \leq k \}$, $k \in \N$;
\item $\TV = \{ C_k : k \in \N \}$.
\end{itemize}
\par 
For functions of bounded variation we refer to \cite{EVANS.GARIEPY.2}.
For instance, $\|Du\|(\Rm)$ is the real number defined in \cite[definition 5.1(i)]{EVANS.GARIEPY.2} or \cite[remarks (ii) and (iii) p.197]{EVANS.GARIEPY.2}.
It is clear from the definition that $u \mapsto \lno u \rno$ is a seminorm on $BV_c(\Rm)$.
It is lower-semicontinuous with respect to the topology $\calM$, according to \cite[theorem 5.2 p.199]{EVANS.GARIEPY.2}.
Since $X_k$ is clearly $\calM$-closed, it follows that $C_k$ is a $\calM$-closed subset of $X$ for each $k \in \N$.
Consequently, $\TV$ is a localizing family (recall, \eg the preliminaries to the proof of \ref{CCC.4}).
Therefore, the localized topology $\calM_{\TV}$ is defined on $BV_c(\Rm)$, by \ref{EUL.4}.
\begin{enumerate}
\item[(A)] {\it A sequence $\la u_j \ra_j$ in $BV_c(\Rm)$ converges to $u \in BV_c(\Rm)$ with respect to the topology $\calM_{\TV}$ if and only if the following three conditions are satisfied:
\begin{enumerate}
\item[(i)] There is a compact set $K \subset \Rm$ such that $\rmspt u_j \subset K$ for all $j$;
\item[(ii)] $\lim_j \|u-u_j\|_{L_1} = 0$;
\item[(iii)] $\sup_j \|Du_j\|(\Rm) < + \infty$.
\end{enumerate}
}
\end{enumerate}
\par 
{\it Proof.}
This a reformulation of \ref{CCC.1}\refmod{\ref{CCC.1.c1}} or \ref{CCC.4}\refmod{\ref{CCC.4.c1}}.\cqfd
\par 
Note that this differs from the notion of ``strict convergence'' considered in \cite[definition 3.14]{AMBROSIO.FUSCO.PALLARA} since, there, condition (iii) is strengthened to $\lim_j \|Du_j\|(\Rm) = \|Du\|(\Rm)$.
\par 
\begin{enumerate}
\item[(B)] {\it Each $C_k$ is $\calM$-compact.}
\end{enumerate}
\par 
{\it Proof.}
Fix a positive integer $k$.
Abbreviate $U_k = \Rm \cap \{ x : |x| < k+1 \}$ which is open, bounded, has Lipschitz boundary, and $\|Du\|(\Rm) = \|Du\|(U_k)$ for all $u \in X_k$.
Moreover, $\|u\|_{L_1} \leq \bc(k,m) \cdot \|Du\|(\Rm)$ for all $u \in X_k$, by the Poincar\'e \cite[theorem 5.10(i) p.216]{EVANS.GARIEPY.2} and H\"older inequalities.
Accordingly, $C_k$ is $\calM$-compact, by the $BV$-compactness theorem \cite[theorem 5.5]{EVANS.GARIEPY.2} applied in $U_k$.\cqfd
\par 
\begin{enumerate}
\item[(C)] {\it A subset of $BV_c(\Rm)$ is $\calM_{\TV}$-bounded if and only if it is contained in some $C_k$.}
\end{enumerate}
\par 
{\it Proof.}
This is a reformulation of \ref{CWC.2}\refmod{\ref{CWC.2.c3}}.\cqfd
\begin{enumerate}
\item[(D)] {\it $BV_c(\Rm)[\calM_{\TV}]$ is semireflexive.}
\end{enumerate}
\par 
{\it Proof.}
Each $C_k$ is clearly absolutely convex.
Thus, it follows from (B) and (C) above that Theorem \ref{SD.5} applies.\cqfd
\end{Empty}

\begin{Empty}[Strong charges]
\label{OE.2}
Following the vocabulary of \cite{DEP.PFE.06b}, an $\calM_{\TV}$-continuous linear form $F : BV_c(\Rm) \to \R$ is called a {\em strong charge}. 
Thus, strong charges are the members of the dual space $BV_c(\Rm)[\calM_{\TV}]^*$ which we abbreviate as $\SCH(\Rm)$.
We let $\beta$ be the (locally convex) strong topology on $\SCH(\Rm)$, \ie the topology of uniform convergence on $\calM_{\TV}$-bounded subsets of $BV_c(\Rm)$.
It follows from \ref{OE.1}(C) that $\beta$ is associated with the filtering sequence of seminorms $\la \vvvert \cdot \vvvert_k \ra_k$ defined by
\begin{equation*}
\begin{split}
\vvvert F \vvvert_k & = \sup \left\{ |\la u , F \ra| : u \in C_k \right\} \\
& = \sup \left\{ |\la u , F \ra| : \rmspt u \subset B(0,k) \text{ and } \|Du\|(\Rm) \leq k \right\}.
\end{split}
\end{equation*}
\par 
\begin{enumerate}
\item[(A)] {\it Let $F : BV_c(\Rm) \to \R$ be linear. The following are equivalent:
\begin{enumerate}
\item[(a)] $F$ is a strong charge.
\item[(b)] $(\forall \veps > 0)(\exists \delta > 0)(\forall u \in BV_c(\Rm))$:
\begin{equation*}
\big[ \rmspt u \subset B(0,\lceil \veps^{-1} \rceil) \text{ and } \|u\|_{L_1} \leq \delta
\\
\text{ and } \|Du\|(\Rm) \leq \veps^{-1} \big] \Rightarrow |\la u,F \ra| \leq \veps.
\end{equation*}
\item[(c)] $(\forall \veps > 0)(\exists \theta > 0)(\forall u \in BV_c(\Rm))$:
$$
\rmspt u \subset B(0,\lceil \veps^{-1} \rceil) \Rightarrow |\la u , F \ra| \leq \theta \cdot \|u\|_{L_1} + \veps \cdot \|Du\|(\Rm).
$$
\end{enumerate}
}
\end{enumerate}
\par 
{\it Proof.}
This is a reformulation of \ref{CCC.5}.\cqfd
\par 
Together with \ref{OE.1}(D), the following shows that $BV_c(\Rm)$ ``is'' the dual of a Fr\'echet space, namely is identified with the dual of $\SCH(\Rm)[\beta]$ via the evaluation map.
\begin{enumerate}
\item[(B)] {\it $\SCH(\Rm)[\beta]$ is a Fr\'echet space.}
\end{enumerate}
\par 
{\it Proof.}
This follows from \ref{FD.2} and \ref{OE.1}(C).\cqfd
\begin{enumerate}
\item[(C)] {\it If $v \in C(\Rm;\Rm)$ then the linear functional $\rmdiv v : BV_c(\Rm) \to \R$ defined by the formula
$$
\la u , \rmdiv \new{v} \ra = - \sum_{i=1}^m \int_{\Rm} v_i \, d\partial_iu
$$
is a strong charge.
}
\end{enumerate}
\par 
{\it Proof.}
In view of (A)(c) above, the proof is similar to that given in the introduction in the vicinity of the displayed inequality \eqref{eq.INTRO.3}.\cqfd
\par 
\begin{enumerate}
\item[(D)] {\it If $f \in L_{m,\rmloc}(\Rm)$ then the linear functional $F_f : BV_c(\Rm) \to \R$ defined by $\la u , F_f \ra = \int_{\Rm} f \cdot u \,d\ssfL^m$ is a strong charge.}
\end{enumerate}
\par 
{\it Proof.}
We first note that $F_f$ is defined, \ie $f \cdot u \in L_1(\Rm)$ for all $u \in BV_c(\Rm)$.
Indeed, letting $1^*$ be the H\"older conjugate of $m$ we have $u \in L_{1^*}(\Rm)$, by Poincar\'e's inequality \cite[theorem 5.10(i)]{EVANS.GARIEPY.2}, hence, $f \cdot u = (f \cdot \ind_{\rmspt u}) \cdot u$ is summable, by H\"older's inequality, since $f \cdot \ind_{\rmspt u} \in L_m(\Rm)$.
\par 
Let $\veps > 0$ and assume that $\rmspt u \subset B(0,\lceil \veps^{-1} \rceil) =: B$.
There exists $\theta > 0$ such that $\int_{B \cap \{ |f| \geq \theta\}} |f|^m \,d\ssfL^m < \new{\left( \frac{\veps}{\bc(m)} \right)^m}$\new{, where $\bc(m)$ is the constant in Poincar\'e's inequality}.
Thus,
\begin{multline*}
\left| \int_{\Rm} f \cdot u \, d\ssfL^m \right| \leq \left| \int_{B \cap \{ |f| < \theta \}} f \cdot u \, d\ssfL^m \right| + \left| \int_{B \cap \{ |f| \geq \theta\}} f \cdot u \, d\ssfL^m \right| \\ < \theta \cdot \|u\|_{L_1} + \left( \int_{B \cap \{ |f| \geq \theta\}} |f|^m \,d\ssfL^m \right)^{\frac{1}{m}} \cdot \left( \int_{\Rm} |u|^{1^*} \, d\ssfL^m\right)^{\frac{1}{1^*}} \\ \leq \theta \cdot \|u\|_{L_1} + \new{\veps} \cdot \|Du\|(\Rm),
\end{multline*}
where we have used Poincar\'e's inequality again.
The conclusion follows from (A)(c) above.\cqfd
\end{Empty}

The vector space $C(\Rm;\Rm)$ is equipped with its structure of a Fr\'echet space associated with the increasing sequence $\la \|\cdot\|_{\infty,B(0,k)}\ra_k$ of seminorms defined by
$$
\|v\|_{\infty,B(0,k)} = \sup \{ |v(x)|_2 : x \in B(0,k) \}.
$$
The following can be interpreted as an existence and optimal regularity result for the PDE $\rmdiv v = F$.

\begin{Theorem}
\label{OE.3}
The bounded linear operator $\rmdiv : C(\Rm;\Rm) \to \SCH(\Rm)$ is surjective.
Moreover, for every $k \in \N$ and $\veps > 0$ there exists a continuous map 
$$
\bv : \SCH(\Rm) \cap \{ F : \vvvert F \vvvert_k > 0 \} \to C(\Rm;\Rm)
$$ 
such that $\rmdiv(\bv(F)) = F$ and $\|\bv(F)\|_{\infty,B(0,k)} \leq (1+ \veps) \cdot \vvvert F \vvvert_k$ for all $F$.
\end{Theorem}

\begin{proof}
We shall show that this is a consequence of the abstract existence theorem \ref{AET.FR} applied with $\bE = C(\Rm;\Rm)$, $\bq_k = \|\cdot\|_{\infty,B(0,k)}$, $X = BV_c(\Rm)$, $\calT_\calC = \calM_{\TV}$, $\lno \cdot \rno = \|D \cdot\|(\Rm)$, and $\bD = \rmdiv$. 
\par 
Hypotheses \ref{AET.FR}\refmod{\ref{AET.FR.h1}}, \ref{AET.FR}\refmod{\ref{AET.FR.h2}}, and \ref{AET.FR}\refmod{\ref{AET.FR.h3}} have already been checked.
The important hypothesis \ref{AET.FR}\refmod{\ref{AET.FR.h4}} is \ref{OE.1}(B).
The map $(u,v) \mapsto \la u , \rmdiv v\ra$ is linear in both $u$ and $v$, which is hypothesis \ref{AET.FR}\refmod{\ref{AET.FR.h5}}.
Hypothesis \ref{AET.FR}\refmod{\ref{AET.FR.h6}} is \ref{OE.2}(C).
Hypothesis \ref{AET.FR}\refmod{\ref{AET.FR.h7}} readily holds with $\bcG(k)=1$ \new{and $\kappa(k)=k$}, indeed, if $v \in C(\Rm;\Rm)$, $u \in BV_c(\Rm)$, and $\rmspt u \subset B(0,k)$ then 
$$
|\la u , \rmdiv v\ra| \leq \|v\|_{\infty,B(0,k)} \cdot \|Du\|(\Rm).
$$
Hypothesis \ref{AET.FR}\refmod{\ref{AET.FR.h8}} holds with $\bcH(k)=1$, according to the definition of the total variation.
Indeed, if $u \in BV_c(\Rm)$ and $\rmspt u \subset B(0,k)$ then
\begin{equation*}
\begin{split}
\|Du\|(\Rm) & = \sup \{ | \la u , \rmdiv v \ra| : v \in C^\infty_c(\Rm;\Rm) \text{ and } \|v\|_\infty \leq 1 \} \\
& = \sup \{ | \la u , \rmdiv v \ra| : v \in C(\Rm;\Rm) \text{ and } \|v\|_{\infty,B(0,k)} \leq 1 \},
\end{split}
\end{equation*}
by smoothing.
It remains to check that hypothesis \ref{AET.FR}\refmod{\ref{AET.FR.h9}} is satisfied. 
Let $u \in BV_c(\Rm)$ and $k \in \N$.
Recall that $B(0,k)$ is a closed ball.
Assume that for every open ball $U \subset \Rm$ such that $(\rmclos U) \cap B(0,k) = \emptyset$ we have $\|Du\|(U) = 0$.
Then for such $U$ there exists $c_U \in \R$ such that $u$ equals $c_U$ a.e. in $U$, by the constancy theorem for distributions applied in $U$.
Since $\Rm \setminus B(0,k)$ is connected, we infer that all $c_U$ are equal to some $c \in \R$.
As $\rmspt u$ is compact, it ensues that $c = 0$.
Therefore, $\rmspt u \subset B(0,k)$, \ie $u \in X_k$ showing that hypothesis \ref{AET.FR}\refmod{\ref{AET.FR.h9}} holds.
\par 
The first sentence in the \refmod{statement} of the theorem is conclusion \refmod{\ref{AET.FR.c2}} of Theorem \ref{AET.FR} and the second sentence is conclusion \refmod{\ref{AET.FR.c4}} of \ref{AET.FR}.
\end{proof}

\begin{Empty}[Awkwardness of $\calM_{\TV}$]
\label{OE.4}
If $Y$ is any Hausdorff topological space and $f : BV_c(\Rm) \to Y$ is sequentially $\calM_{\TV}$-continuous then $f$ is $\calM_{\TV}$ continuous, by (A) below and \ref{PR.2}(E), yet if $A \subset BV_c(\Rm)$ is $\calM_{\TV}$-dense then it may not be sequentially $\calM_{\TV}$-dense, by (B) below and \ref{PR.3}(C).
\begin{enumerate}
\item[(A)] {\it $BV_c(\Rm)[\calM_{\TV}]$ is sequential.}
\end{enumerate}
\par 
{\it Proof.} 
This follows from \ref{CWC.6}\refmod{\ref{CWC.6.c1}}, \ref{OE.1}(B), and the fact that each topology $\calM \hel C_k$ is sequential (being metrizable).\cqfd
\begin{enumerate}
\item[(B)] {\it $BV_c(\Rm)[\calM_{\TV}]$ is not Fr\'echet-Urysohn.}
\end{enumerate}
\par 
{\it Proof.}
We show this very explicitly by means of \ref{PR.3}(D) (which applies, by (A) above).
Let $\new{\la r_k \ra_k}$ be any null-sequence of positive real numbers, $\new{\theta_{j,k} = r_k^{1-m} \cdot j}$, and $u_{j,k} = \theta_{j,k} \cdot \new{\ind_{B(0,r_k)}}$.
Observe\footnote{$\balpha(m)$ is the $\ssfL^m$-measure of the unit Euclidean ball in $\Rm$.} that $\|Du_{j,k}\|(\Rm) = \new{m \cdot \balpha(m) \cdot j}$ and $\|u_{j,k}\|_{L_1} = \new{\balpha(m) \cdot j \cdot r_k}$ for all $(j,k) \in \N \times \N$.
Clearly all $u_{j,k}$ are supported in a common compact set.
Therefore, $(\calM_{\TV}) \lim_{\new{k}} u_{j,k} = 0$ for all $\new{j} \in \N$, by \ref{OE.1}(A).
Yet, for all pairs of increasing functions $j : \N \to \N$ and $k : \N \to \N$ the sequence $\la u_{j(n),k(n)} \ra_n$ is not $\calM_{\TV}$-convergent, by \ref{OE.1}(A), since $\sup_n \|Du_{j(n),k(n)}\|(\Rm) = + \infty$.\cqfd
\begin{enumerate}
\item[(C)] {\it $BV_c(\Rm)[\calM_{\TV}]$ is not barrelled.}
\end{enumerate}
\par 
{\it Proof.}
We consider the bounded Borel-measurable vector field $v : \Rm \to \Rm$ defined by
\begin{equation*}
v(x) = \begin{cases}
\frac{x}{|x|_2} & \text{if } x \neq 0 \\
0 & \text{if } x = 0.
\end{cases}
\end{equation*} 
Note that $F : BV_c(\Rm) \to \R$ is well-defined by means of the formula
\begin{equation*}
F(u) = - \sum_{i=1}^m \int_{\Rm} v_i \, d\partial_i u .
\end{equation*}
Let $\la \phi_\veps \ra_{\veps > 0}$ be a (compactly supported) smooth approximation of the identity, $\veps_j \downarrow 0$ as $j \uparrow \infty$, and $v_j = \phi_{\veps_j}*v \in C(\Rm;\Rm)$.
Let $F_j = \rmdiv v_j \in \SCH(\Rm)$ (recall \ref{OE.2}(C)), $j \in \N$, and observe that $\la u , F_j \ra \to F(u)$ as $j \to \infty$ for all $u \in BV_c(\Rm)$, by the dominated convergence theorem, since $\|v_j\|_\infty \leq \ind_{\Rn} \in L_1(\|Du\|)$, $v_j(x) \to v(x)$ as $j \to \infty$ for all $x \neq 0$, and $\|Du\|(\{0\})=0$ (as $\|Du\| \ll \ssfH^{m-1}$, see \eg \cite[lemma 3.76 p.170]{AMBROSIO.FUSCO.PALLARA}).
In particular, $F$ is linear.
In order to complete the proof, in view of \ref{PR.8}\refmod{\ref{PR.8.a3}}, it suffices to show that $F$ is not a strong charge.
With the same notation as in the proof of (B) above, we observe that $\new{\lim_k u_{1,k} = 0}$ with respect to $\calM_{\TV}$, yet
\begin{equation*}
F(\new{u_{1,k}}) = \new{\theta_{1,k}} \int_{\rmbdry B(0,\new{r_k})} \frac{x}{|x|_2} \ip \vec{n}_{\rmbdry B(0,\new{r_k})}(x) \, d\ssfH^{m-1}(x) = m \cdot \balpha(m) \not \to 0
\end{equation*}
as $\new{k} \to \infty$.\cqfd
\begin{enumerate}
\item[(D)] {\it $BV_c(\Rm)[\calM_{\TV}]$ is not bornological.}
\end{enumerate}
\par 
{\it Proof.}
Let $F$ be as in the proof of (C) above and recall that $F$ is linear and not $\calM_{\TV}$-continuous.
In view of \ref{PR.13}\refmod{\ref{PR.13.a2}} applied to the seminorm $|F|$, \refmod{it} suffices to show that $|F|$ sends $\calM_{\TV}$-bounded subsets $B$ of $BV_c(\Rm)$ to bounded subsets of $\R$.
By \ref{OE.1}(C), it is enough to prove this for $B = C_k$.
If $u \in C_k$ then 
\begin{equation*}
|F(u)| = \left| \int_{\Rm} v \ip d(Du) \right| \leq \|v\|_\infty \cdot \|Du\|(\Rm) \leq k. 
\end{equation*}\cqfd
\end{Empty}

\begin{Empty}[Strong charges as distributions]
\label{OE.5}
At the beginning of the introduction to this paper, we considered $\rmdiv v$, $v \in C(\Rm;\Rm)$, as a distribution satisfying a certain continuity property \ref{eq.INTRO.3}.
In the course of proving Theorem \ref{OE.3} we had to replace the domain $\calD(\Rm)$ of $\rmdiv v$ by $BV_c(\Rm)$ in order to gain a compactness property that was instrumental in proving the semireflexivity theorem \ref{OE.1}(D) allowing, in turn, for the abstract existence theorem \ref{AET.FR} to apply.
Here, we address the question whether the final result \ref{OE.3} can be formulated in terms of distributions only, forgetting altogether about $BV_c(\Rm)$.
In order to state the positive answer we let $\widetilde{\SCH}(\Rm)$ be the set of linear forms $F : \calD(\Rm) \to \R$ satisfying the following condition:
\begin{equation*}
(\forall \veps > 0)(\exists \theta > 0)(\forall \vphi \in \calD(\Rm)):
\rmspt \vphi \subset B(0,\lceil \veps^{-1} \rceil) \Rightarrow |\la \vphi , F \ra | \leq \theta \cdot \|\vphi\|_{L_1} + \veps \cdot \| \nabla \vphi \|_{L_1}.
\end{equation*}
\begin{enumerate}
\item[(A)] {\it $\widetilde{\SCH}(\Rm)$ is a vector space of which each member is a distribution in $\Rm$.}
\end{enumerate}
\par 
{\it Proof.}
Trivial.\cqfd
\begin{enumerate}
\item[(B)] {\it The map $\rmrestr : \SCH(\Rm) \to \widetilde{\SCH}(\Rm) : F \mapsto F|_{\calD(\Rm)}$ is a linear bijection.}
\end{enumerate}
\par 
{\it Proof.}
{\bf (i)}
We note that $\calD(\Rm) \subset BV_c(\Rm)$ and that $\|D\vphi\|(\Rm) = \|\nabla \vphi\|_{L_1}$ for all $\vphi \in \calD(\Rm)$.
Therefore, $\rmrestr$ is well-defined.
It is readily linear.
\par 
{\bf (ii)}
Here, we observe that $\calD(\Rm)$ is uniformly sequentially $\calM_{\TV}$-dense in $BV_c(\Rm)$, recall \ref{FD.5}.
We can take $\gamma(k)=k+1$ and $\Gamma = 1$ in the definition.
Indeed, let $\la \phi_{\veps} \ra_{\veps > 0}$ be a smooth approximation of the identity with $\rmspt \phi_{\veps} \subset B(0,\veps)$ and $\veps_j \downarrow 0$ as $j \uparrow \infty$.
Assume that $\veps_0 \leq 1$ so that $\rmspt \phi_{\veps_j} \subset B(0,1)$ for all $j$.
Given $u \in BV_c(\Rm)$ and $k \in \N$ such that $\rmspt u \subset B(0,k)$ we define $u_j = \phi_{\veps_j} * u$ and we note that $u_j \in \calD(\Rm)$, $\rmspt u_j \subset B(0,k+1)$, $\|u-u_j\|_{L_1} \to 0$ as $j \to \infty$, and $\|Du_j\|(\Rm) \leq \|Du\|(\Rm)$ for all $j$ (the latter follows from the definition of total variation and properties of convolution product). 
\par 
{\bf (iii)}
It follows from {\bf (ii)} and Theorem \ref{FD.6} that each $F \in \widetilde{\SCH}(\Rm)$ admits a linear extension $\hat{F} : BV_c(\Rm) \to \R$ which is a strong charge.
We note that if $\check{F}$ is any other strong charge extending $F$ then $\check{F}=\hat{F}$, since they coincide on $\calD(\Rm)$ which is $\calM_{\TV}$-dense (recall \ref{FD.5}).
Define $\rmext (F) \in \SCH(\Rm)$ to be the unique $\calM_{\TV}$-continuous linear extension of $F$.
It is plain that $\rmrestr \circ \rmext = \rmid_{\widetilde{\SCH}}$.
Moreover, by uniqueness, $\rmext \circ \rmrestr = \rmid_{\SCH}$.
It is now trivial that $\rmrestr$ is bijective and $\rmext$ is linear.\cqfd
\par
We then recover the main result of \cite{DEP.PFE.06b}:
\begin{enumerate}
\item[(C)] {\it Let $F$ be a distribution in $\Rm$. The following are equivalent.
\begin{enumerate}
\item[(a)] $F \in \widetilde{\SCH}(\Rm)$.
\item[(b)] There exists $v \in C(\Rm;\Rm)$ such that $F = \rmdiv v$.
\end{enumerate}
}
\end{enumerate}
\par 
{\it Proof.}
This is an immediate consequence of Theorem \ref{OE.3} and (B) above.\cqfd
\end{Empty}

\appendix
\numberwithin{Theorem}{section}

\section{Sequential vs. Fr\'echet-Urysohn topological spaces}
\label{app.A}

\begin{Empty}[Topological spaces]
Our notation for a topological space is $X[\calT]$ where $X$ is the underlying set and $\calT$ is the topology.
When only one topology is at \refmod{stake} (less often than not) we might write $X$ instead of $X[\calT]$.
All topological spaces considered in this paper are Hausdorff except for a little while in Section \ref{sec.EUL} until it is proved that $\calS$ is actually Hausdorff.
In case $Y \subset X$, recall our notation \ref{EUL.1} $\calT \hel Y$ for the subspace topology of $Y$.
The closure (resp. interior) of $A \subset X$ is denoted $\rmclos_\calT A$ (resp. $\rmint_\calT A$) or simply $\rmclos A$ (resp. $\rmint A$) when no confusion can occur.
\end{Empty}

\begin{Empty}[Sequential Spaces]
\label{PR.2}
Let $X$ be a Hausdorff topological space.
A set $A \subset X$ is called {\em sequentially closed} iff each sequence in $A$ that is convergent in $X$ converges to a point of $A$. 
A set $A \subset X$ is called {\em sequentially open} iff each sequence in $X$ converging to a point of $A$ is eventually in $A$. 
\begin{enumerate}
\item[(A)] {\it If $A \subset X$ is closed (resp. open) then it is sequentially closed (resp. sequentially open).}
\end{enumerate}
\par 
{\it Proof.}
Trivial.\cqfd
\par 
The converse is false in general. 
For instance, the unit sphere $A = \ell_1(\N) \cap \{ x : \|x\|_1 = 1 \}$ is sequentially weakly closed (\ie it is sequentially $\sigma(\ell_1,\ell_\infty)$-closed), since it is closed and $\ell_1(\N)$ has the Schur property \cite[Theorem 2.3.6]{ALBIAC.KALTON.2nd}, but not weakly closed, since any weak neighbourhood of the origin contains a finite-codimensional subspace, hence, intersects $A$.
\begin{enumerate}
\item[(B)] {\it If $A \subset X$ is sequentially closed (resp. sequentially open) then its complement $A^c$ is sequentially open (resp. sequentially closed).}
\end{enumerate}
\par 
{\it Proof.}
Trivial.\cqfd
\par 
Therefore, in a Hausdorff topological space $X$ every sequentially closed set is closed iff every sequentially open set is open.
In case $X$ satisfies these properties we call it {\em sequential}.
Of course each metrizable space is sequential. 
The examples of localized topologies studied in this paper are sequential but non-metrizable.
We now observe that the sequential character of $X$ is inherited by its open and closed subspaces.
\par 
\begin{enumerate}
\item[(C)] {\it Let $X[\calT]$ be a sequential Hausdorff topological space.
\begin{enumerate}
\item[(a)] If $C \subset X$ is closed (equivalently, if $C$ is sequentially closed) then $C[\calT \hel C]$ is sequential.
\item[(b)] If $O \subset X$ is open (equivalently, if $O$ is sequentially open) then $O[\calT \hel O]$ is sequential.
\end{enumerate}
}
\end{enumerate}
\par 
{\it Proof.}
Assume $A \subset C$ is sequentially closed in $C$.
Let $\la x_j \ra_j$ be a sequence in $A$ that converges to $x \in X$.
Then $x \in C$, since $C$ is closed in $X$.
Thus, $x \in A$, since $A$ is sequentially closed in $C$. 
By the arbitrariness of $x$ and $\la x_j \ra_j$, $A$ is sequentially closed in $X$, thus, closed in $X$, since $X$ is sequential.
Accordingly, $A$ is closed in $C$.
This proves (a).
The proof of (b), being similar to that of (a), is left to the reader.
\cqfd
\par 
Given a Hausdorff topological space $X[\calT]$ we let $\calT_\rmseq$ be the set consisting of the sequentially open subsets of $X$.
\begin{enumerate}
\item[(D)] {\it The following hold.
\begin{enumerate}
\item[(a)] $\calT \subset \calT_\rmseq$;
\item[(b)] $\calT_\rmseq$ is a Hausdorff topology on $X$;
\item[(c)] $X[\calT]$ is sequential iff $\calT = \calT_\rmseq$.
\end{enumerate}
}
\end{enumerate}
\par 
{\it Proof.}
None of these resist routine verification.\cqfd
\par 
Letting $X, Y$ be Hausdorff topological spaces and $f : X \to Y$ we say that {\em $f$ is sequentially continuous} iff $\lim_j f(x_j) = f(x)$ for every convergent sequence $\la x_j \ra_j$ in $X$ whose limit is $x$. 
If $f$ is continuous then it is sequentially continuous. 
The following characterizes sequential spaces.
\begin{enumerate}
\item[(E)] {\it Let $X$ be a Hausdorff topological space. The following are equivalent.
\begin{enumerate}
\item[(a)] $X$ is sequential.
\item[(b)] For every Hausdorff topological space $Y$ and every $f : X \to Y$ {\bf if} $f$ is sequentially continuous {\bf then} $f$ is continuous.
\end{enumerate}
}
\end{enumerate}
\par 
{\it Proof that $(a) \Rightarrow (b)$}.
Letting $C \subset Y$ be closed we observe that $f^{-1}(C)$ is sequentially closed: If $\la x_j\ra_j$ is a sequence in $f^{-1}(C)$ converging to $x \in X$ then $f(x) = \lim_j f(x_j) \in C$. 
Thus, $C$ is closed, since $X$ is sequential.
\par 
{\it Proof that $(b) \Rightarrow (a)$}.
Let $\calT$ be the topology of $X$ and let $\calT_\rmseq$ be as defined above. 
Consider the identity map $\rmid_X : X[\calT] \to X[\calT_\rmseq]$.
By definition of $\calT_\rmseq$, $\rmid_X$ is sequentially continuous.
Thus, $\rmid_X$ is continuous, by assumption (b), \ie $\calT_\rmseq = \calT$.\cqfd
\end{Empty}

\begin{Empty}[Fr\'echet-Urysohn spaces]
\label{PR.3}
Let $X$ be a Hausdorff topological space. We say that $X$ is {\em Fr\'echet-Urysohn} whenever the following holds: For every set $A \subset X$ and every $x \in \rmclos A$ there exists a sequence in $A$ converging to $x$. 
\begin{enumerate}
\item[(A)] {\it If $X$ is Fr\'echet-Urysohn then it is sequential.}
\end{enumerate}
\par 
{\it Proof.}
Let $A \subset X$ be sequentially closed and $x \in \rmclos A$.
Since $X$ is Fr\'echet-Urysohn, there is a sequence $\la x_j \ra_j$ in $A$ that converges to $x$.
Since $A$ is sequentially closed, $x \in A$.
By the arbitrariness of $x$, we infer that $A$ is closed.\cqfd.
\par 
The converse is false in general. 
In fact, most of the localized topologies introduced in this paper are sequential \ref{CWC.6}\refmod{\ref{CWC.6.c1}} but not Fr\'echet-Urysohn \ref{CCC.6.R}(Z). 
\begin{enumerate}
\item[(B)] {\it If $X[\calT]$ is Fr\'echet-Urysohn and $Y \subset X$ then $Y[\calT \hel Y]$ is Fr\'echet-Urysohn.}
\end{enumerate}
\par 
{\it Proof.}
Let $A \subset Y$ and $x \in \rmclos_{\calT \shel Y}A$.
Note that, as a subset of $X$, $\rmclos_{\calT \shel Y} A \subset \rmclos_{\calT} A$.
As $X$ is Fr\'echet-Urysohn, there exists a sequence in $A$ that $\calT$-converges to $x$.
Clearly, this sequence also $\calT \hel Y$-converges to $x$.\cqfd
\par 
In order to state our next observation we need the following definition.
With a subset $A \subset X$ of a Hausdorff topological space $X$ we associate its {\em sequential closure} defined as
\begin{equation*}
\rmclos_\rmseq A = X \cap \left\{ x : \text{ there is a sequence $\la x_j \ra_j$ in $A$ that converges to $x$}\right\}.
\end{equation*}
The phrasing, though traditional, is somewhat confusing, since the sequential closure of a set does not need to be sequentially closed.
Indeed, in a general space $X$ one needs to repeat transfinitely many times the sequential closure operation in order to obtain a sequentially closed set which is then, if $X$ is sequential, the (topological) closure of $A$.
In a Fr\'echet-Urysohn space the situation is much simpler as we now show.
\begin{enumerate}
\item[(C)] {\it The following are equivalent.
\begin{enumerate}
\item[(a)] $X$ is Fr\'echet-Urysohn.
\item[(b)] $\rmclos_\rmseq A = \rmclos A$ for all $A \subset X$.
\end{enumerate}
}
\end{enumerate}
\par 
{\it Proof.}
It is most obvious that $\rmclos_\rmseq A \subset \rmclos A$ in general.
The reverse inclusion is the definition of Fr\'echet-Urysohn.\cqfd
\par
Our last observation in this number is critical in \ref{PR.6} below\footnote{The proof of $(a) \Rightarrow (b)$ we give here is from \cite[Appendix 1(1)]{YAMAMURO} who claims it is from \cite{AVE.SMO.68}. In order to avoid confusion it is useful to note that what we call Fr\'echet-Urysohn here is what \textsc{S. Yamamuro} calls sequential.}.
It is to the extent that, in a Fr\'echet-Urysohn space, one is allowed to make ``diagonal arguments''.
\begin{enumerate}
\item[(D)] {\it Let $X$ be a sequential topological vector space. The following are equivalent.
\begin{enumerate}
\item[(a)] $X$ is Fr\'echet-Urysohn.
\item[(b)] If $\la x_j \ra_j$ is a sequence converging to $x$ and if, for each $j \in \N$, $\la x_{j,k} \ra_k$ is a sequence converging to $x_j$ then there are increasing functions $j : \N \to \N$ and $k : \N \to \N$ such that $\lim_n x_{j(n),k(n)} = x$.
\end{enumerate}
}
\end{enumerate}
\par 
{\it Proof that $(a) \Rightarrow (b)$.}
We first treat the case when $x_j = 0$ for all $j \in \N$.
Choose $a \neq 0$\footnote{If $X=\{0\}$ then (b) clearly holds.}.
For each $(j,k) \in \N \times \N$ define $y_{j,k} = x_{j,k} + \epsilon_{j,k} \left(\frac{a}{j}\right)$ where $\epsilon_{j,k} \in \{-1,+1\}$ is chosen in order that $y_{j,k} \neq 0$.
Thus, the set $A = X \cap \{ y_{j,j+l} : (j,l) \in \N \times \N \}$ does not contain 0.
Yet, its closure does.
Indeed, if $U$ is a neighbourhood of 0, we first choose\footnote{See Appendix \ref{app.LCTVS} for the vocabulary.} a balanced neighbourhood $V$ of 0 such that $V+V \subset U$, we then select an integer $j$ large enough for $\frac{a}{j} \in V$, and finally we select $l$ large enough for $x_{j,j+l} \in V$.
Thus, $y_{j,j+l} = x_{j,j+l} + \epsilon_{j,j+l} \left( \frac{a}{j} \right) \in V+V \subset U$.
\par 
Since we assume that $X$ is Fr\'echet-Urysohn, there exists a function $\N \to \N \times \N : n \mapsto (j(n),l(n))$ such that $\la y_{j(n),j(n)+l(n)} \ra_n$ is a null sequence.
Define $J = \N \cap \{ j(n) : n \in \N \}$ and $L = \N \cap \{ l(n) : n \in \N \}$.
Note that one or both of $J$ and $L$ is infinite, for otherwise $\la y_{j(n),j(n)+l(n)} \ra_n$ would admit a constant subsequence, according to the pigeonhole principle, in contradiction with $0 \not \in A$.
\par 
We next claim that $J$ is infinite.
Otherwise, as $L$ would be infinite (by the previous paragraph), referring to the pigeonhole principle again, the sequence $\la y_{j(n),j(n)+l(n)} \ra_n$ would admit a subsequence $\la y_{j_0,j_0+l'(n)} \ra_n$ with $l'(n) \uparrow \infty$ a $n \uparrow \infty$.
In that case, taking a further subsequence if necessary,
\begin{equation*}
0 = \lim_n  y_{j_0,j_0+l'(n)} = \lim_n \left( \epsilon_{j_0,l'(n)} \cdot \frac{a}{j_0} + x_{j_0,j_0+l'(n)}\right) = \pm \frac{a}{j_0},
\end{equation*}
in contradiction with $a \neq 0$.
\par 
Accordingly, one can select a subsequence $\la y_{j'(n),j'(n)+l'(n)} \ra_n$ with $j'(n) \uparrow \infty$ as $n \uparrow \infty$.
Since $n \leq j'(n)$, one has $j'(n)+l'(n) \leq j'(j'(n)+l'(n)) < j'(j'(n)+l'(n)) + l'(j'(n)+l'(n))$ for all $n$.
Thus, one can inductively define subsequences $\la j''(n) \ra_n$ and $\la j''(n)+l''(n) \ra_n$ of, respectively, $\la j'(n) \ra_n$ and $\la j'(n)+l'(n) \ra_n$ such that both are increasing.
\par
The proof is now complete in the special case when $x_j = 0$ for all $j$.
In the general case, we define $y_{j,k} = x_{j,k} - x_j$ for all $(j,k) \in \N \times \N$ and we note that $\lim_k y_{j,k} = 0$ for all  $j \in \N$.
Having, thereby, reduced to the special case we infer the existence of increasing functions $j$ and $k$ such that $\lim_n y_{j(n),k(n)} = 0$.
Since $x_{j(n),k(n)} = y_{j(n),k(n)} + x_{j(n)}$, we conclude that 
\begin{equation*}
 \lim_n  x_{j(n),k(n)} = \lim_n \left(y_{j(n),k(n)} + x_{j(n)} \right) = \lim_n x_{j(n)} = x.
\end{equation*}\cqfd
\par 
{\it Proof that $(b) \Rightarrow (a)$.}
Let $A \subset X$.
We claim that $\rmclos_\rmseq A$ is sequentially closed.
In order to prove this, consider a sequence $\la x_j \ra_j$ in $\rmclos_\rmseq A$ that converges to $x \in X$.
By definition of $\rmclos_\rmseq A$, with each $k \in \N$ we can associate a sequence $\la x_{j,k} \ra_k$ in $A$ that converges to $x_j$.
By (b), we infer that $\la x_{j(n),k(n)} \ra_n$ converges to $x$, for some increasing $j,k : \N \to \N$.
Accordingly, $x \in \rmclos_\rmseq A$.
Since $X$ is sequential, this implies that $\rmclos_\rmseq A$ is closed.
Therefore, $\rmclos_\rmseq A = \rmclos A$.
Since $A$ is arbitrary, the conclusion follows from (C).\cqfd
\end{Empty}

\begin{Empty}
\label{PR.20}
We recall that a topological space $X[\calT]$ is {\bf metrizable} iff there exists a metric $d$ on $X$ such that for every $E \subset X$ the following holds: $E \in \calT$ iff for every $x \in E$ there exists $r > 0$ such that $B_d(x,r) \subset E$ where $B_d(x,r) = X \cap \{ \xi : d(\xi,x) \leq r \}$.
\begin{enumerate}
\item[(A)] {\it If $X[\calT]$ is metrizable and $Y \subset X$ then $Y[\calT \hel Y]$ is metrizable.}
\end{enumerate}
\par 
{\it Proof.}
Trivial.\cqfd
\par 
We recall that a topological space $X[\calT]$ is {\bf first countable} iff for every $x \in X$ there exists a countable base of neighbourhoods of $x$, \ie a (finite or) countable set $\calV$ of neighbourhoods of $x$ such that each neighbourhood of $x$ contains a member of $\calV$.
\begin{enumerate}
\item[(B)] {\it If $X[\calT]$ is first countable and $Y \subset X$ then $Y[\calT \hel Y]$ is first countable.}
\end{enumerate}
\par 
{\it Proof.}
Trivial.\cqfd
\begin{enumerate}
\item[(C)] {\it If $X$ is metrizable then it is first countable.}
\end{enumerate}
\par 
{\it Proof.}
$\calV = \{ B_d(x,q) : q \text{ is positive and rational}\}$ is a base of neighbourhoods of $x$.\cqfd
\begin{enumerate}
\item[(D)] {\it If $X$ is first countable then it is Fr\'echet-Urysohn.}
\end{enumerate}
\par 
{\it Proof.}
Let $A \subset X$, $x \in \rmclos A$, and $\calV$ be a countable base of neighbourhoods of $x$.
Let $\la V_j \ra_{j \in \N}$ be a numbering of $\calV$.
For each $k \in \N$ there exists $x_k \in A \cap \left( \cap_{j=1}^k V_j\right)$.
One easily checks that $\la x_k \ra_k$ converges to $x$.\cqfd
\par 
\begin{enumerate}
\item[(E)] {\it For Hausdorff topological spaces the following hold.
\begin{equation*}
\text{metrizable } \Rightarrow \text{ first countable } \Rightarrow \text{ Fr\'echet-Urysohn } \Rightarrow \text{ sequential} 
\end{equation*}
}
\end{enumerate}
\par 
Each of these implications has been proved above.
\refmod{None of the reverse implications is true} in general.
With regard to the last one, it is worth noting that a Hausdorff topological space $X$ is Fr\'echet-Urysohn iff it is hereditarily sequential (\ie each of its subspaces is sequential).
Thus, each of the first three properties above is hereditary (\ie inherited by all subspaces) but the last one is not (though open and closed subspaces inherit sequentiality \ref{PR.2}(C)).
\end{Empty}

\section{Locally convex topological vector spaces}
\label{app.LCTVS}

Here, we gather all the vocabulary and facts regarding locally convex topological vector spaces used in the paper.

\begin{Empty}[Absolutely convex sets]
\label{LCTVS.1}
All vector spaces considered in this paper are real. 
In a vector space $X$ a set $A \subset X$ is:
\begin{itemize}
\item {\em balanced} if $t \cdot x \in A$ whenever $-1 \leq t \leq 1$ and $x \in A$;
\item {\em convex} if $x+t\cdot (y-x) \in A$ whenever $0 \leq t \leq 1$ and $x,y \in A$;
\item {\em absolutely convex} if it is both balanced and convex.
\end{itemize}
Obviously, $X$ itself satisfies the three properties.
Furthermore, each $A \subset X$ is contained in a smallest convex (resp. absolutely convex) set denoted $\rmconv(A)$ (resp. $\rmabsconv(A)$) and called the {\em convex hull} (resp. {\em absolutely convex hull}) of $A$.
\end{Empty}

\begin{Empty}[Topological vector spaces]
\label{LCTVS.2}
If $X$ is a vector space and $\calT$ is a topology on $X$ we say that $\calT$ is a {\em vector topology} or equivalently that $X[\calT]$ is a {\em topological vector space} whenever the following conditions are satisfied:
\begin{itemize}
\item All singletons in $X$ are $\calT$-closed;
\item $X \times X \to X : (x,y) \mapsto x+y$ is $(\calT \times \calT, \calT)$ continuous;
\item $\R \times X \to X : (t,x) \mapsto t\cdot x $ is $(\calT_\R \times \calT , \calT)$ continuous.
\end{itemize}
In case a topology $\calT$ on $X$ satisfies the last two conditions above, it is Hausdorff iff it satisfies the first condition, see \cite[Theorem 1.12]{RUDIN}. 
When the vector topology $\calT$ that equips $X$ is clear from context we will sometimes write $X$ instead of $X[\calT]$.
\par 
A {\em local base} in a topological vector space is a set $\calB$ of neighbourhoods of 0 such that every neighbourhood of 0 contains a member of $\calB$. 
It follows that the open sets are unions of translates of members of $\calB$.
\par 
A topological vector space is termed {\em locally convex} if it admits a local base whose members are convex sets. In that case a local base can be chosen to consist of absolutely convex closed sets, see \cite[Theorem 1.13 and Corollary]{RUDIN}. 
\end{Empty}

\begin{Empty}[Absorbing sets and Minkowski functionals]
\label{LCTVS.3}
\label{SR.2}
Let $X$ be a vector space and $A \subset X$.
We say that $A$:
\begin{itemize}
\item {\em absorbs a set} $B \subset X$ if there exists $t > 0$ such that $B \subset t \cdot A$;
\item is {\em absorbing} if it absorbs every singleton.
\end{itemize}
Each absorbing set contains 0.
With an absorbing set $A \subset X$ we associate its {\em Minkowski functional} $q_A$ defined by
\begin{equation*}
q_A(x) = \inf \left\{ t > 0 : t^{-1} \cdot x \in A \right\} < \infty \,.
\end{equation*}
One checks \refmod{that if} $A$ is absolutely convex then $q_A$ is a seminorm on $X$ and $\{ q_A < 1 \} \subset A \subset \{ q_A \leq 1\}$. 
Furthermore, if a vector topology $\calT$ is given on $X$ and if $A$ is balanced and $\calT$-closed then $\{ q_A \leq 1 \} = A$.
\end{Empty}

\begin{Empty}[Locally convex topological vector spaces]
\label{LCTVS.4}
We say that a family $\la q_i \ra_{i \in I}$ of seminorms on a real vector space $X$ is {\em separating} \refmod{if and only if} for every $x \in X \setminus \{0\}$ there exists $i \in I$ such that $q_i(x) > 0$. 
In that case the family $\la q_i \ra_{i \in I}$ defines a locally convex vector topology $\calT$ on $X$ with a local base consisting of the sets 
\begin{equation}
\label{eq.LCTVS.1}
\bigcap_{n=1}^N X \cap \{ x : q_{i_n}(x) \leq \veps_n \}
\end{equation}
corresponding to all choices of $N \in \N \setminus \{0\}$, $i_1,\ldots,i_N \in I$, and $\veps_1 > 0, \ldots, \veps_N > 0$. 
We say that $\la q_i \ra_{i \in I}$ {\em generates} $\calT$. 
Furthermore, we say that $\la q_i \ra_{i \in I}$ is {\em filtering} whenever for each $i_1,\ldots,i_N \in I$ there exists $i \in I$ such that $\max_{n=1,\ldots,N} q_{i_n} \leq q_i$.
\par 
Reciprocally, if $\calT$ is a locally convex vector topology on $X$ with a local base $\calB$ consisting of absolutely convex $\calT$-closed sets then the family of seminorms $\la q_V \ra_{V \in \calB}$ is separating, filtering, and generates $\calT$. 
\end{Empty}

\begin{Empty}[Fr\'echet spaces]
\label{LCTVS.11}
If the topology $\calT$ of a locally convex topological vector space $X$ is generated by a countable family of seminorms $\la q_n \ra_{n \in \N}$ then $\calT$ is metrizable by a metric $d$ defined by the formula $d(x_1,x_2) = \sum_{n \in \N} \frac{q_n(x_1-x_2)}{2^n\cdot (1+q_n(x_1-x_2))}$, for $x_1,x_2 \in X$.
In this case, the topology $\calT$ admits a countable local base.
Notice that the metric $d$ defined above is {\em translation invariant}, \ie $d(x_1+h,x_2+h) = d(x_1,x_2)$ for all $x_1,x_2,h \in X$.
\par 
Reciprocally, if $X[\calT]$ is a locally convex topological vector space whose topology $\calT$ admits a countable local base then there exists on $X$ a metric $d$ which is compatible with $\calT$, translation invariant, and satisfies the property that all open balls are convex, see \cite[1.24]{RUDIN}.
\par 
If $X[\calT]$ is a locally convex topological vector space whose topology is metrized by a translation invariant metric $d$ then one can define the notion of a Cauchy sequence with respect to the topology $\calT$ and also with respect to the metric $d$.
These notions coincide, see \cite[1.25]{RUDIN}.
We say that $X[\calT]$ is a {\em Fr\'echet space} if it is a metrizable locally convex topological vector space and every Cauchy sequence in $X$ is convergent.
\end{Empty}

\begin{Empty}[Banach spaces]
\label{LCTVS.12}
If $X$ is a locally convex topological vector space whose topology is generated by one single seminorm (which is then a norm) we say that $X$ is {\em normable} or a {\em normed space} when a specific norm has been singled out.
If $X$ is both normed and complete with respect to the corresponding metric then we say that $X$ is a {\em Banach space} (thus, Banach spaces are Fr\'echet spaces).
\end{Empty}

\begin{Empty}[Bounded sets]
\label{LCTVS.5}
\label{PR.1}
In a topological vector space $X$ each neighbourhood of 0 is absorbing, see \cite[Theorem 1.15(a)]{RUDIN}. 
We say that a set $B \subset X$ is {\em bounded}\footnote{In case $X$ is metrizable this is {\it not} equivalent to being bounded with respect to a metric defining the topology of $X$. Nonetheless, if $X$ is a normed space then this notion of boundedness coincides with the metric notion relative to the ambient norm.} if it is absorbed by each neighbourhood of 0.
Thus, singletons are bounded.
More generally, one checks that each compact subset of $X$ is bounded.
Obviously, subsets of bounded sets are bounded as well.
In particular, if $\la x_j \ra_{j \in \N}$ is a convergent sequence then the set $\{ x_j : j \in \N\}$ is bounded. 
\begin{enumerate}
\item[(A)] {\it Let $X$ be a topological vector space and $B \subset X$. The following are equivalent.
\begin{enumerate}
\item[(a)] $B$ is bounded.
\item[(b)] For every open set $U$ containing 0 there exists $s > 0$ such that for every $t \geq s$ one has $B \subset t \cdot U$.
\item[(c)] For every open set $U$ containing 0 there exists $t \geq 1$ such that $B \subset t \cdot U$.
\end{enumerate}
}
\end{enumerate}
\par 
{\it Proof.}
This easily follows from the fact that each neighbourhood of 0 contains a balanced neighbourhood of 0, see \eg \cite[1.14(a)]{RUDIN}.\cqfd
\par 
A {\em null sequence} in a topological vector space is one that converges to 0.
The following shows that the knowledge, which among all sequences are null sequences is enough to recover the class of bounded sets.
\begin{enumerate}
\item[(B)] {\it Let $X$ be a topological vector space and $B \subset X$. The following are equivalent.
\begin{enumerate}
\item[(a)] $B$ is bounded.
\item[(b)] If $\la x_j \ra_j$ is a sequence in $B$ and $\la t_j \ra_j$ is a null sequence in $\R$ then $\la t_j \cdot x_j \ra_j$ is a null sequence in $X$.
\end{enumerate}
}
\end{enumerate}
\par 
{\it Proof.}
See for instance \cite[1.30]{RUDIN}.\cqfd
\begin{enumerate}
\item[(C)] {\it Let $X$ be a locally convex topological vector space and $B \subset X$ a bounded set.
\begin{enumerate}
\item[(a)] $\rmclos(B)$ is bounded.
\item[(b)] $\rmabsconv(B)$ is bounded.
\end{enumerate}
}
\end{enumerate}
\par 
{\it Proof.}
Both statements follow from the fact that each neighbourhood of 0 contains an absolutely convex closed neighbourhood of 0.\cqfd
\begin{enumerate}
\item[(D)] {\it Let $X$ be a locally convex topological vector space whose topology is generated by a family $\la q_i \ra_{i \in I}$ of seminorms and let $B \subset X$. The following are equivalent.
\begin{enumerate}
\item[(a)] $B$ is bounded.
\item[(b)] $\sup_{x \in B} q_i(x) < \infty$ for every $i \in I$.
\end{enumerate}
}
\end{enumerate}
\par 
{\it Proof.}
This follows from the special form \eqref{eq.LCTVS.1} of a local base of $X$.\cqfd
\par 
A {\em fundamental system of bounded sets} in $X$ is a set $\calF$ consisting of bounded subsets of $X$ so that each bounded subset of $X$ is contained in some member of $\calF$.
\end{Empty}

\begin{Empty}[Continuous and bounded linear functions]
\label{LCTVS.6}
Here, we are given two topological vector spaces $X$ and $Y$ and a linear function $f : X \to Y$.
\begin{enumerate}
\item[(A)] {\it The following are equivalent.
\begin{enumerate}
\item[(a)] $f$ is continuous.
\item[(b)] $f$ is continuous at 0.
\item[(c)] $f$ is continuous at some point of $X$.
\end{enumerate}
}
\end{enumerate}
\par 
{\it Proof.}
This follows from the continuity of translations in the domain and codomain of $f$ and from its linearity.\cqfd
\begin{enumerate}
\item[(B)] {\it Assume $X$ is locally convex and its topology is generated by the family of seminorms $\la q_i \ra_{i \in I}$ and $Y$ is locally convex and its topology is generated by the family of seminorms $\la r_j \ra_{j \in J}$. The following are equivalent.
\begin{enumerate}
\item[(a)] $f$ is continuous.
\item[(b)] For every $j \in J$ there are $a > 0$, $N \in \N \setminus \{0\}$, and $i_1,\ldots,i_N \in I$ such that $r_j \circ f \leq a \cdot \max_{n=1,\ldots,N} q_{i_n}$. 
\end{enumerate}
}
\end{enumerate}
\par 
{\it Proof.}
This follows from (A) and the form of a special local base of $Y$ and $X$ in \eqref{eq.LCTVS.1}.\cqfd
\par 
If $\la q_i \ra_{i \in I}$ is filtering then (b) is equivalent to the same with $N=1$.
If $Y=\R$ then (b) is equivalent to: There are $a > 0$, $N \in \N\setminus \{0\}$, and $i_1,\ldots,i_N \in I$ such that $|f| \leq a \cdot \max_{n=1,\ldots,N} q_{i_n}$. 
\par 
We say that $f$ is {\em bounded} if $f(B)$ is bounded in $Y$ whenever $B \subset X$ is bounded.
Recall that we say $f$ is {\em sequentially continuous} if $\lim_j f(x_j) = f(x)$ for every convergent sequence $\la x_j \ra_j$ in $X$ whose limit is $x$.
If $f$ is continuous then it is sequentially continuous but the converse fails in general (see also \ref{PR.2}).
\begin{enumerate}
\item[(C)] {\it If $f$ is sequentially continuous then it is bounded.}
\end{enumerate}
\par 
{\it Proof.}
Let $B \subset X$ be bounded and $\la y_j \ra_j$ a sequence in $f(B)$.
Choose a sequence $\la x_j \ra_j$ in $B$ such that $f(x_j) = y_j$ for all $j \in \N$.
If $\la t_j \ra_j$ is a null sequence in $\R$ then $\la t_j \cdot x_j \ra_j$ is a null sequence in $X$, by \ref{LCTVS.5}(B).
Since $f$ is sequentially continuous and linear, we infer that $\la t_j  \cdot y_j \ra_j$ is a null sequence in $Y$.
Since $\la y_j \ra_j$ is arbitrary, the boundedness of $f(B)$ follows from \ref{LCTVS.5}(B).\cqfd
\end{Empty}

\begin{Empty}[First dual and its strong topology]
\label{LCTVS.7}
Let $X[\calT]$ be a topological vector spa\-ce.
We let $X[\calT]^*$ (or simply $X^*$ when no confusion can arise) be the set consisting of all $\calT$-continuous linear functions $X \to \R$ and we call it the {\em first dual} of $X$.
It is obviously a real vector space.
If $B \subset X$ is a bounded set we define a seminorm $p_B : X^* \to \R$ by means of the formula
\begin{equation*}
p_B(f) = \sup \{ |f(x)| : x \in B \}.
\end{equation*}
Notice that $p_B(f) < \infty$ for all $f \in X^*$, by \ref{LCTVS.6}(C).
Let $\frS_\rmbd$ be the set of all bounded subsets of $X$.
Then $\la p_B \ra_{B \in \frS_\rmbd}$ is a family of seminorms on $X^*$ and is separating (since singletons are bounded).
The locally convex topology on $X^*$ generated by this family is called the {\em strong topology} of $X^*$ and denoted $\beta(X^*,X)$ (or $T_{\frS_\rmbd}$ in Appendix \ref{app.DUAL}).
\par 
If $X$ is a Banach space then $\beta(X^*,X)$ corresponds to the usual normed topology in $X^*$.
In case $X$ is a Fr\'echet space then $X^*[\beta(X^*,X)]$ does not need to be a Fr\'echet space itself, since $X$ may not possess a {\it countable} fundamental system of bounded sets.
\end{Empty}

\begin{Empty}[Hahn-Banach]
\label{LCTVS.8}
In this paper we refer to the following four versions of the Hahn-Banach theorem.
Here, $X$ is a locally convex topological vector space.
\begin{enumerate}
\item[(A)] {\it Let $x_0 \in X$. If $x_0 \neq 0$ then there exists $x^* \in X^*$ such that $x^*(x_0) \neq 0$.}
\end{enumerate}
\par 
{\it Proof.}
Apply \eg \cite[3.4(b)]{RUDIN} with $A=\{0\}$ and $B=\{x_0\}$.\cqfd
\begin{enumerate}
\item[(B)] {\it If $A \subset X$ is absolutely convex and closed and $x_0 \not \in A$ then there exists $x^* \in X^*$ such that
\begin{enumerate}
\item[(a)] $|x^*(x)| \leq 1$ for all $x \in A$;
\item[(b)] $x^*(x_0) > 1$.
\end{enumerate}
}
\end{enumerate}
\par 
{\it Proof.}
See \eg \cite[3.7]{RUDIN}.\cqfd
\begin{enumerate}
\item[(C)] {\it If $C \subset X$ is convex and closed and $x_0 \not \in C$ then there exists an open convex set $U$ such that $C \subset U$ and $x_0 \not \in U$.}
\end{enumerate}
\par 
{\it Proof.}
Apply \cite[3.4(b)]{RUDIN} with $A=\{x_0\}$ and $B=C$ and let $U = X \cap \{ x : \gamma_2 < \Lambda(x) \}$ where $\gamma_2$ and $\Lambda$ are as in the reference.\cqfd
\begin{enumerate}
\item[(D)] {\it If $X[\calT]$ is normed by $\|\cdot\|$ and $x_0 \in X$ then there exists $x^* \in X[\calT]^*$ such that $x^*(x_0)=\|x_0\|$ and $|x^*(x)| \leq \|x\|$ for all $x \in X$.}
\end{enumerate}
\par 
{\it Proof.}
See \eg \cite[Corollary of 3.3]{RUDIN}.\cqfd
\end{Empty}

\begin{Empty}[Second dual and the evaluation map]
\label{LCTVS.9}
Let $X[\calT]$ be a locally convex topological vector space.
We consider its first dual $X^*[\beta(X^*,X)]$ equipped with the strong topology, recall \ref{LCTVS.7}.
The first dual of the latter is denoted $X^{**}$ (thus, $X^{**}= X^*[\beta(X^*,X)]^*$) and called the {\em second dual} of $X$.
Accordingly, $X^{**}$ is a real vector space and we do not equip it (at this time) with a topology.
\par 
Given $x \in X$ we define $\rmev(x) : X^* \to \R$ by the formula $\rmev(x)(x^*) = x^*(x)$.
This is readily a linear map and $|\rmev(x)| \leq p_{\{x\}}$, whence, $\rmev(x)$ is $\beta(X^*,X)$-continuous.
In other words, $\rmev(x) \in X^{**}$.
This defines a function $\rmev : X \to X^{**} : x \mapsto \rmev(x)$ which we call the {\em evaluation map}.
It is also plain that $\rmev$ is linear.
It follows from the Hahn-Banach theorem \ref{LCTVS.8}(A) that $\rmev$ is injective.
We say that $X$ is {\em semireflexive} if and only if $\rmev$ is surjective.
In \ref{DUAL.6}(H) we recall a characterization of semireflexive spaces, itself a consequence of the Mackey-Arens theorem \ref{DUAL.6}(E).
This is instrumental in the proof of our main existence theorem \ref{AET.FR} via \ref{SD.5}.
\par 
We end this number with a trivial observation.
Recall the definition of weak* topologies, \eg \ref{DUAL.3}.
\begin{enumerate}
\item[(A)] {\it The evaluation map is $(\calT,\sigma(X^{**},X^*))$-continuous.}
\end{enumerate}
\par 
{\it Proof.}
Write $\bev$ for the obvious evaluation map $X^* \to (X^{**})^\bigstar$ (recall \ref{DUAL.1} for the notation) and note that $\sigma(X^{**},X^*)$ is generated by $\la |\bev(x^*)| \ra_{x^* \in X^*}$.
Let $\la q_i \ra_{i \in I}$ be a family of seminorms that generates $\calT$. 
We ought to observe (recall \ref{LCTVS.4}) that for each $x^* \in X^*$ there are $a > 0$ and $i_1,\ldots,i_N \in I$ such that $|\bev(x^*)| \circ ev \leq a \cdot \max_{n=1,\ldots,N} q_{i_n}$, \ie $| \la x,x^* \ra | \leq  a \cdot \max_{n=1,\ldots,N} q_{i_n}(x)$ for all $x \in X$.
This is readily a formulation of the $\calT$-continuity of $x^*$.\cqfd
\end{Empty}

\begin{Empty}[Closed absolutely convex hull]
\label{LCTVS.10}
We recall the definitions of absolutely convex set and absolutely convex hull from \ref{LCTVS.1}.
These make sense in a real vector space.
We now consider a locally convex topological vector space $X[\calT]$ and we notice that any subset $A \subset X$ is contained in a smallest closed absolutely convex set called its {\em closed absolutely convex hull} and denoted $\overline{\rmabsconv}(A)$ or $\overline{\rmabsconv}^{\calT}(A)$ when there is a need to emphasize the topology on $X$ with respect to which the closure is being considered.
\begin{enumerate}
\item[(A)] {\it Assume that $A_1,\ldots,A_N \subset X$ are absolutely convex and compact. Then $\overline{\rmabsconv}(A_1 \cup \cdots \cup A_N)$ is compact as well.}
\end{enumerate}
\par 
{\it Proof.}
\textbf{(i)}
Define $S = [0,1]^N \cap \left\{ (t_1,\ldots,t_N) : \sum_{n=1}^N t_n = 1 \right\}$ and 
\begin{equation*}
C = X \cap \left\{ \sum_{n=1}^N t_n \cdot x_n : (t_1,\ldots,t_N) \in S \text{ and } (x_1,\ldots,x_N) \in A_1 \times \cdots \times A_N \right\}.
\end{equation*}
Abbreviating $A = A_1 \cup \cdots \cup A_N$, it is obvious that $A \subset C \subset \rmabsconv(A)$.
\par 
\textbf{(ii)}
Letting $x = \sum_n s_n \cdot x_n$ and $y = \sum_n t_n \cdot y_n$ be members of $C$ and $0 \leq \lambda \leq 1$, we notice that $\lambda \cdot x + (1-\lambda) \cdot y = \sum_n [\lambda \cdot  s_n \cdot  x_n + (1-\lambda)\cdot t_n \cdot y_n] = \sum_n [\lambda \cdot s_n + (1-\lambda)\cdot t_n] \cdot z_n \in C$ where $z_n \in A_n$ are chosen as follows.
Abbreviating $\alpha_n = \lambda \cdot s_n$ and $\beta_n = (1-\lambda) \cdot t_n$, if $\alpha_n=0=\beta_n$ then we choose $z_n = x_n$, otherwise we choose $z_n = (\alpha_n \cdot x_n + \beta_n \cdot y_n)/(\alpha_n + \beta_n)$.
It follows at once that $C$ is convex.
As $C$ is readily balanced, we infer that $C = \rmabsconv(A)$, in view of \textbf{(i)}.
\par 
\textbf{(iii)}
It follows from \textbf{(ii)} that $C \subset \overline{\rmabsconv}(A)$.
From the definition of $C$ in \textbf{(i)} we observe that $C$ is the image of the compact set $S \times A$ under the continuous map
\begin{equation*}
S \times A \to X : (t_1,\ldots,t_N,x_1,\ldots,x_N) \mapsto \sum_{n=1}^N t_n \cdot x_n.
\end{equation*}
Therefore, $C$ is compact.
Being, in particular, closed and absolutely convex, we infer that $\overline{\rmabsconv}(A) \subset C$.
Thus, $\overline{\rmabsconv}(A) = C$ is compact.\cqfd
\begin{enumerate}
\item[(B)] {\it If $A$ is absolutely convex then so is $\rmclos A$.}
\end{enumerate}
\par 
{\it Proof.}
See \eg \cite[1.13(d) and (e)]{RUDIN}.\cqfd
\begin{enumerate}
\item[(C)] {\it If $A$ is bounded then so is $\overline{\rmabsconv}(A)$.}
\end{enumerate}
\par 
{\it Proof.}
Notice that $\rmabsconv(A)$ is bounded, by \ref{LCTVS.5}(C)(b), thus, $\rmclos(\rmabsconv(A))$ is bounded as well, by \ref{LCTVS.5}(C)(a).
As $\rmclos(\rmabsconv(A))$ is absolutely convex, by (B) above, we infer that $\overline{\rmabsconv}(A) \subset \rmclos(\rmabsconv(A))$, by definition of the former, and the conclusion follows.\cqfd
\end{Empty}

\section{Dual systems and the Mackey-Arens theorem}
\label{app.DUAL}

\begin{Empty}[Algebraic vs. topological dual]
\label{DUAL.1}
If $X$ is a real vector space then its algebraic dual (consisting of all linear functions $X \to \R$) is denoted $X^\bigstar$. 
We recall that if $X[\calT]$ is a locally convex topological vector space then its topological dual is denoted $X[\calT]^*$, \ie
\begin{equation*}
X[\calT]^* = X^\bigstar \cap \{ f : f \text{ is $\calT$-continuous} \}
\footnote{There does not seem to be a universal agreement with notation in this respect. For instance, \cite{EDWARDS} uses $X^*$ for the algebraic dual and $X'$ for the topological dual whereas \cite{RUDIN} uses $X^*$ for the topological dual.}
.
\end{equation*}
\end{Empty}

\begin{Empty}[Dual systems]
\label{DUAL.2}
A {\em dual system} is a triple $(X,Y,b)$ such that $X$ and $Y$ are real vector spaces and $b : X \times Y \to \R$ is a function satisfying the following conditions:
\begin{itemize}
\item $b$ is bilinear;
\item $b$ is separating in the $1^{st}$ variable, \ie $(\forall x \in X \setminus \{0\})(\exists y \in Y) : b(x,y) \neq 0$;
\item $b$ is separating in the $2^{nd}$ variable, \ie $(\forall y \in Y \setminus \{0\})(\exists x \in X) : b(x,y) \neq 0$.
\end{itemize}
If $(X,Y,b)$ is a dual system then so is $(Y,X,\tilde{b})$ where $\tilde{b}(y,x)=b(x,y)$.
We will call $(Y,X,\tilde{b})$ the {\em symmetric} of $(X,Y,b)$.
We now give the two main examples of dual systems.
\begin{enumerate}
\item[(1)] If $X$ is a real vector space and $X^\bigstar$ its algebraic dual then we define a bilinear map
\begin{equation*}
\la \cdot , \cdot \ra : X \times X^\bigstar \to \R : (x,x^\bigstar) \mapsto \la x , x^\bigstar \ra
\end{equation*}
by the formula $\la x , x^\bigstar \ra = x^\bigstar(x)$.
It is trivially separating in the second variable.
It is separating in the first variable thanks to the existence of linear extensions of partially defined linear functions.
Thus, $(X,X^\bigstar,\la \cdot , \cdot \ra)$ is a dual system.
\item[(2)] If $X[\calT]$ is a locally convex topological vector space and $X^*=X[\calT]^*$ is its topological dual then $X^* \subset X^\bigstar$ and the restriction of $\la \cdot , \cdot \ra$ to $X \times X^*$ (still denoted $\la \cdot , \cdot \ra$) readily defines a bilinear map which is separating in the second variable.
It is separating in the first variable as well, according to the Hahn-Banach theorem \ref{LCTVS.8}(A).
Thus, $(X,X^*,\la \cdot , \cdot \ra)$ is a dual system.
\end{enumerate}
\par 
More generally if $Y \subset X^\bigstar$ is a vector subspace such that the restriction of $\la \cdot , \cdot \ra$ to $X \times Y$ is separating in the first variable then $(X,Y,\la \cdot , \cdot \ra)$ is a dual system.
We now explain why this is the generic situation.
\par 
Given a dual system $(X,Y,b)$ and $y \in Y$ we define $\iota_2(y) : X \to \R$ by the formula $\iota_2(y)(x) = b(x,y)$.
Clearly, $\iota_2(y)$ is linear and $\iota_2 : Y \to X^\bigstar$ is linear as well.
Furthermore, $\iota_2$ is injective, since $b$ is separating in the second variable.
Thus, $Y$ is linearly isomorphic to the linear subspace $\iota_2(Y)$ of $X^\bigstar$.
\par 
Similarly, one defines a linear map $\iota_1 : X \to Y^\bigstar : x \mapsto (y \mapsto b(x,y))$ which is injective, since $b$ is separating in the first variable, and $X$ is linearly isomorphic to the linear subspace $\iota_1(X)$ of $Y^\bigstar$.
\end{Empty}

\begin{Empty}[Weak topologies]
\label{DUAL.3}
Let $(X,Y,b)$ be a dual system.
For each $y \in Y$, $|\iota_2(y)|$ is a seminorm on $X$.
Furthermore, the family of seminorms $\la |\iota_2(y)| \ra_{y \in Y}$ is separating, since $b$ is separating in the first variable.
Therefore, it defines a locally convex vector topology on $X$ which we denote by $\sigma(X,Y)$.
Similarly, $\sigma(Y,X)$ is the locally convex vector topology on $Y$ associated with the family of seminorms $\la |\iota_1(x)| \ra_{x \in X}$.
\par 
For instance, if $X[\calT]$ is a locally convex topological vector space and $(X,X^*,\la \cdot,\cdot \ra)$ its associated dual system as in \ref{DUAL.2}(2) then $\sigma(X,X^*)$ is the usual {\em weak topology} of $X$ and $\sigma(X^*,X)$ is the usual {\em weak* topology} of $X^*$.
\begin{enumerate}
\item[(A)] {\it If $(X,Y,b)$ is a dual system and $\calT$ is a vector topology on $X$ such that $\iota_2(y)$ is $\calT$-continuous for all $y \in Y$ then $\sigma(X,Y) \subset \calT$.}
\end{enumerate}
\par 
{\it Proof.}
Easy.\cqfd
\begin{enumerate}
\item[(B)] {\it Let $(X,Y,b)$ be a dual system. Then $X[\sigma(X,Y)]^* = \iota_2(Y)$.}
\end{enumerate}
\par 
{\it Proof.}
{\bf (i)}
If $E$, $F$, and $G$ are real vector spaces and if $f : E \to F$ and $g : E \to G$ are linear maps such that $\ker g \subset \ker f$ then there exists a linear map $h : G \to F$ such that $f = h \circ g$.
It suffices to notice that $h$ is unambiguously defined on $g(E)$ by means of $h(g(x))=f(x)$ and to extend it to the entire $G$.
\par 
{\bf (ii)}
If $f \in X^\bigstar$ and $y_1,\ldots,y_N \in Y$ are such that $\cap_{n=1}^N \ker \iota_2(y_n) \subset \ker f$ then $f$ is a linear combination of $\iota_2(y_1),\ldots,\iota_2(y_N)$.
It suffices to apply {\bf (i)} with $E=X$, $F=\R$, $G=\R^N$, and $g(x)=(\iota_2(y_1)(x),\ldots,\iota_2(y_N)(x))$.
\par 
{\bf (iii)}
Let $f \in X^\bigstar$ be $\sigma(X,Y)$-continuous.
There are $y_1,\ldots,y_N \in Y$ and positive reals $\veps_1,\ldots,\veps_N$ such that
\begin{equation*}
\bigcap_{n=1}^N X \cap \{ x : |\iota_2(y_n)(x)| \leq \veps_n \} \subset X \cap \{ x : |f(x)| < 1 \}.
\end{equation*}
In particular, $\cap_{n=1}^N \ker \iota_2(y_n) \subset \ker f$.
Therefore, $f \in \iota_2(Y)$, by {\bf (ii)}.
\par 
{\bf (iv)}
If $y \in Y$ then $\iota_2(y)$ is readily $\sigma(X,Y)$-continuous.\cqfd
\end{Empty}

\begin{Empty}[Polarity]
\label{DUAL.4}
Let $(X,Y,b)$ be a dual system, $A \subset X$, and $B \subset Y$.
We define the {\em polar} of $A$ by means of the formula
\begin{equation*}
A^\circ = Y \cap \{ y : |b(x,y)| \leq 1 \text{ for all }x\in A\}
\end{equation*}
and the {\em polar} of $B$ by
\begin{equation*}
\leftindex^\circ B = X \cap \{ x : | b(x,y)| \leq 1 \text{ for all }y \in B\}.
\end{equation*}
\begin{enumerate}
\item[(A)] {\it $A^\circ$ is absolutely convex and $\sigma(Y,X)$-closed. $^\circ B$ is absolutely convex and $\sigma(X,Y)$-closed.}
\end{enumerate}
\par 
{\it Proof.}
It suffices to observe that $A^\circ$ is the intersection of the absolutely convex and $\sigma(Y,X)$-closed sets $Y \cap \{ y : |\iota_1(x)(y)| \leq 1 \}$ corresponding to all $x \in A$.
A similar argument applies to $^\circ B$.\cqfd
\begin{enumerate}
\item[(B)] {\it The following hold.
\item[(a)] If $A_1 \subset A_2 \subset X$ then $A_1^\circ \supset A_2^\circ$ and if $B_1 \subset B_2 \subset Y$ then $\leftindex^\circ B_1 \supset \leftindex^\circ B_2$.
\item[(b)] If $t \in \R \setminus \{0\}$, $A \subset X$, and $B \subset Y$ then $t \cdot A^\circ = (t^{-1} \cdot A)^\circ$ and $t \cdot  \leftindex^\circ B = \leftindex^\circ(t^{-1} \cdot B)$.
\item[(c)] If $\la A_i \ra_{i \in I}$ is a family in $X$ and $\la B_i \ra_{i \in I}$ a family in $Y$ then $\left( \cup_{i \in I} A_i \right)^\circ = \cap_{i \in I} A_i^\circ$ and $^\circ\left( \cup_{i \in I} B_i \right) = \cap_{i \in I} \leftindex^\circ B_i$.
}
\end{enumerate}
\par 
{\it Proof.}
All trivial.\cqfd
\begin{enumerate}
\item[(C)] {\bf (Bipolar theorem)} {\it If $A \subset X$ then
\begin{equation*}
\leftindex^\circ (A^\circ) = \overline{\rmabsconv}^{\sigma(X,Y)}(A)
\end{equation*}
and if $B \subset Y$ then
\begin{equation*}
 (\leftindex^\circ B)^\circ = \overline{\rmabsconv}^{\sigma(Y,X)}(B).
\end{equation*}
}
\end{enumerate}
\par 
{\it Proof.}
We prove only the first part, since the proof of the second part is similar.
One easily checks that $A \subset \leftindex^\circ (A^\circ)$.
Therefore, it follows from (A) applied to $B = A^\circ$ and from the definition of closed absolutely convex hull that $\overline{\rmabsconv}^{\sigma(X,Y)}(A) \subset \leftindex^\circ (A^\circ)$.
Assume if possible that the exists $x_0 \in \leftindex^\circ (A^\circ) \setminus \overline{\rmabsconv}^{\sigma(X,Y)}(A)$.
According to the Hahn-Banach theorem \ref{LCTVS.8}(B), there exists $f \in X[\sigma(X,Y)]^*$ such that $|f(x)| \leq 1$ for all $x \in \overline{\rmabsconv}^{\sigma(X,Y)}(A)$ and $f(x_0) > 1$.
By \ref{DUAL.3}(B), $f = \iota_2(y)$ for some $y \in Y$.
Thus, $|b(x,y)| = |\iota_2(y)(x)| = |f(x)| \leq 1$ for all $x \in A$, \ie $y \in A^\circ$.
As $x_0 \in \leftindex^\circ(A^\circ)$, we have $1 \geq |b(x_0,y)| = |\iota_2(y)(x_0)| = |f(x_0)|$, a contradiction.\cqfd
\begin{enumerate}
\item[(D)] {\bf (Banach-Alaoglu theorem)} {\it Let $X$ be a locally convex topological vector space and let $V$ be a neighbourhood of the origin in $X$. Then $V^\circ$ is $\sigma(X^*,X)$-compact (here, polarity is understood in the dual system described in \ref{DUAL.2}(2)).}
\end{enumerate}
\par 
{\it Proof.}
See for instance \cite[3.15]{RUDIN}.\cqfd
\par 
It will happen that we will consider in the same calculation polarities with respect to two different dual systems.
Specifically, assume we are considering a dual system $(X,Y,b)$ and we are using the generic symbol $\circ$ for the corresponding polarity.
We may simultaneously need to consider the dual system $(X,X^\bigstar,\la \cdot,\cdot \ra)$ (associated with $X$ in \ref{DUAL.2}(1)) and we will then use the symbol $\bullet$ for the polarity corresponding to this second dual system.
Recall that in this situation one identifies $Y$ with the vector subspace $\iota_2(Y)$ of $X^\bigstar$. 
\begin{enumerate}
\item[(E)] {\it Let $X[\calT]$ be a locally convex topological vector space and $\calB$ be a local base for $\calT$.
Then, in the dual system $(X,X^\bigstar,\la \cdot,\cdot \ra)$ one has
\begin{equation*}
X[\calT]^* = \bigcup \{ V^\bullet : V \in \calB \}.
\end{equation*}
}
\end{enumerate}
\par 
{\it Proof.}
If $V \in \calB$ and $f \in V^\bullet$ then $V \subset X \cap \{ x : |f(x)| \leq 1 \}$. 
This shows that $f$ is $\calT$-continuous, \ie $f \in X[\calT]^*$.
Reciprocally, if $f \in X^\bigstar$ is $\calT$-continuous then there exists $V \in \calB$ such that $V  \subset X \cap \{ x : |f(x)| \leq 1 \}$, \ie $f \in V^\bullet$.\cqfd
\begin{enumerate}
\item[(F)] {\it Let $(X,Y,b)$ be a dual system and let $B \subset Y$ be absolutely convex and $\sigma(Y,X)$-compact. Then
\begin{equation*}
(\leftindex^\circ B)^\bullet = \iota_2(B).
\end{equation*}
}
\end{enumerate}
\par 
{\it Proof.}
{\bf (i)} We first claim that $\leftindex^\circ B = \leftindex^\bullet(\iota_2(B))$. 
Indeed, for all $x \in X$ we have
\begin{equation*}
\begin{split}
x \in \leftindex^\circ B & \iff (\forall y \in B) : |b(x,y)| \leq 1 \\
& \iff (\forall y \in B) : |\la x , \iota_2(y) \ra | \leq 1 \\
& \iff (\forall \zeta \in \iota_2(B)) : | \la x , \zeta \ra | \leq 1 \\
& \iff x \in \leftindex^\bullet(\iota_2(B)).
\end{split}
\end{equation*}
\par 
{\bf (ii)} 
We leave it to the reader to check that $\iota_2 : Y \to X^\bigstar$ is $(\sigma(Y,X),\sigma(X^\bigstar,X))$-continuous (recall \ref{LCTVS.6} and \ref{DUAL.3}).
Since $B$ is $\sigma(Y,X)$-compact, it follows that $\iota_2(B)$ is $\sigma(X^\bigstar,X)$-compact, in particular $\sigma(X^\bigstar,X)$-closed.
As it is also clearly absolutely convex, we infer from the bipolar theorem (C), applied in the dual system $(X,X^\bigstar,\la \cdot,\cdot \ra)$, that
\begin{equation*}
\iota_2(B) = (\leftindex^\bullet \iota_2(B))^\bullet = (\leftindex^\circ B)^\bullet,
\end{equation*}
where the second equality follows from {\bf (i)}.\cqfd
\end{Empty}

\begin{Empty}[Polar topologies]
\label{DUAL.5}
Let $(X,Y,b)$ be a dual system. 
Our aim is to define on $X$ a locally convex vector topology associated with seminorms $p_B$, where $B \subset Y$, defined as follows:
\begin{equation*}
p_B(x) = \sup \{ |b(x,y)| : y \in B \} = \sup \{ |\iota_2(y)(x) | : y \in B \}.
\end{equation*}
It is clear that $p_B$ is a seminorm {\it provided} that $p_B(x) < \infty$ for all $x \in X$.
The latter is equivalent to $B$ being $\sigma(Y,X)$-bounded (\ref{LCTVS.5}(D) applied to $Y[\sigma(Y,X)]$ in place of $X$ and $\la |\iota_1(x)| \ra_{x \in X}$ in place of $\la q_i \ra_{i \in I}$).
We are led to consider 
\begin{equation*}
\frS_\rmbd = \calP(Y) \cap \{ B : B \text{ is $\sigma(Y,X)$-bounded} \} .
\end{equation*}
We say that $\frS \subset \frS_\rmbd$ {\em covers} $Y$ if $Y = \cup \frS$.
In that case the family of seminorms $\la p_B \ra_{B \in \frS}$ is readily separating, whence, it defines a locally convex vector topology on $X$ denoted $T_\frS$.
\begin{enumerate}
\item[(A)] {\it Assume that $\frS_1, \frS_2 \subset \frS_\rmbd$ and each member of $\frS_1$ is contained in a member of $\frS_2$. Then $T_{\frS_1} \subset T_{\frS_2}$. In particular, if $\frS_1 \subset \frS_2 \subset \frS_\rmbd$ then $T_{\frS_1} \subset T_{\frS_2}$.}
\end{enumerate}
\par 
{\it Proof.}
Recalling \eqref{eq.LCTVS.1}, the hypothesis clearly implies that each $T_{\frS_1}$-neighbourhood of 0 is a $T_{\frS_2}$-neighbourhood of 0.\cqfd
\par 
We say that $\frS \subset \frS_\rmbd$ is {\em admissible} if the following three conditions are satisfied:
\begin{itemize}
\item $\frS$ covers $Y$;
\item $(\forall N \in \N \setminus \{0\})\left(\forall (B_1,\ldots,B_N) \in \frS^N \right)(\exists B \in \frS) : \cup_{n=1}^N B_n \subset B$;
\item $(\forall B \in \frS)(\forall t \in \R \setminus \{0\}) : t \cdot B \in \frS$.
\end{itemize}
The topologies $T_\frS$ corresponding to admissible $\frS$ are called {\em polar topologies} probably because of the following observation.
\begin{enumerate}
\item[(B)] {\it Let $\frS \subset \frS_\rmbd$ be admissible. Then $\{ \leftindex^\circ B : B \in \frS \}$ is a local base for the topology $T_\frS$.}
\end{enumerate}
\par 
{\it Proof.}
Given $B \in \frS$ one has $\leftindex^\circ B = X \cap \{ x : p_B(x) \leq 1 \}$.
Whence, $\leftindex^\circ B$ is readily a $T_\frS$-neighbourhood of 0.
Reciprocally, if $U$ is a $T_\frS$-neighbourhood of 0 then there are $B_1,\ldots,B_N \in \frS$ and positive reals $\veps_1,\ldots,\veps_N$ such that 
\begin{equation}
\bigcap_{n=1}^N X \cap \{ x : p_{B_n}(x) \leq \veps_n \} \subset U.
\end{equation}
Upon noticing that $X \cap \{ x : p_{B_n}(x) \leq \veps_n \} = \veps_n \cdot \leftindex^\circ B_n$ we infer (using (B)) that $U$ contains the set
\begin{equation*}
\bigcap_{n=1}^N \veps_n  \cdot \leftindex^\circ B_n = \bigcap_{n=1}^N \leftindex^\circ(\veps_n^{-1} \cdot B_n) = \leftindex^\circ ( \bigcup_{n=1}^N \veps_n^{-1} \cdot B_n ) \supset \leftindex^\circ B,
\end{equation*}
where $B \in \frS$ is a set containing $\cup_{n=1}^N \veps_n^{-1}B_n$, whose existence follows from the last two conditions in the definition of $\frS$ being admissible.\cqfd
\par 
We now give the three main examples of polar topologies defined on $X$, relative to a dual system $(X,Y,b)$.
\begin{enumerate}
\item[(1)] The {\em weak topology} of $X$ is $T_{\frS_\rms}$ where $\frS_\rms = \frS_\rmbd \cap \{ F : F \text{ is finite}\}$. 
One easily checks that $T_{\frS_\rms}$ coincides with $\sigma(X,Y)$ defined in \ref{DUAL.3} and that $\frS_\rms$ is admissible.
\item[(2)] The {\em Mackey topology} of $X$, denoted $\tau(X,Y)$, is the polar topology $T_{\frS_\rmcomp}$ where 
\begin{equation*}
\frS_\rmcomp = \frS_{\rmbd} \cap \{ C : C \text{ is absolutely convex and $\sigma(Y,X)$-compact} \}.
\end{equation*}
We claim that $\frS_\rmcomp$ is admissible.
The first and third conditions of the definition are trivially satisfied but there is something to say about the second one: it is a consequence of \ref{LCTVS.10}(A).
\item[(3)] The {\em strong topology} of $X$, denoted $\beta(X,Y)$, is the polar topology $T_{\frS_\rmbd}$.
The collection $\frS_\rmbd$ is admissible, since a scalar multiple of a bounded set and a finite union of bounded sets are bounded as well.
If $X[\calT]$ is a locally convex topological vector space then the topology $\beta(X^*,X)$ defined in \ref{LCTVS.7} coincides with the topology defined here on $X^*$ with respect to the dual system $(X^*,X,\la \cdot,\cdot \ra)$ as follows from \ref{DUAL.6}(F) (\ie that a subset of $X$ is $\calT$-bounded iff it is $\sigma(X,X^*)$-bounded).
\end{enumerate}
\par 
We observe that $\frS_\rms \subset \frS_\rmcomp \subset \frS_\rmbd$, therefore, we infer from (A) that
\begin{equation}
\label{eq.DUAL.1}
\sigma(X,Y) \subset \tau(X,Y) \subset \beta(X,Y).
\end{equation}
\end{Empty}

\begin{Empty}[Topologies compatible with a dual system]
\label{DUAL.6}
Let $(X,Y,b)$ be a dual system and $\calT$ a locally convex vector topology on $X$.
We say that $\calT$ is {\em compatible} with the dual system $(X,Y,b)$ if $X[\calT]^* = \iota_2(Y)$.
\begin{enumerate}
\item[(A)] {\it The weak topology $\sigma(X,Y)$ is compatible with the dual system $(X,Y,b)$.}
\end{enumerate}
\par 
{\it Proof.}
This is a rephrasing of \ref{DUAL.3}(B).\cqfd
\begin{enumerate}
\item[(B)] {\it The Mackey topology $\tau(X,Y)$ is compatible with the dual system $(X,Y,b)$.}
\end{enumerate}
\par 
{\it Proof.}
In this proof, we denote by $\circ$ the polarity with respect to the dual system $(X,Y,b)$ and we denote by $\bullet$ the polarity with respect to the dual system $(X,X^\bigstar,\la \cdot,\cdot \ra)$.
Recalling \ref{DUAL.5}(3) that $\frS_\rmcomp$ is admissible, we infer from \ref{DUAL.5}(B) that $\calB = \calP(X) \cap \{ \leftindex^\circ C : C \in \frS_\rmcomp \}$ is a local base for the Mackey topology $\tau(X,Y)$.
Thus, it follows from \ref{DUAL.4}(E) that
\begin{equation*}
\begin{split}
X[\tau(X,Y)]^* & = \bigcup \{ V^\bullet : V \in \calB \} \\
& = \bigcup \left\{ (\leftindex^\circ C)^\bullet : C \in \frS_\rmcomp \right\} \\
& = \bigcup \{ \iota_2(C) : C \in \frS_\rmcomp \} \\
& = \iota_2 \left( \bigcup \frS_\rmcomp \right) \\
& = \iota_2(Y),
\end{split}
\end{equation*}
where we have used \ref{DUAL.4}(F).\cqfd
\begin{enumerate}
\item[(C)] {\bf (Mazur theorem)} {\it If $\calT_1$ and $\calT_2$ are two locally convex vector topologies on $X$ both compatible with the dual system $(X,Y,b)$ and if $A \subset X$ is absolutely convex then $A$ is $\calT_1$-closed if and only if $A$ is $\calT_2$-closed.}
\end{enumerate}
\par 
{\it Proof.}
Define $\calF_A = X^\bigstar \cap \left\{ x^\bigstar : \sup_{x \in A} | \la x , x^\bigstar \ra | \leq 1 \right\}$ and, corresponding to a locally convex vector topology $\calT$ on $X$, define $\calF_{A,\calT} = \calF_A \cap X[\calT]^*$.
It follows from the Hahn-Banach theorem \ref{LCTVS.8}(B) that if $A$ is $\calT$-closed then
\begin{equation}
\label{eq.DUAL.2}
A = \bigcap \big\{ \{ x : | \la x , x^\bigstar \ra | \leq 1 \} : x^\bigstar \in \calF_{A,\calT} \big\}.
\end{equation}
\par 
Since $X[\calT_1]^*=X[\calT_2]^*$, by hypothesis, we have $\calF_{A,\calT_1} = \calF_{A,\calT_2}$.
Accordingly, if $A$ is $\calT_1$-closed then it follows from \eqref{eq.DUAL.2} that it equals the intersection of a collection of $\calT_2$-closed sets $X \cap \{ x : | \la x , x^\bigstar \ra \ \leq 1 \}$ corresponding to $x^\bigstar \in \calF_{A,\calT_2}$, hence, $A$ is itself $\calT_2$-closed.
Upon switching the role of $\calT_1$ and $\calT_2$ the proof is complete.\cqfd

\begin{enumerate}
\item[(D)] {\it For every $\sigma(Y,X)$-bounded set $B \subset Y$, $\leftindex^\circ B$ is absorbing and $p_B = q_{\leftindex^\circ B}$.\footnote{Recall \ref{DUAL.5} for the definition of $p_B$ and \ref{LCTVS.3} for the definition of the Minkowski functional $q_{\leftindex^\circ B}$}} 
\end{enumerate}
\par 
{\it Proof.}
Recall that $p_B(x) < \infty$ for all $x \in X$, since $B$ is $\sigma(Y,X)$-bounded.
Furthermore, for all $x \in X$ and all $t > 0$ we have
\begin{equation*}
\begin{split}
p_B(x) \leq t & \iff (\forall y \in B) : |b(x,y)| \leq t \\
& \iff (\forall y \in B) : |b(t^{-1} \cdot x,y)| \leq 1 \\
& \iff t^{-1} \cdot x \in \leftindex^\circ B.
\end{split}
\end{equation*}
This completes the proof. \cqfd
\begin{enumerate}
\item[(E)] {\bf (Mackey-Arens theorem)} {\it Let $(X,Y,b)$ be a dual system and let $\calT$ be a locally convex vector topology on $X$.
The following are equivalent.
\begin{enumerate}
\item[(a)] $\calT$ is compatible with the dual system $(X,Y,b)$.
\item[(b)] $\sigma(X,Y) \subset \calT \subset \tau(X,Y)$.
\end{enumerate}
}
\end{enumerate}
\par 
{\it Proof that $(a) \Rightarrow (b)$.}
Assume $\calT$ is compatible.
The inclusion $\sigma(X,Y) \subset \calT$ has already been noted \ref{DUAL.3}(A).
Let $\calB$ be a local base for $\calT$ consisting of absolutely convex $\calT$-closed neighbourhoods of the origin.
Consider $\frS = \calP(Y) \cap \{ V^\circ : V \in \calB \}$.
Since $\calT$ is compatible, $\frS$ covers $Y$ and, therefore, $T_\frS$ is a locally convex vector topology on $X$.
We shall show that
\begin{equation*}
\calT = T_\frS \subset \tau(X,Y)
\end{equation*}
and the proof will be complete.
\par 
Let $V \in \calB$.
As $V$ is absolutely convex and $\calT$-closed and $\calT$ is compatible with $(X,Y,b)$, it follows from (A) and Mazur's theorem (C) that $V$ is $\sigma(X,Y)$-closed.
Thus, we infer from the bipolar theorem \ref{DUAL.4}(C) that $V = \leftindex^\circ(V^\circ) = \leftindex^\circ B$, where we have abbreviated $B = V^\circ$.
It follows from (D) that $q_V = p_B$.
Since $\la q_V \ra_{V \in \calB}$ generates $\calT$ and $\la p_B \ra_{B \in \frS}$ generates $T_\frS$, we have shown that $\calT = T_\frS$.
\par 
It remains to show that $T_\frS \subset \tau(X,Y)$.
Let $B \in \frS$ and choose $V \in \calB$ such that $B = V^\circ$.
It follows from Banach-Alaoglu's theorem \ref{DUAL.4}(D) that $V^\bullet$ is $\sigma(X^*,X)$-compact where the polarity $\bullet$ is understood in the dual system $(X,X[\calT]^*,\la \cdot,\cdot \ra)$.
Since $\calT$ is compatible with $(X,Y,b)$, this is easily seen to be equivalent to $V^\circ$ being $\sigma(Y,X)$-compact.
Since $V^\circ$ is clearly absolutely convex, one concludes that $V^\circ \in \frS_\rmcomp$.
Thus, $\frS \subset \frS_\rmcomp$ and, in turn, $T_\frS \subset T_{\frS_\rmcomp} = \tau(X,Y)$.\cqfd
\par 
{\it Proof that $(b) \Rightarrow (a)$.}
As $\sigma(X,Y) \subset \calT$, it follows from (A) that $\iota_2(Y) = X[\sigma(X,Y)]^* \subset X[\calT]^*$.
As $\calT \subset \tau(X,Y)$, it follows from (B) that $X[\calT]^* \subset X[\tau(X,Y)]^* = \iota_2(Y)$. \cqfd
\begin{enumerate}
\item[(F)] {\bf (Uniform boundedness principle)} {\it If $X[\calT]$ is a locally convex topological vector space and $B \subset X$ then the following are equivalent.
\begin{enumerate}
\item[(a)] $B$ is $\calT$-bounded.
\item[(b)] $B$ is $\sigma(X,X^*)$-bounded.
\end{enumerate}
}
\end{enumerate}
\par 
{\it Proof.}
Since $\sigma(X,X^*) \subset \calT$, one trivially has $(a) \Rightarrow (b)$.
We turn to proving the reverse implication.
Let $B \subset X$ and let $\la q_i \ra_{i \in I}$ be a family of seminorms that generates the topology $\calT$.
In view of \ref{LCTVS.5}(D), we ought to show that $\sup \{ q_i(x) : x \in B \} < \infty$ for all $i \in I$.
Fix $i \in I$ and define
\begin{equation*}
Y_i = X^\bigstar \cap \left\{ f : |f| \leq a \cdot q_i \text{ for some } a > 0 \right\} \subset X^*.
\end{equation*}
This is readily a vector space.
For $f \in Y_i$ we define 
\begin{equation*}
\|f\|_i = \sup \{ | \la x , f \ra| : x \in X \text{ and } q_i(x) \leq 1 \}.
\end{equation*}
We leave to the reader the routine verification that $Y_i[\|\cdot\|_i]$ is a Banach space and that $|\la x , f\ra| \leq \|f\|_i q_i(x)$ for all $x \in X$ and all $f \in Y_i$.
It follows that if $x \in X$ then $\rmev(x)|_{Y_i} : Y_i \to \R$ is $\|\cdot\|_i$-continuous and that $\|\rmev(x)|_{Y_i}\|_i^* \leq q_i(x)$.
In fact, $\|\rmev(x)|_{Y_i}\|_i^* = q_i(x)$ (since there exists $f \in X^\bigstar$ such that $f(x) = q_i(x)$ and $|f| \leq q_i$, according to the Hahn-Banach theorem \ref{LCTVS.8}(D).
We see that the collection $\{\rmev(x)|_{Y_i} : x \in B \} \subset Y_i^*$ is pointwise bounded, \ie $\sup \{ | \la x , f \ra | : x \in B \} < \infty$ for every $f \in Y_i$, since $Y_i \subset X^*$ and $B$ is $\sigma(X,X^*)$-bounded.
By the usual uniform bounded principle in the context of Banach spaces (see \eg \cite[2.4 and 2.5]{RUDIN}), we infer that $\sup \{ \|\rmev(x)|_{Y_i}\|_i^* : x \in B \} < \infty$. \cqfd
\begin{enumerate}
\item[(G)] {\it Let $X[\calT]$ be a locally convex topological vector space. The strong topology $\beta(X^*,X)$ defined in \ref{LCTVS.7} coincides with the strong topology $\beta(X^*,X)$ relative to the dual system $(X^*,X,\la \cdot , \cdot \ra)$ defined in \ref{DUAL.5}(3).}
\end{enumerate}
\par 
{\it Proof.}
The first is defined as the topology of uniform convergence on the $\calT$-bounded subsets of $X$.
The second is defined as the topology of uniform convergence on the $\sigma(X^*,X)$-bounded subsets of $X$.
Thus, the conclusion is an immediate consequence of (F).\cqfd
\begin{enumerate}
\item[(H)] {\bf (Characterization of semireflexivity)} {\it Let $X[\calT]$ be a locally convex topological vector space. The following are equivalent.
\begin{enumerate}
\item[(a)] $X[\calT]$ is semireflexive\footnote{Recall \ref{LCTVS.9} for the definition of semireflexivity.}.
\item[(b)] $\beta(X^*,X) = \tau(X^*,X)$.
\item[(c)] Each $\calT$-bounded subset of $X$ is contained in a $\sigma(X,X^*)$-compact subset of $X$.
\end{enumerate}
}
\end{enumerate}
\par 
In the above statement and the proof below, reference is made to the dual system $(X^*,X,\la \cdot,\cdot \ra)$ which is the symmetric of $\new{(X,X^*,\la \cdot,\cdot \ra)}$ (recall \ref{DUAL.2}(2)).
The polarity with respect to this dual system will simply be denoted by $\circ$.
\par 
{\it Proof that $(a) \iff (b)$.}
If $\iota_2 : X \to (X^*)^\bigstar$ refers to $(X^*,X,\la \cdot,\cdot \ra)$ and if $\widehat{\rmev} : X \to (X^*)^\bigstar$ is defined by same formula as in \ref{LCTVS.9} (forgetting about the strong continuity of each $\rmev(x)$) then, clearly, $\iota\new{_2}(X)=\widehat{\rmev}(X)$ and $\widehat{\rmev}(X)=\rmev(X)$.
Thus,
\begin{equation*}
\begin{split}
X \text{ is semireflexive} & \iff X^*[\beta(X^*,\new{X})]^*=\rmev(X)=\widehat{\rmev}(X) \\
& \iff X^*[\beta(X^*,\new{X})]^*=\iota_2(X) \\
& \iff \beta(X^*,X) \text{ is compatible with } (X^*,X,\la \cdot,\cdot \ra)\\
& \iff \sigma(X^*,X) \subset \beta(X^*,X) \subset \tau(X^*,X) \\
& \iff \beta(X^*,X) = \tau(X^*,X).
\end{split}
\end{equation*}
We have used the Mackey-Arens theorem (E) and \eqref{eq.DUAL.1}.\cqfd
\par 
{\it Proof that $(b) \Rightarrow (c)$.}
Let $B \subset X$ be $\calT$-bounded.
Then $B$ is trivially $\sigma(X,X^*)$-bounded, hence, $B^\circ$ is a $\beta(X^*,X)$-neighbourhood of the origin in $X^*$ (recall the first two sentences of the proof of \ref{DUAL.5}(B)).
Therefore, $B^\circ$ is a $\tau(X^*,X)$-neighbourhood of the origin, by (b).
According to \ref{DUAL.5}(B) (applied to the symmetric of the current dual system), there exists an absolutely convex $\sigma(X,X^*)$-compact set $C \subset X$ such that $C^\circ \subset B^\circ$.
Thus, $B \subset \leftindex^\circ(B^\circ) \subset \leftindex^\circ(C^\circ) = C$, where the first inclusion is trivial, the second one follows from \ref{DUAL.4}(B)(a), and the last one is an application of the bipolar theorem \ref{DUAL.4}(C).\cqfd
\par 
{\it Proof that $(c) \Rightarrow (b)$.}
In view of \ref{DUAL.5}(A), it is sufficient to show that each $\sigma(X,X^*)$-bounded subset of $X$ is contained in an absolutely convex $\sigma(X,X^*)$-compact set.
Let $B \subset X$ be $\sigma(X,X^*)$-bounded.
By the uniform boundedness principle (F), $B$ is $\calT$-bounded.
According to \ref{LCTVS.5}(C)(b), $\rmabsconv(B)$ is $\calT$-bounded.
By (c), there exists a $\sigma(X,X^*)$-compact set $S$ such that $\rmabsconv(B) \subset S$.
It follows from \ref{LCTVS.10}(B) that $C = \rmclos_{\sigma(X,X^*)}(\rmabsconv(B))$ is absolutely convex.
Since $C \subset S$, $C$ is $\sigma(X,X^*)$-compact.
As it contains $B$, the proof is complete.\cqfd
\end{Empty}

\section{Barrelled spaces}
\label{app.BARREL}

\begin{Empty}
\label{PR.7}
In a topological vector space $X[\calT]$ a set $A \subset X$ is called a {\em barrel} if it \new{is} absolutely convex, closed, and absorbing. We say that a locally convex topological vector space $X[\calT]$ is {\em barrelled} if each barrel in $X$ is a neighborhood of 0.
\end{Empty}

\begin{Empty}
\label{PR.9}
Given a topological space $X[\calT]$ and a function $f : X \to \R$ we say that $f$ is {\em lower-semicontinuous} if $X \cap \{ x : f(x) \leq r \}$ is closed for every $r \in \R$. If $\la f_i \ra_{i \in I}$ is a non-empty family of continuous functions $X \to \R$ such that $f(x) :=  \sup \{ f_i(x) : i \in I \} < \infty$ for every $\new{x \in X}$ then $f : X \to \R$ is lower-semicontinuous. Indeed, for each $r \in \R$ the set $X \cap \{ x : f(x) \leq r \} = \cap_{i \in I} X \cap \{ x : f_i(x) \leq r \}$ is closed. 
\end{Empty}

\begin{Empty}
\label{PR.10}
Given a locally convex topological vector space $X[\calT]$ we say that $B \subset X[\calT]^*$ is {\em equicontinuous} if $\cap_{x^* \in B} X \cap \{ x : | \la x,x^* \ra | < 1 \}$ is a neighbourhood of 0. Notice that this set is contained in $B^\circ$, the polar of $B$, see \ref{DUAL.4}.
\end{Empty}

\begin{Proposition}
\label{PR.8}
Let $X[\calT]$ be a locally convex topological vector space. The following are equivalent.
\begin{abcenum}
\item \label{PR.8.a1} $X[\calT]$ is barrelled;
\item \label{PR.8.a2} Every lower semicontinuous seminorm on $X[\calT]$ is continuous;
\item \label{PR.8.a3} Every weakly* bounded subset of $X[\calT]^*$ is equicontinuous.
\end{abcenum}
\end{Proposition}

\begin{proof}[Proof that \refmod{\ref{PR.8.a1}} $\Rightarrow$ \refmod{\ref{PR.8.a2}}]
Let $\lno \cdot \rno$ be a lower-semicontinuous seminorm on $X$. Observe that $A = X \cap \{ x : \lno x \rno \leq 1 \}$ is a barrel. Consequently, $A$ is a neighbourhood of 0. It now easily follows from the triangle inequality that $\lno \cdot \rno^{-1}(O)$ is open in $X$ whenever $O \subset \R$ is open.
\end{proof}
\begin{proof}[Proof that \refmod{\ref{PR.8.a2}} $\Rightarrow$ \refmod{\ref{PR.8.a3}}]
Let $\emptyset \neq B \subset X[\calT]^*$ be weakly* bounded and note that the formula $\lno x \rno_B := \sup \{ | \la x,x^* \ra | : x^* \in B \} < \infty$, $x \in X$, defines a lower-semicontinuous seminorm on $X$. The continuity of $\lno \cdot \rno_B$ is equivalent to the equicontinuity of $B$.
\end{proof}
\begin{proof}[Proof that \refmod{\ref{PR.8.a3}} $\Rightarrow$ \refmod{\ref{PR.8.a1}}]
Let $A$ be a barrel in $X$.
Since $A$ is $\calT$-closed and $\sigma(X,X^*)$ is compatible with the duality $(X,X^*,\la\cdot,\cdot\ra)$, $A$ is also $\sigma(X,X^*)$-closed, according to Mazur's theorem \ref{DUAL.6}(C).
Therefore, $A = \leftindex^\circ(A^\circ)$, by the bipolar theorem \ref{DUAL.4}(C).
Since $A$ is absorbing, $A^\circ$ is $\sigma(X^*,X)$-bounded.
Thus, by hypothesis, \new{it} is equicontinuous, \ie $\cap_{x^* \in A^\circ} X \cap \{ x : |\la x,x^*\ra| < 1 \} \subset \leftindex^\circ(A^\circ) = A$ is a $\calT$-neighbourhood of 0.
\end{proof}

\begin{Empty}
\label{PR.12}
Condition \ref{PR.8}\refmod{\ref{PR.8.a3}} is a version of the Banach-Steinhaus theorem, as we now explain. Assume $X[\|\cdot\|]$ is a normed space. The reader will happily check that $B \subset X^*$ is equicontinuous if and only if $\sup_{x^* \in B} \|x^*\|^* < \infty$ where $\|\cdot\|^*$ is the usual dual norm on the dual space $X^*$.
\par 
Therefore, if $X[\| \cdot \|]$ is a barrelled normed space and $\la x^*_i \ra_{i \in I}$ is a family in $X^*$ such that $\sup_{i \in I} | \la x,x_i^* \ra| < \infty$ for every $x \in X$ then $\sup_{i \in I} \|x_i^*\|^* < \infty$. In particular if $\la x_n^* \ra_n$ is a sequence in $X^*$ and the limit $x^*(x) := \lim_n \la x,x_n^* \ra$ exists for every $x \in X$ then the linear form $x^*$ is continuous. 
\par 
Banach spaces are barrelled normed spaces, according to \ref{PR.11}, but not all barrelled normed spaces are complete. 
\end{Empty}

\begin{Proposition}
\label{PR.11}
Let $X$ be a locally convex topological vector space. If $X$ is not meagre in itself then $X$ is barrelled. In particular, every Banach space and every Fr\'echet space is barrelled.
\end{Proposition}

\begin{proof}
Let $A$ be a barrel in $X$. Since $A$ is absorbing, $X = \cup_{j \in \N} j \cdot A$. As each $j \cdot A$ is closed and $X$ is not meagre in itself, some $j\cdot A$ has non-empty interior and, therefore, so does $\frac{1}{2} A$. Thus, there exists an open set $U$ such that $x \in U \subset \frac{1}{2} A$. Since $\frac{1}{2} A$ is symmetric, we also have $-x \in -U \subset \frac{1}{2} A$ and, therefore, $0 = x + (-x) \subset U + (-U) \subset \frac{1}{2} A + \frac{1}{2} A \subset A$, where the last inclusion follows from the convexity of $A$. Since $U + (-U)$ is open, the proof is complete.
\end{proof}

\begin{Theorem}[Another version of Banach-Alaoglu]
\label{PR.100}
Let $X$ be a barrelled locally convex topological vector space. If $B \subset X^*$ is weakly* bounded and weakly* closed then $B$ is weakly* compact.
\end{Theorem}

\begin{proof}
Letting $C = \overline{\rmabsconv}^{\sigma(X^*,X)}(B)$ it suffices to show that $C$ is weakly* compact.
Note that $\leftindex^\circ (C^\circ) = C$, according to the bipolar theorem \ref{DUAL.4}(C). As follows from the Banach-Alaoglu theorem \ref{DUAL.4}(D), it is now enough to establish that $C^\circ$ is a neighbourhood of 0 in $X$. Since $C$ is weakly* bounded, according to \ref{LCTVS.10}(C), it ensues from \ref{PR.8}\refmod{\ref{PR.8.a3}} that $C$ is equicontinuous, thus, $\cap_{x^* \in C} X \cap \{ x : |\la x,x^* \ra| < 1 \} \subset C^\circ$ is a \new{neighbourhood of 0 in $X$}.
\end{proof}

\section{Bornological spaces}
\label{app.BORN}

\begin{Empty}
\label{PR.5}
In a topological vector space $X[\calT]$ a set $A \subset X$ is called {\em bornivorous} if it absorbs every bounded set. In other words for each bounded set $B \subset X$ there exists $t > 0$ such that $B \subset t \cdot A$. It follows from \ref{PR.1} (\ie the definition of bounded set) that each neighborhood of 0 is bornivorous. 
We say that a locally convex topological vector space $X[\calT]$ is {\em bornological} if each convex balanced bornivorous set is a neighborhood of 0. 
\end{Empty}

\begin{Proposition}
\label{PR.13}
Let $X[\calT]$ be a locally convex topological vector space. Consider the following conditions:
\begin{abcenum}
\item \label{PR.13.a1} $X[\calT]$ is bornological;
\item \label{PR.13.a2} Every bounded seminorm on $X[\calT]$ (that is, bounded on each $\calT$-bounded subset of $X$) is continuous;
\item \label{PR.13.a3} Every strongly bounded subset of $X[\calT]^*$ is equicontinuous.
\end{abcenum}
We have \refmod{\ref{PR.13.a1}} $\Leftrightarrow$ \refmod{\ref{PR.13.a2}} $\Rightarrow$ \refmod{\ref{PR.13.a3}}.
\end{Proposition}

\par
{\it Proof that \refmod{\ref{PR.13.a1}} $\Rightarrow$ \refmod{\ref{PR.13.a2}}.} Let $\lno \cdot \rno$ be a bounded seminorm on $X$ and observe that $A = X \cap \{ x : \lno x \rno \leq 1 \}$ is convex, balanced, and bornivorous. Therefore, $A$ is a neighbourhood of 0 and the continuity $\lno \cdot \rno$ follows from the triangle inequality.\cqfd
\par 
{\it Proof that \refmod{\ref{PR.13.a2}} $\Rightarrow$ \refmod{\ref{PR.13.a3}}.} Let $\emptyset \neq B \subset X[\calT]^*$ be strongly bounded and observe that the formula $\lno x \rno_B := \sup \{ | \la x,x^* \ra | : x^* \in B \} < \infty$, $x \in X$, defines a bounded seminorm on $X$. The continuity of $\lno \cdot \rno_B$ is equivalent to the equicontinuity of $B$.\cqfd
\par 
{\it Proof that \refmod{\ref{PR.13.a2}} $\Rightarrow$ \refmod{\ref{PR.13.a1}}.} Let $A$ be convex, balanced, and bornivorous. In particular, $A$ is absorbing and its Minkowski functional $q_A$ is a seminorm of $X$. Furthermore, $q_A$ is bounded, since $A$ is bornivorous.  Therefore, $q_A$ is continuous and $A \supset X  \cap \{ x : q_A(x) < 1 \}$ is a neighbourhood of 0.\cqfd

\begin{Remark}
Arguing as in the proof of \ref{PR.8} \refmod{\ref{PR.8.a3}} $\Rightarrow$ \refmod{\ref{PR.8.a1}} one can show that if every strongly bounded subset of $X[\calT]^*$ is equicontinuous and $A \subset X$ is a convex, balanced, bornivorous, {\it and closed} then $A$ is a neighbourhood of 0. 
\end{Remark}

\begin{Proposition}
\label{PR.21}
Let $X$ be a Fr\'echet-Urysohn topological vector space and $\la y_k \ra_k$ a null sequence.
There are positive integers $j(1) < j(2) < \cdots$ and $k(1) < k(2) < \cdots$ such that $\lim_n j(n) \cdot y_{k(n)} = 0$.
\end{Proposition}

\par 
{\it Proof.}
Let $x=0$, $x_j=0$ for all $j \in \N$, and $x_{j,k} = j \cdot y_k$ for all $(j,k) \in \N \times \N$.
Apply \ref{PR.3}(D) (recall that $X$ is sequential, by \ref{PR.3}(A)).\cqfd

\begin{Proposition}
\label{PR.6}
Let $X[\calT]$ be a locally convex topological vector space. 
If $X[\calT]$ is Fr\'echet-Urysohn then it is bornological. 
In particular, if the topology $\calT$ is metrizable then $X[\calT]$ is bornological.
\end{Proposition}

\par 
{\it Proof.}
Assume if possible that $A \subset X$ is absolutely convex and bornivorous, yet not a neighbourhood of 0. 
Let $\calU$ denote the set of open sets containing 0. 
Thus, with each $U \in \calU$ there corresponds $x_U \in U \setminus A$. 
Since $0 \in \rmclos B$, where $B = \{ x_U : U \in \calU \}$, and $X$ is Fr\'echet-Urysohn, there is a sequence $\la U_k \ra_k$ in $\calU$ such that $\lim_k x_{U_k}=0$. 
Abbreviating $y_k = x_{U_k}$ and letting $j(1) < j(2) < \cdots$ and $k(1) < k(2) < \cdots$ be associated with $\la y_k \ra_\new{k}$ in \ref{PR.21} we infer that $\la j(n) \cdot y_{k(n)} \ra_n$ is a null sequence. 
Thus, the set $\{ j(n) \cdot  y_{k(n)} : n \in \N \}$ is bounded and, consequently, it is absorbed by $A$: There exists $t > 0$ such that $j(n) \cdot y_{k(n)} \in t \cdot A$ for all $n$. 
If $n$ is so large that $t^{-1} \cdot j(n) \geq 1$ then $y_{k(n)} \in A$, since $A$ is balanced, a contradiction.\cqfd

\par 
The particular case appears to be classical, see for instance \cite[7.3.2]{EDWARDS}, whereas the proposition itself was noted in \cite[Appendix 1(4)]{YAMAMURO}.

\section{Making continuous choices}
\label{app.MCC}

\begin{Empty}
\label{MCC.2}
We choose to use notations compatible with the applications in Section \ref{sec.AET}.
Let $\bF$ and $\bE$ be two topological spaces.
We say that $\Gamma : \bF \to \calP(\bE)$ is {\em lower semi-continuous} if $\bF \cap \{ F : \Gamma(F) \cap O \neq \emptyset\}$ is open whenever $O \subset \bE$ is open.
\begin{enumerate}
\item[(A)] {\it Assume that $\Gamma : \bF \to \calP(\bE)$ is lower semi-continuous and define $\bar{\Gamma} : \bF \to \calP(\bE)$ by $\bar{\Gamma}(F) = \rmclos \Gamma(F)$. Then $\bar{\Gamma}$ is lower semi-continuous as well.}
\end{enumerate}
\par 
{\it Proof.}
Let $O \subset \bE$ be open and let $F_0$ be a member of $\bF \cap \{F : \bar{\Gamma}(F) \cap O \neq \emptyset \}$, \ie $(\rmclos \Gamma(F_0)) \cap O \neq \emptyset$.
Therefore, $\Gamma(F_0) \cap O \neq \emptyset$.
Since $\Gamma$ is lower semi-continuous, there is a neighbourhood $U$ of $F_0$ in $\bF$ such that $\Gamma(F) \cap O \neq \emptyset$ for all $F \in U$.
As $\bar{\Gamma}(F) \cap O \neq \emptyset$ for all such $F$, the proof is complete. \cqfd

\par 
In what follows, $\bE$ is a Fr\'echet space.
We fix on $\bE$ a metric $d$ that is compatible with its topology, is translation invariant, is complete, and so that open $d$-balls are convex.
Such a metric $d$ exists, recall \ref{LCTVS.11}.
\begin{enumerate}
\item[(B)] {\it Assume that $\bF$ is a metric space, that $\bE$ is a Fr\'echet space, and that $\Gamma : \bF \to \calP(\bE)$ satisfies the following conditions.
\begin{enumerate}
\item[(a)] For all $F \in \bF$, $\Gamma(F)$ is non-empty and convex.
\item[(b)] $\Gamma$ is lower semi-continuous.
\end{enumerate}
Then for every $\veps > 0$ there exists a continuous (in fact, locally Lipschitz) $\gamma : \bF \to \bE$ such that $\rmdist(\gamma(F),\Gamma(F)) < \veps$ for all $F \in \bF$.
}
\end{enumerate}
\par 
{\it Proof.}
Fix $\veps > 0$ and consider the open ball $V \subset \bE \cap \{ v : d(0,v) < \veps \}$.
Recall \ref{LCTVS.11} that the metric can be chosen so that $V$ is convex.
With $v \in \bE$ associate $O_v = \bF \cap \{ F : \Gamma(F) \cap (v + V) \neq \emptyset \}$.
Note that each $O_v$ is open, since $\Gamma$ is lower semi-continuous, and that $\bF = \cup_{v \in \bE} O_v$, since $\Gamma(F) \neq \emptyset$ for all $F \in \bF$.
Let $\la \psi_i \ra_{i \in I}$ be a partition of unity subordinated to $\la O_v \ra_{v \in \bE}$, see \eg \cite[Chap. IX 5.3]{DUGUNDJI}.
For each $i \in I$, choose $v_i \in \bE$ so that $\rmspt(\psi_i) \subset O_{v_i}$.
Define 
\begin{equation*}
\gamma : \bF \to \bE : F \mapsto \sum_{i \in I} \psi_i(F) v_i.
\end{equation*}
Thus, $\gamma$ is continuous.
Let $F \in \bF$ and $I_F = I \cap \{ i : \psi_i(F) \neq 0 \}$.
If $i \in I_F$ then $F \in \rmspt(\psi_i) \subset O_{v_i}$, whence, $\Gamma(F) \cap (v_i + V) \neq \emptyset$.
One may then choose $w_i \in \Gamma(F) \cap (v_i + V)$ and define $w = \sum_{i \in I_F} \psi_i(F)w_i$.
Observe that $w \in \Gamma(F)$, since $\Gamma(F)$ is convex, and that $w - \gamma(F) = \sum_{i \in I_F} \psi_i(F)(w_i-v_i) \in V$, since $V$ is convex.
Accordingly, $\rmdist(\gamma(F),\Gamma(F)) < \veps$, by the choice of $V$ and the translation invariance of $d$.\cqfd
\begin{enumerate}
\item[(C)] {\it Assume that $\bF$ is a metric space, $\bE$ is a Fr\'echet space, $\veps > 0$, and that $\Gamma : \bF \to \calP(\bE)$ and $\gamma : \bF \to \bE$ satisfy the following conditions.
\begin{enumerate}
\item[(a)] $\gamma$ is continuous.
\item[(b)] $\Gamma$ is lower semi-continuous.
\end{enumerate}
Then $\hat{\Gamma} : \bF \to \calP(\bE)$ defined by $\hat{\Gamma}(F) = \Gamma(F) \cap \{ v : d(\gamma(F),v) < \veps \}$ is lower semi-continuous.
}
\end{enumerate}
\par 
{\it Proof.}
Let $O \subset \bE$ be open and abbreviate $U = \bF \cap \{ F : \hat{\Gamma}(F) \cap O \neq \emptyset \}$.
We ought to show that $U$ is open.
If $U = \emptyset$ then we are done.
If not, let $F_0 \in U$, \ie $\Gamma(F_0) \cap B(\gamma(F_0),\veps) \cap O \neq \emptyset$.
There exists $0 < \tau < 1$ such that $\Gamma(F_0) \cap B(\gamma(F_0),\tau \veps) \cap O \neq \emptyset$.
Consider the open set $\tilde{O} = B(\gamma(F_0),\tau \veps) \cap O$ and define $\tilde{U} = \bF \cap \{ F : \Gamma(F) \cap \tilde{O} \neq \emptyset\}$.
Since $\Gamma$ is lower semi-continuous, $\tilde{U}$ is open.
Note that $\tilde{V} = \gamma^{-1} [ B(\gamma(F_0),(1-\tau)\veps)]$ is also open, as $\gamma$ is continuous.
It is clear that $F_0 \in \tilde{U} \cap \tilde{V}$ and we observe, furthermore, that $\tilde{U} \cap \tilde{V} \subset U$.
Indeed, if $F \in \tilde{U}$ then there exists $v \in \Gamma(F) \cap B(\gamma(F_0),\tau \veps) \cap O$ and, since $F \in \tilde{V}$, $d(v,\gamma(F)) \leq d(v,\gamma(F_0)) + d(\gamma(F_0),\gamma(F)) < \tau \veps + (1-\tau)\veps < \veps$, \ie $v \in \Gamma(F) \cap B(\gamma(F),\veps) \cap O$.
Thus, $U$ is a neighbourhood of $F_0$ and, since $F_0$ is arbitrary, the proof is complete.\cqfd
\begin{enumerate}
\item[(D)] {\bf (Michael's selection theorem) }{\it Assume that $\bF$ is a metric space, $\bE$ is a Fr\'echet space, and that $\Gamma : \bF \to \calP(\bE)$ satisfies the following conditions.
\begin{enumerate}
\item[(a)] For all $F \in \bF$, $\Gamma(F)$ is non-empty, convex, and closed. 
\item[(b)] $\Gamma$ is lower semi-continuous.
\end{enumerate}
Then there exists a continuous $\gamma : \bF \to \bE$ such that $\gamma(F) \in \Gamma(F)$ for all $F \in \bF$.
}
\end{enumerate}
\par 
{\it Proof.}
Let $\la \veps_j \ra_{j \in \N}$ be a sequence of positive real numbers such that $\sum_{j \in \N} \veps_j < + \infty$.
We define inductively a sequence $\la \gamma_j \ra_{j \in \N}$ of continuous maps $\bF \to \bE$ by letting $\gamma_0$ be associated with $\Gamma$ and $\veps_0$ in \new{(B)} and letting $\gamma_{j+1}$ be associated with $\Gamma_{j+1}$ and $\veps_{j+1}$ in \new{(B)}, where $\Gamma_{j+1}(F) = \Gamma(F) \cap B(\gamma_j(F),\veps_{j})$, $F \in \bF$.
In order to check that (B) applies to $\Gamma_{j+1}$ we note that (a) $\Gamma_{j+1}(F)$ is convex, since so are $\Gamma(F)$ and all $d$-balls, and non-empty, since $\rmdist(\gamma_j(F),\Gamma(F)) < \veps_j$ and (b) $\Gamma_{j+1}$ is lower semi-continuous, according to (C) and the continuity of $\gamma_j$.
Note that $\sup \{ d(\gamma_j(F),\gamma_{j+1}(F)) : F \in \bF \} \leq \veps_j + \veps_{j+1}$ for all $j \in \N$.
Accordingly, $\la \gamma_j \ra_{j \in \N}$ is uniformly $d$-Cauchy.
As $\bE$ is $d$-complete, $\la \gamma_j \ra_{j \in \N}$ converges uniformly to some $\gamma : \bF \to \bE$.
Since each $\gamma_j$ is continuous, so is $\gamma$.
Finally, $\gamma(F) \in \Gamma(F)$, since $\new{\rmdist(\gamma(F),\Gamma(F)) = 0}$ and $\Gamma(F)$ is closed, $F \in \bF$.\cqfd
\end{Empty}



\new{\section*{Tool and computational resource disclosure}}

\new{No computer algebra system, proof assistant, or numerical computation was used in obtaining or verifying the results of this paper; in preparing the present revision a large-language-model assistant was used to proofread the manuscript and to help draft a small number of explanatory sentences, each output having been assessed and verified by the author, who takes sole responsibility for the correctness and completeness of the results.}


\bibliographystyle{plainurl}
\bibliography{../../../Bibliography/thdp}




\end{document}